\documentclass[11pt]{article}

\usepackage[utf8]{inputenc}
\usepackage[T1]{fontenc}
\usepackage[noBBpl]{mathpazo}
\usepackage{microtype}

\RequirePackage[a4paper,margin=2.5cm]{geometry}
\RequirePackage{parskip}

\usepackage[dvipsnames]{xcolor}
\definecolor{cblue}{RGB}{0,70,140}
\definecolor{cgreen}{RGB}{100,140,0}
\definecolor{cred}{RGB}{190,10,50}

\usepackage[labelsep = period, labelfont = bf, justification = justified, font=footnotesize]{caption}
\usepackage{float,graphicx,subcaption}
\DeclareCaptionSubType*[roman]{figure}

\usepackage{booktabs,makecell}
\usepackage{enumitem,mdwlist}
\setlist[itemize]{topsep=0ex,itemsep=0ex,parsep=0.4ex}
\setlist[enumerate]{topsep=0ex,itemsep=0ex,parsep=0.4ex}

\usepackage{pgf,tikz,ifthen,calc}
\usepackage{float, graphicx}
\usepackage{tkz-euclide}
\tkzSetUpPoint[fill=black, size = 3pt]
\tkzSetUpLine[color=black, line width=0.6pt]
\tkzSetUpCircle[color=black, line width=0.6pt]

\usepackage{amsmath}
\usepackage{amsthm}
\usepackage{amssymb}
\usepackage{xcolor}
\usepackage{graphicx}
\usepackage{float}
\usepackage{enumitem}
\usepackage{thmtools}

\DeclareFontFamily{U} {cmr}{}

\DeclareFontShape{U}{cmr}{m}{n}{
	<-6> cmr5
	<6-7> cmr6
	<7-8> cmr7
	<8-9> cmr8
	<9-10> cmr9
	<10-12> cmr10
	<12-> cmr12}{}

\DeclareSymbolFont{Xcmr} {U} {cmr}{m}{n}
\DeclareMathSymbol{\Delta}{\mathord}{Xcmr}{'001}
\DeclareMathSymbol{\Upsilon}{\mathord}{Xcmr}{'007}
\DeclareMathSymbol{\Omega}{\mathord}{Xcmr}{'012}

\usepackage[hidelinks,colorlinks,pagebackref]{hyperref} 
\hypersetup{citecolor=green!70!black,linkcolor=blue!70!black,urlcolor=blue!70!black}
\usepackage[capitalise,nameinlink,noabbrev,compress]{cleveref} 

\newcommand{\affiliation}{\footnote}

\makeatletter
\renewcommand*{\eqref}[1]{%
  \hyperref[{#1}]{\textup{\tagform@{\ref*{#1}}}}%
}
\makeatother

\newcommand{\define}[1]{\textcolor{Maroon}{\emph{#1}}}

\newcommand{\eps}{\varepsilon}
\newcommand{\ca}{\mathcal}

\renewcommand{\phi}{\varphi}

\newcommand{\bbold}[1]{\textnormal{\textbf{#1}}}
\let\emptyset\varnothing
\newcommand{\codeg}{\operatorname{codeg}}

\theoremstyle{plain}
\newtheorem{thm}{Theorem}[section]
\newtheorem{claim}{Claim}[thm]
\newtheorem{conjecture}[thm]{Conjecture}
\newtheorem{problem}[thm]{Problem}
\newtheorem{lem}[thm]{Lemma}
\newtheorem{proposition}[thm]{Proposition}
\newtheorem{fact}[thm]{Fact}
\newtheorem{corollary}[thm]{Corollary}
\theoremstyle{definition}
\newtheorem{defn}[thm]{Definition}
\newtheorem{rmk}[thm]{Remark}



\title{Monochromatic cycle partitions of $r$-edge-coloured\\ graphs with high minimum degree}
\author{Francesco Di Braccio\affiliation{Department of Mathematics, London School of Economics and Political Science, United Kingdom (\textsf{\href{mailto:f.di-braccio@lse.ac.uk}{f.di-braccio@lse.ac.uk}).}} \and Viresh Patel\affiliation{School of Mathematical Sciences, Queen Mary University of London, United Kingdom (\textsf{\href{mailto:viresh.patel@qmul.ac.uk}{viresh.patel@qmul.ac.uk}}).}}
\date{}

\begin{document}

\maketitle

\begin{abstract}
    A question posed independently by Letzter and Pokrovskiy asks: how many vertex-disjoint monochromatic cycles are needed to cover the vertex set of an $r$-edge-coloured graph, as a function of its minimum (uncoloured) degree? We resolve this problem up to a $(\log r)$-factor. Specifically, we prove that, for any $r \geq 2$ and $\delta \in (0,1/2)$, any $n$-vertex $r$-edge-coloured graph $G$ with $\delta(G) \geq (1- \delta)n$ can be covered with \[\ca O\left(r \log r \cdot \left\lceil \frac{r}{\log(1/\delta)} \right\rceil\right)\] vertex-disjoint monochromatic cycles. We construct graphs that show this is tight up to the $(\log r)$-factor for all values of $r$ and $\delta$, and along the way disprove a conjecture of Bal and DeBiasio about monochromatic tree covering.
\end{abstract}

\section{Introduction}

Monochromatic partitioning is a classical research area at the intersection of Ramsey theory and the study of spanning structures in graphs. In line with a common trend in Ramsey theory, it concerns \define{$r$-edge-coloured graphs}—that is, graphs $G$ equipped with an edge-colouring $\chi: E(G) \to [r]$—and their \define{monochromatic} subgraphs, whose edges all receive the same colour. Since generic $r$-edge-coloured graphs may not contain a monochromatic subgraph covering all vertices, one instead seeks a cover by a small collection of such subgraphs. Many problems additionally require these subgraphs to be vertex-disjoint, in which case we say they \define{partition} the graph.

Cycles are one of the most natural and intensively studied subgraph classes in this setting.\footnote{Throughout the paper, we consider an empty graph, a single vertex, and a single edge to be (degenerate) cycles.} This line of inquiry has produced several compelling and influential questions. Lehel conjectured that every $2$-edge-colouring of $K_n$ admits a partition into two monochromatic cycles of different colours. This was first proven asymptotically by \L{}uczak, R\"odl, and Szemer\'edi \cite{luczak} and then exactly by Bessy and Thomassé \cite{BESSY2010176}. For general $r \ge 2$, Erd\H{o}s, Gyárfás, and Pyber \cite{EPB} conjectured that every $r$-edge-coloured $K_n$ can be partitioned into $r$ monochromatic cycles, and proved an upper bound of $\mathcal{O}(r^2 \log r)$ on the number of cycles required. Pokrovskiy \cite{POKROVSKIY201470} later constructed an example showing that at least $r+1$ cycles are sometimes necessary, though the conjecture may still hold with $\ca O(r)$ cycles. Despite substantial work, the best current upper bound is $\mathcal{O}(r \log r)$ due to Gyárfás, Sárközy, Ruzsink\'o, and Szemerédi \cite{GYARFAS2006855}.

In this paper, we consider analogous problems where the host graph is not assumed to be complete. Specifically, we address the following question: given $\alpha \in (0,1)$, how many monochromatic cycles are needed to partition an $n$-vertex $r$-edge-coloured graph $G$ whose minimum (uncoloured) degree is at least $\alpha n$? Interest in this problem started with a conjecture of Schelp \cite{Schelp2012SomeRT} stating that every $2$-edge-coloured graph $G$ with minimum degree at least $\frac{3}{4}|V(G)|$ admits a partition into two cycles, which was later solved over a series of papers \cite{baloghbarat, debiasio, shohammono}. The general version of the problem was posed independently by Letzter \cite{shohammono}, Pokrovskiy \cite{pokrovskiysparse}, and later reiterated by Allen, Böttcher, Lang, Skokan, and Stein \cite{ALLEN2024103838}.

To state the problem precisely, let $cp(G)$ denote the size of the smallest partition into monochromatic cycles of an $r$-edge-coloured graph $G$, and let $\delta(G)$ be its minimum degree (viewed as an uncoloured graph). For $r \ge 2$ and $\delta \in (0,1)$, define
\[
cp_r(\delta) := \limsup_{n\to\infty} \max_{G \in \mathcal{G}(n,r,\delta)} cp(G),
\]
where $\mathcal{G}(n,r,\delta)$ is the family of $n$-vertex $r$-edge-coloured graphs with $\delta(G) \ge (1-\delta)n$. Thus $cp_r(\delta)$ measures the number of monochromatic cycles needed to partition any sufficiently large graph of minimum degree at least $(1-\delta)n$, provided this quantity is bounded in terms of $r$ and $\delta$ alone. For example, $cp_r(\delta)$ is unbounded for $\delta > 1/2$, since any colouring of $K_{\delta n,(1-\delta)n}$ requires at least $(2\delta - 1)n$ monochromatic cycles. On the other hand, a result of Kor\'andi, Lang, Letzter, and Pokrovskiy \cite{korandilangletzterpokrovksiy} shows that $cp_r(\delta) = \mathcal{O}(r^2)$ whenever $\delta < 1/2$ (and that it is also unbounded for $\delta=1/2$). This leads to the following formulation of the central question.

\begin{problem}\label{prob:main_prob}
Determine $cp_r(\delta)$ for each $r \ge 2$ and $\delta \in (0,1/2)$.
\end{problem}

Here, we typically allow $\delta$ to depend on $r$. This problem is by now relatively well understood for $r=2$. Note that $cp_r(\delta)$ is increasing in $\delta$ since $\delta' > \delta$ implies $\ca G(n, r, \delta) \subseteq \ca G(n, r, \delta')$. Resolving Schelp's conjecture, Letzter \cite{shohammono} proved that $cp_2(1/4)=2$, which is optimal as $cp_2(\delta) > 2$ for $\delta > 1/4$. Allen, Böttcher, Lang, Skokan, and Stein \cite{ALLEN2024103838} showed that $cp_2(\delta)= 3$ for $\delta < 1/3$. Pokrovskiy \cite{pokrovskiysparse} conjectured that $cp_2(1/3)=3$ and that $cp_2(\delta)=4$ for all $\delta < 1/2$, which would settle the problem completely.

For general $r\ge 3$, the picture is much less complete. The abovementioned result of Kor\'andi, Lang, Letzter, and Pokrovskiy shows that $cp_r(\delta) = \mathcal{O}(r^2)$ for $\delta$ just below $1/2$, and this is tight up to a constant factor. At the opposite end, a straightforward adaptation of the methods yielding the $\mathcal{O}(r\log r)$ bound for complete graphs \cite{GYARFAS2006855} shows that $cp_r(\delta) = \mathcal{O}(r\log r)$ already when $\delta = r^{-\Omega(r)}$. The main gap lies in the intermediate region: how does $cp_r(\delta)$ drop from $\Theta(r^2)$ to $\mathcal{O}(r \log r)$ (or possibly $\mathcal{O}(r)$) as $\delta$ decreases from $1/2$ toward $0$? Our main result provides a sharp description of this transition.

\begin{thm}\label{thm:main_theorem}
There exist constants $k, K > 0$ such that for all $r \ge 2$ and $\delta \in (0,1/2)$,
\[ k r \left\lceil \frac{r}{\log(1/\delta)} \right\rceil \;\le\; cp_r(\delta) \;\le\;
K r \log r \left\lceil \frac{r}{\log(1/\delta)} \right\rceil. \]
\end{thm}

The theorem resolves \cref{prob:main_prob} up to a $\log$-factor and may be viewed as a minimum degree generalization of Gy\'arf\'as, Ruszink\'o, S\'ark\"ozy, and Szemer\'edi's $\ca O(r\log r)$ bound. Note that $\lfloor r/\log(1/\delta)\rfloor$ decreases from $\Theta(r)$, when $\delta$ is a constant, to $1$ when $\delta \leq e^{-r}$. Thus, as expected, the theorem states that $cp_r(\delta)$ is between $\Omega(r^2)$ and $\ca O(r^2 \log r)$ for constant $\delta$ (in which regime, the upper bound from \cite{korandilangletzterpokrovksiy} performs better) and between $\Omega(r)$ and $\ca O(r \log r)$ when $\delta$ is close to $0$. 

As a crucial ingredient in the proof of \cref{thm:main_theorem}, we will also obtain results about the analogue of \cref{prob:main_prob} where instead of a cycle partition we seek a \define{tree cover}: a collection of (not necessarily vertex-disjoint) monochromatic trees covering all vertices. Given an $r$-edge-colored graph $G$, let $tc(G)$ be the size of its smallest tree cover, and define
\[tc_r(\delta) := \limsup_{n\to\infty} \max_{G \in \mathcal{G}(n,r,\delta)} tc(G).\]
The problem of determining $tc_r(\delta)$ represents a weakening of \cref{prob:main_prob} in two different ways (from cycles to trees, and from partitions to covers), and so in general we have $tc_r(\delta) \leq cp_r(\delta)$. It has also received attention in its own right. Settling affirmatively a conjecture of Bal and DeBiasio \cite{baldebiasio}, Buci\'c, Korandi, and Sudakov \cite{bucickorandisudakov} proved that every $n$-vertex $r$-edge-coloured graph $G$ with $\delta(G) \geq (1 - 2^{-r})n$ can be covered with $r$ trees. This shows that $tc_r(2^{-r}) = r$, which highlights the jump in difficulty from tree covering to cycle partitioning.\footnote{For instance, covering $r$-edge-coloured $K_n$ with $r$ monochromatic trees is trivial: consider the $r$ monochromatic stars centred at any given vertex.} At the opposite extreme, Kor\'andi, Lang, Letzter, and Pokrovskiy \cite{korandilangletzterpokrovksiy} showed that $tc_r(\delta)= \Theta(r^2)$ when $\delta < 1/2$ is a constant independent of $r$. Regarding the intermediate range, Bal and DeBiasio suggested the following conjecture. 

\begin{conjecture}[\cite{baldebiasio}]
    For all $r \geq 2$, if $G$ is an $n$-vertex $r$-edge-coloured graph with $\delta(G) \geq \frac{r(n-r+1) + 1}{r+1}$, then $tc(G) \leq r$.
\end{conjecture}

Note that this conjecture would imply $tc_r(1/(r+1)) = r$. Our second result disproves this conjecture and provides close-to-matching lower and upper bounds for $tc_r(\delta)$. 

\begin{thm}\label{thm:tree_cover_main}
    There exist constants $k, K > 0$ such that for all $r \geq 2$ and $\delta \in (0,1)$,
    \[kr \left\lceil \frac{r}{\log(1/\delta)} \right\rceil \;\leq\; tc_r(\delta) \;\leq\; Kr \left\lceil \frac{r}{\log(1/\delta)} \cdot \log \left( 1 + \frac{r}{\log(1/\delta)}\right) \right\rceil.\]
    If in addition $\delta \geq 1/2$, then
    \[tc_r(\delta) \leq \frac{Kr^2}{\log(1/\delta)}.\]
\end{thm}
Note that, differently from $cp_r(\delta)$, the quantity $tc_r(\delta)$ is bounded for any $\delta \in (0,1)$. Another difference with respect to \cref{thm:main_theorem} is that the upper and lower bound are separated by a factor inside the ceiling of $\log(1 + r/\log(1/\delta))$, which is better than $\log r$ for most choices of $\delta$.

\subsection{Proof overview and organization}\label{sec:overview}

We now briefly outline the proofs of our main theorems. We begin by establishing the upper and lower bounds for the monochromatic tree covering problem, \cref{thm:tree_cover_main}, which forms the backbone of several subsequent arguments. We show that determining $tc_r(\delta)$ is equivalent to determining the transversal number of a certain class of $r$-partite multi-$r$-graphs, which we call \define{$\delta$-intersecting} (see \cref{defn:deltainter}). This equivalence allows us to derive the lower bound on $ tc_r(\delta)$ from constructions of $\delta$-intersecting hypergraphs with large transversal number, while the upper bounds follow from bounds on how large this transversal number can be. Several of these arguments rely on results of Buci\'c, Kor\'andi, and Sudakov \cite{bucickorandisudakov}, who studied a closely related problem.

Our lower bound on $tc_r(\delta)$ immediately implies the same lower bound for $cp_r(\delta)$. The remainder (and indeed the majority) of the paper is therefore devoted to proving the upper bound on $cp_r(\delta)$ stated in \cref{thm:main_theorem}. Our approach combines the connected matching method, a now standard technique for constructing long cycles in dense graphs via the Szemer\'edi regularity lemma that goes back to {\L}uczak \cite{LUCZAK1999174}, with the absorption method.

The high-level strategy, similar to that used in \cite{korandilangletzterpokrovksiy, GYARFAS2006855}, proceeds as follows. We start by applying the multi-colour version of the Szemer\'edi regularity lemma to $G$, obtaining a regular partition $\{V_0, V_1, \dots, V_t\}$ of $V(G)$ and the associated reduced graph $ R $, which is itself $r$-edge-coloured. We then carefully select a small collection of monochromatic components of $R$ and show that their union, viewed as an uncoloured graph, contains a spanning subgraph $H$ of $V(R)$ with special properties. In particular, $H$ is the vertex-disjoint union of graphs $F$ and $M$ where
\begin{itemize}
    \item $F$ covers a significant proportion of $V(R)$, has bounded maximum degree, and satisfies a `robust matchability' property (as defined in \cite{korandilangletzterpokrovksiy}): after blowing up its vertices to large sets, any of its induced subgraphs obtained by deleting a small number of vertices contains a perfect matching; and
    \item $ M $ is a perfect matching of $G - V(F)$.
\end{itemize}

Using the regularity of the partition, together with the bounded maximum degree of $H$, we can remove a small number of vertices from each cluster of $H = F \cup M$ to guarantee that any of its edges (of colour $c$, say) corresponds to a regular pair of clusters in $G$ with positive minimum degree in colour $c$.

We combine the discarded vertices from each cluster into a (small) leftover set $L$. We further remove, and add to $L$, a few vertices from each cluster in $M$ so that these clusters all have equal size. Let $W$ denote the union of all clusters covered by $F$. We then show (in the \emph{absorption step}) that all vertices in $L$ can be incorporated into a small number of monochromatic cycles contained in the bipartite edge-coloured graph $G[L, W]$. This is possible because $F$ covers a large portion of $V(R)$, ensuring that $W$ is large and that every vertex of $L$ retains most of its degree in $G[L, W]$.

After handling the vertices in $L$, the remainder of the graph exhibits strong regularity properties, enabling us to cover it using blow-up lemma-style arguments. Since all previously constructed cycles intersect each cluster of $F$ in a small number of vertices, the robust matchability of $F$ allows us to construct a small family of cycles covering its remaining vertices. The clusters in $M$ can be covered with a similar argument, this time exploiting the fact that they contain equal numbers of uncovered vertices. The number of cycles required to complete this step roughly matches the number of monochromatic components used to build $F$ and $M$, which is in turn closely tied to $tc_r(\delta)$ and depends crucially on our bounds from \cref{thm:tree_cover_main}.

Although this overall strategy is fairly standard, implementing each step presents substantial difficulties and requires several new ideas. The most demanding part is the absorption step, which is handled via a separate lemma asserting that any edge-coloured bipartite graph $G[A,B]$, where $A$ is small relative to $B$ and every $v \in A$ satisfies $\deg(v,B) \ge (1-\delta)|B| $, contains a small number of vertex-disjoint monochromatic cycles covering $A$. The proof relies on constructing large families of structures we call \define{connecting hubs} (see \cref{defn:connecting_hub}) and applying our tree covering results (i.e. \cref{thm:tree_cover_main}) to suitable auxiliary graphs based on these. We defer a detailed discussion to \cref{sec:absorption}---see the beginning of that section for a thorough proof sketch.

The other main challenge is the construction of $H = F \cup M$. For $F$, we take a vertex-disjoint union of triangles and graphs we call \define{barbells} (that is, \(7\)-vertex graphs formed from two triangles by adding a new vertex adjacent to one vertex in each triangle). Both triangles and barbells readily satisfy the robust matchability property required of $F$. To ensure that $F$ and $M$ are vertex-disjoint and together span $V(R)$, we proceed as follows:
\begin{itemize}
    \item find a large collection of vertex-disjoint triangles in \( R \);
    \item cover most of the remaining vertices with a matching \( M \); and
    \item use a small number of triangles from the first step to construct barbells covering all remaining vertices of \( R \) (finally, $F$ is formed from the union of all triangles and barbells thus found).
\end{itemize}

Throughout, we ensure that the edges added to $F\cup M$ lie in only a few monochromatic components. In particular, roughly $\ca O(tc_r(\delta))$ components are required for both the first and third step, whereas the second step uses a simple greedy argument requiring at most $\ca O(r\log r)$ components.

The paper is organized as follows. In \cref{sec:prelims} we introduce notation and tools, including the regularity lemma and various results about \define{sublinear expanders}, which are used in the absorption step. \cref{sec:tree_covering} contains the proof of \cref{thm:tree_cover_main}. The absorption lemma is proved in \cref{sec:absorption}. In \cref{sec:reduced_graph} we establish several lemmas for constructing the graphs \(F\) and \(M\) and verifying their required properties. Finally, \cref{sec:main_theorem} combines all previous ingredients to prove \cref{thm:main_theorem}, and \cref{sec:concluding_remarks} discusses some open problems.


\section{Preliminaries}\label{sec:prelims}

\subsection{Notation}\label{sec:notation}

An \define{$r$-edge-coloured graph} $G$ is a triple $(V(G), E(G), \chi(G))$ where $(V(G), E(G))$ is a simple graph and $\chi(G) : E(G) \to [r]$ is a colouring of its edges by elements of $[r]$. Given an $r$-edge-coloured graph $G$ and some colour $c \in [r]$, we write \define{$G_c$} to denote the graph obtained by considering only the edges of colour $c$ (which is monochromatic and thus may be viewed as an uncoloured graph). Given $C \subseteq [r]$, we write \define{$G_C$} to denote the $r$-edge-coloured graph $\bigcup_{c \in C} G_c$.

Given a graph $G$, a vertex $v \in V(G)$ and a set $S \subseteq V(G)$, \define{$N_G(v, S)$} is the set of neighbours of $v$ in $S$; if $G$ is an $r$-edge-coloured graph itself, this refers to the set of neighbours regardless of the colour of the edge connecting them to $G$.
So, for instance, if we wish to refer to the set of vertices in $S$ connected to $v$ by an edge of colour $c \in [r]$ or some colour belonging to $C \subseteq [r]$, we can just write $N_{G_c}(v, S)$ and $N_{G_C}(v, S)$ (respectively). Degrees and codegrees are now defined in the obvious way: \define{$\deg_G(v, S)$} is $|N_G(v, S)|$ and \define{$\codeg_G(v, u, S)$} is $|N_G(v, S) \cap N_G(u, S)|$. Whenever the host graph $G$ is clear from context, we will suppress the use of the letter $G$ and simply write, for instance, \define{$\deg(v, S)$}, \define{$\deg_c(v, S)$}, \define{$\deg_C(v, S)$} in place of $\deg_G(v, S), \deg_{G_c}(v, S), \deg_{G_C}(v, S)$. When $S = V(G)$, we simply write \define{$\deg(v)$}, \define{$\deg_c(v)$}, \define{$\deg_C(v)$}. Finally, the minimum degree \define{$\delta(G)$} of an $r$-edge-coloured graph $G$ is just $\min_{v \in V(G)}\deg(v)$. 

Given an $r$-edge-coloured graph $G$ and some $i \in [r]$, the \define{ $i$-coloured components } of $G$ are simply the components of $G_i$. Throughout the paper, we view the components themselves as subgraphs of $G$ rather than subsets of $V(G)$. This allows us to talk about, say, certain edges being contained in the union of some components. A \define{monochromatic component} is any component of $G_i$ for some $i \in [r]$. A \define{monochromatic connected matching} is a matching that is also a subgraph of a monochromatic component.

Given a graph $G$ and disjoint subsets $A, B \subseteq V(G)$, we write \define{$G[A,B]$} for the subgraph of $G$ whose edges are in $A \times B$. We will also often use $G[A,B]$ to directly refer to a bipartite (possibly edge-coloured) graph on vertex classes $A$ and $B$. We write \define{$G[A]$} for the induced subgraph of $G$ on $A$, and \define{$G - A$} for $G[V(G) \setminus A]$. Given two graphs $G$ and $H$, \define{$G \cap H$} is the subgraph on edge set $E(G) \cap E(H)$ and \define{$G \setminus H$} the one on edge set $E(G) \setminus E(H)$ (this definition extends in the obvious way to the case when the graphs are edge-coloured and agree on the colouring of their shared edges). We write \define{$e(G)$} for the number of edges in $G$, and \define{$d(G)$} is its average degree $2e(G)/|V(G)|$. 

We will often make assumptions of the form $a \ll b_1, \dots, b_k$. This means that there exists a positive function $f$ for which the relevant result holds provided that $a \leq f(b_1, \dots, b_k)$. Moreover, when writing this expression, we always assume that $a, b_1, \dots, b_k$ are positive. If we write $1/n$ instead of $a$ or $b_i$, it is implicitly assumed that $n$ is an integer. We omit ceilings and floors when they are not crucial to the argument. All logarithms throughout the paper are base $e$.

\subsection{Probability}

We will need the following well-known concentration inequality for the hypergeometric distribution. Recall that a hypergeometric random variable $X$ with parameters $N,n$ and $m$ takes value $k$ with probability $\binom{m}{k}\binom{N-m}{n-k}/\binom{N}{n}$. Also, it satisfies $\mathbb{E}[X] =\frac{nm}{N}$.

\begin{lem}[Chernoff's inequality for the hypergeometric distribution]\label{lem:chernoff}
Let $X$ be a hy\-per\-ge\-o\-met\-ric random variable with parameters $N,n$ and $m$. Then, for any $t>0$,
\[\mathbb{P}\left(|X-\mathbb{E}[ X]|\ge t\right)\le 2e^{-t^2/(3\mathbb{E}[X])}.\]
\end{lem}

\subsection{Regularity}

In the following, we state the variant of the Szemer\'edi regularity lemma \cite{regularpartitions} we will be using, together with several tools associated with it. All the results from this section are widely known and can be found in the survey \cite{komlossimonovits}. 

Given a graph $G$ and disjoint sets of vertices $A, B \subseteq V(G)$, the \define{density} of the pair $(A, B)$ is defined as $d_G(A,B) := e(A, B)/ (|A||B|)$. Given $d, \eps> 0$, the pair $(A, B)$ is said to be \define{$(d, \eps)$-regular} if the following holds: for any subsets $A' \subseteq A, B' \subseteq B$ with $|A'| \geq \eps |A|$ and $|B'| \geq \eps |B|$, 
\[|d(A', B') - d| \leq \eps . \]
We will simply say that $(A, B)$ is $\eps$-regular if the specific value of $d$ is either unnecessary to the argument or clear from context. 

We will be using the following degree form of the multicolour regularity lemma. This follows from the multicolour regularity lemma (Theorem 1.18 in \cite{komlossimonovits}) by the same argument that proves the degree form of the (uncoloured) regularity lemma (see Theorem 1.10 in \cite{komlossimonovits}).

\begin{lem}[Degree form of the multicolour regularity lemma]\label{lem:regularitylemma}
Let $1/n \ll 1/M \ll \eps, 1/r$. Let $G$ be an $n$-vertex $r$-edge-coloured graph, and let $d > 0$. Then there is a partition $\{V_0, V_1, \dots, V_t\}$ of $V(G)$ and a subgraph $G'$ of $G$ with vertex set $V(G) \setminus V_0$ such that
\begin{enumerate}[label = \textnormal{(R\arabic*)}]
    \item\label{prop:reg1} $1/\eps \leq t \leq M$,
    \item\label{prop:reg2} $|V_0| \leq \eps n$ and $|V_1| = \dots = |V_t| \leq \eps n$, 
    \item\label{prop:reg3} $\deg_{G'}(v) \geq \deg_G(v) - (rd + \eps)n$ for each $v \in V(G) \setminus V_0$, 
    \item\label{prop:reg4} $G'[V_i]$ contains no edges for $i \in [t]$, and
    \item\label{prop:reg5} for each colour $c \in [r]$, all pairs $(V_i, V_j)$ are $\eps$-regular in $G'_c$ and have density either $0$ or at least $d$. 
\end{enumerate}
\end{lem}

If $G$ is a graph with the partition $\{V_0, V_1, \dots, V_t\}$ and the subgraph $G'$ given by \cref{lem:regularitylemma}, the \define{$(\eps, \delta)$-reduced graph} of $G$, denoted $R(G)$, is the $r$-edge-coloured multigraph on vertex set $\{V_1, \dots, V_t\}$ in which $V_i$ and $V_j$ are connected by a $c$-coloured edge if $(V_i, V_j)$ is $\eps$-regular and of density at least $d$ in $G'_c$. The following fact is an immediate consequence of \ref{prop:reg2}--\ref{prop:reg5}. 

\begin{fact}\label{fact:reduceddeg}
Let $1/n \ll 1/M \ll \eps, 1/r$ and let $d > 0$. Let $G$ be an $n$-vertex $r$-edge-coloured graph and let $G'$ be the subgraph and $\{V_0, V_1, \dots, V_t\}$ be the partition given by \cref{lem:regularitylemma}. Let $R$ be an $r$-edge-coloured graph obtained from $R(G)$ by retaining exactly one edge between each pair of adjacent vertices. Then, 
\[\delta(R) \geq \left(\frac{\delta(G)}{n} - rd - \eps \right)t.\]
\end{fact}

Next we state some basic facts about regular pairs.

\begin{fact}[Facts 1.3 and 1.5 in \cite{komlossimonovits}]\label{lem:basicfacts} Let $d, \eps > 0$ and $\eps < 1/4$. Let $(A,B)$ be a $(d, \eps)$-regular pair.
    \begin{enumerate}[label = (\arabic*)]
        \item If $A' \subseteq A, B' \subseteq B$ satisfy $|A'| \geq \sqrt{\eps} |A|$ and $|B'| \geq \sqrt{\eps}|B|$, then $(A', B')$ is a $(d, \sqrt{\eps})$-regular pair.
        \item All but at most $\eps |A|$ vertices $v \in A$ satisfy $\deg(v, B) \geq (d- \eps)|B|$. 
    \end{enumerate}
\end{fact}

The next lemma shows how to connect vertices in the host graph using the connectivity of the reduced graph. Its proof is the same as Lemma 3.4 in \cite{korandilangletzterpokrovksiy}. 

\begin{lem}\label{lem:reg_connecting_lemma}
Let $1/m \ll \eps \ll d$ and $t \geq 2$. Let $G$ be a $t$-partite $n$-vertex graph on vertex classes $V_1, \dots, V_t$ with $|V_i| = m$ for each $i \in [t]$. Suppose that each pair $(V_i, V_{i+1})$ is $\eps$-regular with density at least $d$. Then, for any $S \subseteq V(G)$ with $|S \cap V_i| \leq dm/4$ for each $i \in [t]$ and $x, y \in V(G) $ with $\deg(x, V_1), \deg(y, V_t) \geq dm/2$, there is a path $x v_1 \dots v_t y$ in $G$ such that $v_i \in V_i \setminus S$ for each $i \in [t]$.
\end{lem}

The next lemma is a corollary of the celebrated blow-up lemma \cite{blowuplemma}; alternatively, see the proof of Lemma 3.1 in \cite{korandilangletzterpokrovksiy}.

\begin{lem}\label{lem:spanningpath}
    Let $1/m \ll \eps \ll d$. Let $(A, B)$ be a $(d, \eps)$-regular pair in a graph $G$ with $|A| = |B| = m$. Suppose that $\delta(G[A,B]) \geq 5\eps m$. Then, for any $x \in A$ and $y \in B$, the graph $G[A, B]$ contains a $x$--$y$ path spanning $A \cup B$.
\end{lem}

\subsection{Sublinear expanders}\label{subsec:sublinear_expanders}

The proof of our absorption lemma makes use of certain sparse expanding graphs known as \define{sublinear expanders}. These graphs were introduced by Koml\'os and Szemer\'edi \cite{ks1, ks2} and have found widespread use in recent years---see \cite{shohamsublinear} for a survey. The specific variant and related results that we will be using are also due to \cite{ks1, ks2} but specifically stated in the form below in \cite{shohamsublinear}. 

\begin{defn}[sublinear expander]\label{defn:sublinear_expander}
    For $\eps, t > 0$, let $\rho(x) =\rho(x, \eps, t)$ be the function defined by, for $x \geq t/2$, 
    \[\rho(x) := \frac{\eps}{\log^2(15x/ t)}.\]
    An \define{$(\eps, t)$-expander} is a graph $G$ in which every set of vertices $U$ with $t/2 \leq |U| \leq |G|/2$ satisfies
    \[|N_G(U) \setminus U| \geq \rho(|U|)\cdot  |U|. \]
\end{defn}

The next theorem lets us find a sublinear expander in an arbitrary graph at the cost of a small loss in its average degree, whereas the one after says that such expanders have good connectivity properties. 

\begin{thm}[\cite{ks1, ks2}]\label{thm:sublinearexpander}
Let $\eps > 0$ be sufficiently small and let $t >0$. Then every graph $G$ has a subgraph $H$ which is an $(\eps, t)$-expander, and satisfies $d(H) \geq d(G)/2$ and $\delta(H) \geq d(H)/2$.
\end{thm}

\begin{thm}[\cite{ks1,ks2}]\label{thm:conn_in_expanders}
Let $G$ be an $n$-vertex $(\eps, t)$-expander. Then, for every $x \geq t/2$ and every three sets of vertices $U_1, U_2$ and $W$, where $|U_1|, |U_2| \geq x$ and $|W| \leq \rho(x)x/4$, there is a path from $U_1$ to $U_2$ in $G-W$ of length at most $\frac{2}{\eps}\log^3(15n/t)$.      
\end{thm}

\section{Tree covering}\label{sec:tree_covering}

In this section, we prove \cref{thm:tree_cover_main}, which provides close-to-matching upper and lower bounds on $tc_r(\delta)$ for each choice of $r \geq2$ and $\delta \in (0,1)$. These bounds will serve as key tools in our cycle partitioning results later on.

For the proof of \cref{thm:tree_cover_main}, in \cref{sec:hypergraph_connection} we will show that determining $tc_{r}(\delta)$ is essentially equivalent to determining the transversal number of certain multi-hypergraphs we call \define{$\delta$-intersecting} (see \cref{defn:deltainter} below). With this equivalence in hand, we will be able to derive our lower bound (in \cref{sec:lower_bound_tree}) and general upper bound (in \cref{sec:upper_bound_tree}) mostly from results of \cite{bucickorandisudakov} (our upper bound for $\delta \geq 1/2$ uses a different, though straightforward, argument).

\subsection{Translating to a hypergraph transversal problem}\label{sec:hypergraph_connection}

Given a multi-hypergraph $\ca H$, a \define{transversal} for $\ca H$ is a set $S \subseteq V(\ca H)$ intersecting every edge of $\ca H$. The \define{transversal number} of $\ca H$, denoted $\tau(\ca H)$, is the size of the smallest transversal for $\ca H$. It is well-known that questions about monochromatic tree covering can often be reduced to hypergraph transversal problems (for instance, see \cite{bucickorandisudakov, EPB, gyarfaspartite, gyarfassurvey}). This is generally achieved by constructing a multi-hypergraph which encodes the monochromatic component structure of the graph, as in the following definition. 

\begin{defn}[connectivity hypergraph]
Let $G$ be an $r$-edge-coloured graph and let $T_1, \dots, T_\ell$ be its monochromatic components. We define its \define{connectivity hypergraph} $\ca C(G)$ as the multi-$r$-graph on vertex set $\{T_1, \dots, T_\ell\}$ and having, for each vertex $v \in V(G)$, an edge $e_v = \{T_{i_1}, \dots, T_{i_r}\}$ where each $T_{i_j}$ is the (unique) $j$-coloured component containing $v$.  
\end{defn}

Note that, under our notion of monochromatic component (see \cref{sec:notation}), each vertex of $G$ is contained in some monochromatic component in each colour in $[r]$, and so indeed $\ca C(G)$ is $r$-uniform. Further observe that, letting $\ca T_i \subseteq \{T_1, \dots, T_\ell\}$ be the set of $i$-coloured components for each $i \in [r]$, the hypergraph $\ca C(G)$ is $r$-partite with partition $\{\ca T_1, \dots, \ca T_r\}$. Also, we have the following fact. 

\begin{fact}\label{fact:transversal} Let $r \geq 2$. For each $r$-edge-coloured graph $G$, we have $tc(G) = \tau(\ca C(G))$. 
\end{fact}

We will be particularly focusing on hypergraphs which fall under the following definition.

\begin{defn}[$\delta$-intersecting]\label{defn:deltainter} Let $\delta \in [0,1]$. A multi-hypergraph $\ca H$ is said to be \define{$\delta$-intersecting} if every edge $e \in E(\ca H)$ intersects at least $(1 - \delta) \cdot e(\ca H)$ edges of $\ca H$.  
\end{defn}

Next, we relate this definition to edge-coloured graphs with high minimum degree.

\begin{fact}\label{fact:intersectingaux}
Let $r \geq 2$ and $\delta \in (0,1]$. If $G$ is an $r$-edge-coloured graph with $\delta(G) \geq (1- \delta)|V(G)| -1$, then $\ca C(G)$ is a $\delta$-intersecting.     
\end{fact}
\begin{proof}
Any given edge $e_v \in E(\ca C(G))$ intersects $e_{v'}$ for each $v' \in N_G(v)$ since $v$ and $v'$ necessarily belong to the same monochromatic component in the colour of $vv'$. Also including itself, $e_v$ intersects at least $(1 - \delta)|V(G)| -1 + 1= (1-\delta)e(\ca C(G))$ edges of $\ca C(G)$. 
\end{proof}

The next proposition provides a partial converse to the previous observation. 

\begin{proposition}\label{proposition:frommultitocolour}
    Let $r \geq2$ and $\delta \in (0,1]$. Let $\ca H$ be a $\delta$-intersecting $r$-partite multi-$r$-graph. Then there exists an $r$-edge-coloured graph $G$ with $\delta(G) \geq (1-\delta)|V(G)|-1$ and $tc(G) = \tau(\ca H)$.
\end{proposition}

\begin{proof}
We begin by defining $\ca H'$ as the $r$-partite multi-$r$-graph obtained from $\ca H$ by substituting each $e_v \in E(\ca H)$ with $4r$ identical copies $e_v^1, \dots, e^{4r}_v$ copies of $e_v$. $\ca H$ and $\ca H'$ have the same underlying $r$-graph and so $\tau(\ca H') = \tau(\ca H)$. It is also easy to see that each $e_v^i \in E(\ca H')$ intersects at least
\[4r \cdot  (1-\delta)e(\ca H) = (1-\delta) e(\ca H')\]
edges of $\ca H'$, thus implying that $\ca H'$ is also $\delta$-intersecting. We will further assume that $\ca H'$ contains no isolated vertices, as otherwise removing them does not affect any of these properties. Crucially, this guarantees that each vertex in $\ca H'$ is contained in at least $4r$ edges. 

To prove the proposition, it will be sufficient by \cref{fact:transversal} to construct an $r$-edge-coloured graph $G$ with minimum degree at least $(1- \delta)|V(G)|$ satisfying $\ca C(G) = \ca H'$. To accomplish this, we begin by constructing a (possibly) sparse $r$-edge-coloured graph $G'$ also satisfying $\ca C(G') = \ca H'$ and then boost its minimum degree by filling in many of its missing edges.

Recall that $\ca H'$ is $r$-partite and let $V_1 \cup \dots \cup V_r$ be a partition for it. Let $V_i = \{T^i_1, \dots, T^i_{\ell_i}\}$ for each $i \in [r]$. Now we construct $G'$ by starting with an empty graph $G^0$ on vertex set $W := E(\ca H')$ and progressively adding spanning monochromatic linear forests in each colour in $[r]$.

We now describe an iteration of this procedure. Let us suppose that for some $i \in [r]$ we have previously obtained an $(i-1)$-edge-coloured graph $G^{i-1}$ on $W$ with $\Delta(G^{i-1}) \leq 2(i-1)$; we will argue that this property is maintained at the end of the next iteration replacing $i-1$ with $i$. For each $j \in [\ell_i]$, let $E^i_j \subseteq W$ be the set of edges of $\ca H'$ containing $T_j^i$, and observe that $\{E_j^i: j \in [\ell_i]\}$ forms a partition of $W$, as $\ca H'$ is $r$-partite $r$-uniform with $V_i$ being one of its partition classes. Also, each $T_j^i$ is contained in at least $4r$ edges in $\ca H'$ and so $|E_j^i| \geq 4r$. Then, for each $j \in [\ell_i]$, we have that the complement of $G^{i-1}[E_j^i]$ has minimum degree at least
\[|E_j^i| - \Delta(G^{i-1}) \geq |E_j^i| - 2(i-1) \geq |E_j^i| - 2r \geq \frac{|E_j^i|}{2},\]
and thus by Dirac's theorem the complement of $G^{i-1}$ contains a path $P^i_j$ whose vertex set is precisely $E_j^i$. The union $F_i := \bigcup_{j \in [\ell_i]} P^i_j$ is a linear forest spanning $W$. Now we let $G^i$ be obtained from $G^{i-1}$ by adding the edges of $F_i$ and colouring them in colour $i$. We have $\Delta(G^i) \leq 2i$, as needed for the procedure to continue. 

When the procedure terminates, we obtain a graph $G' := G^r$ in which any two vertices belong to the same $i$-coloured component for any given $i \in [r]$ if and only if the corresponding edges in $E(\ca H')$ intersect in $V_i$. It is easy to see that this implies $\ca C(G') = \ca H'$. To finish the construction of $G$, we take every intersecting pair of distinct edges $x, y \in E(\ca H')$ which are not adjacent as vertices of $G'$, and we simply connect them by an edge in a colour $i \in [r]$ such that $V_i \cap x \cap y \neq \emptyset$. Note that this implies that $T_j^i \in x \cap y$ for some $j \in [\ell_i]$ and thus $x, y \in E_j^i$. Thus, $x$ and $y$ are already connected in $G'$ by the $i$-coloured path spanning $E_j^i$ and so adding the $i$-coloured edge $xy$ does not affect the monochromatic component structure. Once this has been done for each choice of $x$ and $y$, we obtain a graph $G$ with $\ca C(G) = \ca H'$ and
\[\delta(G) \geq (1-\delta) e(\ca H') - 1\geq (1-\delta)|V(G)|-1,\]
since each $x$ intersects itself in $\ca H'$ but cannot be adjacent to itself in $G$. 
\end{proof}

\subsection{Lower bound}\label{sec:lower_bound_tree}

Our lower bound construction for \cref{thm:tree_cover_main} is obtained by constructing a $\delta$-intersecting hypergraph with high transversal number and applying \cref{proposition:frommultitocolour} to it. We now describe the hypergraph construction we will be using, which was introduced in \cite{bucickorandisudakov} for a different, though related, problem.

Given $m \geq 1$ and $t \geq r \geq 2$, we define $H_{r,t,m}$ to be the following $r$-partite $r$-graph. In each of the $r$ parts of its partition, there is a set $U_i$ with $|U_i| = m$ of so-called \define{important} vertices. For every choice of $r-t$ important vertices $u_1, \dots, u_{r-t}$ each belonging to a different $U_i$, we add an edge containing those important vertices together with $t$ new vertices unique to this edge. Each of these new vertices is placed into a distinct part not already covered by the $u_i$. 

\begin{fact}[Proposition 6.5 in {\cite{bucickorandisudakov}}]\label{fact:coverHrtm}
    $\tau(H_{r,t,m}) = (t+1)m$. 
\end{fact}

We prove that, with the right choice of parameters, this hypergraph is $\delta$-intersecting. 

\begin{lem}\label{lem:deltaintersect} Let $r \geq 3$ and $\delta \in [e^{-r/5}, 1)$. Let $t := \lfloor r/3\rfloor$ and $m:= \left\lfloor \frac{r}{5\log(1/\delta)} \right\rfloor$. Then $H_{r, t, m}$ is $\delta$-intersecting. 
\end{lem}

\begin{proof} Note that the lower bound on $\delta$ ensures that $m \geq 1$ and so $H_{r,t,m}$ is a non-empty graph. Let $e \in E(H_{r,t,m})$ be an arbitrary edge and let $u_1, \dots, u_{r-t}$ be its important vertices. By relabelling the parts if necessary, we may assume that $u_i \in U_i$ for each $i \in [r-t]$. 

Now let us compute the probability that a uniformly chosen random edge $e' \in E(H_{r,t,m})$ intersects $e$. Note that the uniform distribution on $E(H_{r,t,m})$ can be simulated by first picking $r-t$ important classes $U_{i_1}, \dots, U_{i_{r-t}}$ uniformly at random, and then choosing vertices $u_{i_1}' \in U_{i_1}, \dots, u'_{i_{r-t}} \in U_{i_{r-t}}$ uniformly and independently of each other. Then $e'$ is picked to be the unique edge containing $u_{i_1}', \dots, u'_{i_{r-t}}$. 

Observe that if we condition on any choice of $U_{i_1}, \dots, U_{i_{r-t}}$, then we have \[|\{1, \dots, r-t\} \cap \{i_1, \dots, i_{r-t}\}| \geq r - 2t \geq \frac{r}{3}.\]
So, let $j_1, \dots, j_{\lfloor r/3 \rfloor}$ belong to this intersection. Then, 
\[\mathbb{P}\left( u'_{j_i} \neq u_{j_i} \mbox{ for each } 1 \leq i \leq \lfloor r/3\rfloor \right) = \left( 1- \frac{1}{m} \right)^{\lfloor r/3 \rfloor} \leq \exp\left\{-\frac{r}{5m}\right\} \leq \delta,\]
where in the first inequality we used $\lfloor r/3\rfloor \geq r/5$ since $r \geq 3$. Thus, conditioned on any choice of $U_{i_1}, \dots, U_{i_{r-t}}$, the probability that $e'$ is disjoint from $e$ is at most $\delta$. But then this remains true without conditioning on this choice. This implies that $e$ intersects at least \((1 - \delta) \cdot e(H_{r,t,m})\)
distinct edges of $H_{r,t,m}$. This completes the proof since our choice of $e$ was arbitrary.
\end{proof}

Now we are ready to prove our lower bound.

\begin{lem}\label{lem:lower_bound_tree}
    Let $r \geq 2$ and $\delta \in (0,1)$. There are arbitrarily large $r$-edge-coloured graphs $G$ with $\delta(G) \geq ( 1- \delta)|V(G)|$ and 
    \[tc(G) \geq \frac{r}{120} \left\lceil \frac{r}{\log(1/\delta)} \right\rceil.\]
\end{lem}

\begin{proof}
First note that if $\delta \leq e^{-r/60}$ we have $\frac{r}{60 \log(1/\delta)} \leq 1$, in which case the statement is trivially true by the following simple construction. We take an $n$-vertex graph $G$ with $1/n \ll 1/r, \delta$ with $r$ distinguished vertices $v_1, \dots, v_r$ forming an independent set. We make each $v_i$ complete to $V(G) \setminus \{v_1, \dots, v_r\}$ in colour $i$, whereas $V(G) \setminus\{v_1, \dots, v_r\}$ is complete and coloured arbitrarily. We have $tc(G) = r$ and $\delta(G) \geq n - r \geq (1-\delta)n$. 

From now on, we will assume that $\delta \geq e^{-r/60}$. Our next step is to construct a $\delta^2$-intersecting $r$-partite $r$-graph with no cover of size $\frac{r^2}{60 \log(1/\delta)} \geq \frac{r}{120} \left\lceil \frac{r}{\log(1/\delta)} \right\rceil$. Let us start by showing that this can be done when $r = 2$, as this is handled differently from the case $r \geq 3$. Here we need to construct a bipartite ($2$-)graph, so let us take a matching on $m' := \lfloor (1 - \delta^2)^{-1}\rfloor$ edges. Since each edge intersects itself, the proportion of edges that each edge intersects is \(\frac{1}{m'} \geq 1- \delta^2, \) and so it is $\delta^2$-intersecting. Its transversal number is
\[m' = \left\lfloor \frac{1}{1 - \delta^2}\right\rfloor \geq \frac{1}{2(1-\delta^2)}.\]
However, using the fact that $e^{x/2} \leq 1 + x$ for all $x \in [0,1]$, we obtain 
\[\frac{1- \delta^2}{2} \leq \log(1 + 1- \delta^2) \leq \log \left(1 + \frac{1-\delta^2}{\delta^2} \right) = 2\log(1/\delta),\]
where the third inequality used $\delta \leq 1$. Thus, $m' \geq \frac{1}{8\log(1/\delta)} \geq \frac{r^2}{60\log(1/\delta)}$. This handles the case $r = 2$.

If $r \geq3$, we apply \cref{lem:deltaintersect} with $\delta^2$ in place of $\delta$ and $t= t(r)$ and $m = m(r,\delta^2)$ defined as in its statement. Note that $\delta \geq e^{-r/60} \geq e^{-r/5}$ and so the lemma applies. Then $H_{r,t,m}$ is $\delta^2$-intersecting. By \cref{fact:coverHrtm},
\[\tau(H_{r,t,m}) = (t+1) m \geq r/3 \cdot \left\lfloor \frac{r}{10\log(1/\delta)} \right\rfloor\geq \frac{r^2}{60\log(1/\delta)},\]
where in the last inequality we used $r/(10 \log(1/\delta)) \geq 1$ as $\delta \geq e^{-r/60}$. 

The previous paragraphs yield a $\delta^2$-intersecting $r$-partite $r$-graph $\ca H$ with $\tau(\ca H) \geq \frac{r^2}{60 \log(1/\delta)}$ for all choices of $r \geq 2$. By \cref{proposition:frommultitocolour}, this can be converted into a graph $G'$ with $\delta(G') \geq (1 - \delta^2)|V(G')| - 1$ and $tc(G') = \tau(\ca H)$.  Given any $n \in \mathbb{N}$ with $1/n \ll 1/\delta$, blow up each vertex of $G'$ to a set of size $n$. Rather than letting these $n$-sized sets be independent sets, we make them all complete and colour them arbitrarily, thereby obtaining an $r$-edge-coloured graph $G$ on $n \cdot |G'|$ vertices with
\[\delta(G) \geq n \cdot \left( \left(1 - \delta^2\right)|V(G')| - 1 \right) + (n-1) = \left( 1 - \delta^2 \right)|V(G)| - 1 \geq (1- \delta)|V(G)|.\]
It is easy to see that $tc(G) = tc(G') = \tau(\ca H)$, which finishes the proof. \end{proof}

\subsection{Upper bounds}\label{sec:upper_bound_tree}

In this section, we prove the two upper bounds in \cref{thm:tree_cover_main}. These are proven separately in two lemmas: \cref{lem:simpleUB} is a simple bound for when $\delta$ is large that is tight up to a constant factor, and \cref{lem:generalUB} is our general bound. 

\begin{lem}\label{lem:simpleUB}
    Let $r\geq 2$ and $\delta \in [1/2, 1)$. Any $r$-edge-coloured graph $G$ with $\delta(G) \geq (1 - \delta)|V(G)|$ satisfies
    \[tc(G) \leq \frac{2 r^2}{\log(1/\delta)}.\]
\end{lem}

\begin{proof}
    Consider $G$'s connectivity hypergraph $\ca C(G)$ and define
    \[X := \left\{v \in V(\ca C(G)): \deg_{\ca C(G)}(v) \geq \frac{1-\delta}{r} e(\ca C(G)) \right\}.\]
    By \cref{fact:intersectingaux}, $\ca C(G)$ is $\delta$-intersecting. By averaging, each edge $e \in E(\ca C(G))$ contains some vertex $x \in X$. Therefore, $X$ is a transversal for $\ca C(G)$.  

    Recall that $\ca C(G)$ is $r$-partite. If $V_1, \dots, V_r$ are the $r$ parts of $\ca C(G)$, note that $|V_i \cap X| \leq \frac{r}{1 - \delta}$ using the degree of each vertex in $X$ and the fact that every edge in $\ca C(G)$ contains a unique vertex from $V_i$. Hence,
    \[|X| \leq r \cdot \frac{r}{1-\delta} \leq \frac{2 r^2}{\log(1/\delta)}.\]
    To see the last inequality, let $\alpha := 1 - \delta$ and note
    \[\log(1/\delta) = \log\left( \frac{1}{1-\alpha} \right) \leq \log \left(1 + 2\alpha \right) \leq 2\alpha = 2(1-\delta),\]
    where we use $\alpha \leq 1/2$ in the first inequality. In conclusion, $tc(G) = \tau(\ca C(G)) \leq \frac{2r^2}{\log(1/\delta)}$. \end{proof}

Our general upper bound is derived from the following. 

\begin{lem}[Corollary 5.5 in \cite{bucickorandisudakov}]\label{lem:UBtreecover} Let $e^r \geq k > r \geq 2$ and let $m = \frac{4r}{\log k}$. Let $\ca H$ be an $r$-graph in which any $k$ edges have a transversal of size at most $r$. Then $\tau(\ca H) \leq 4r m \log m$.     
\end{lem}

\begin{lem}\label{lem:generalUB}
    Let $r \geq 2$ and $\delta \in (0, 1/2)$. Any $r$-edge-coloured graph $G$ with $\delta(G) \geq (1-\delta)|V(G)|$ satisfies 
    \[tc(G) \leq 100 r \cdot \left\lceil \frac{r}{\log(1/\delta)} \cdot \log\left( 1 + \frac{r}{\log(1/\delta)} \right) \right\rceil.\]
\end{lem}
\begin{proof}
First suppose that $\delta \geq r^{-2}$. Then we have
\[\frac{r}{\log(1/\delta)} \cdot \log\left(1 +  \frac{r}{\log(1/\delta)}\right) \geq \frac{r}{2\log r} \cdot \log\left(\frac{r}{2 \log r}\right) \geq \frac{r}{2\log r} \cdot \log(r^{0.1}) \geq \frac{r}{20},\]
where we used the fact that $r^{0.9}/(2\log r) \geq 1$ for $r \geq 2$. So, in this case the statement of the lemma is true by \cref{lem:simpleUB} applied with $r, 1/2, G$ playing the roles of $r, \delta, G$. Indeed, we have $\delta(G) \geq |V(G)|/2$ and so the lemma application yields a tree cover with $2r^2/\log(2) \leq 5r^2$ monochromatic trees.

Now let us assume that $\delta \in (0, r^{-2}]$. By \cref{fact:intersectingaux}, $\ca C(G)$ is $\delta$-intersecting. Let $\ca C'$ be the `flattened' $r$-graph obtained from $\ca C(G)$ by retaining only one copy of each of its edges. Let $k := \lceil 1/\delta - 1\rceil$, so that $k \geq r^2 - 1 > r$ and $k < 1/\delta$. Observe that, for any choice of $e_1, \dots, e_k \in E(\ca C')$, the number of edges of $\ca C(G)$ that intersect every $e_i$ is at least
\[(1 - k\delta)e(\ca C(G)) > 0.\]
 Since $e_1, \dots, e_k$ were arbitrary, this shows that any $k$ edges of $\ca C'$ have a transversal of size at most $r$. Let $k' := \min\{k, \lfloor e^r\rfloor \}$. By \cref{lem:UBtreecover} applied with $r, k', \ca C'$ playing the roles of $r, k, \ca H$, we conclude that
\begin{equation}\label{eq:transv}\tau( \ca C(G)) = \tau(\ca C') \leq \frac{16r^2}{\log k'} \cdot \log\left(\frac{4r}{\log k'} \right) \leq \max\left\{100r, \frac{16r^2}{\log k} \cdot \log\left(\frac{4r}{\log k} \right)\right\}.\end{equation}
But we have $k \geq 1/(2\delta) \geq \delta^{-1/2}$ since $\delta \leq r^{-2} \leq 1/4$, and so
\[\frac{16r^2}{\log k} \cdot \log\left( \frac{4r}{\log k}  \right) \leq \frac{32r^2}{\log(1/\delta)} \cdot \log\left(\frac{8r}{\log(1/\delta)}\right).\]
We also have 
\[100r \leq 100r \cdot \left\lfloor \frac{r}{\log(1/\delta)} \cdot \log\left(1+ \frac{r}{\log(1/\delta)} \right) \right\rfloor,\] as the term inside the floor is positive. This shows that the RHS in \eqref{eq:transv} is bounded above by the term from the lemma statement, which completes the proof since $tc(G) = \tau(\ca C(G))$ by \cref{fact:transversal}.
\end{proof}

\begin{proof}[Proof of \cref{thm:tree_cover_main}]
    The lower bound follows immediately from \cref{lem:lower_bound_tree}, whereas the upper bound follows by applying \cref{lem:generalUB} for the case $\delta \in (0,1/2)$ and \cref{lem:simpleUB} for $\delta \in [1/2, 1)$. 
\end{proof}

\section{Absorption}\label{sec:absorption}

In this section, we prove our main absorbing lemma, which will play a central role in our main proof by allowing us to cover a small leftover set of vertices with few monochromatic cycles. Roughly speaking, it says that in any edge-coloured bipartite graph $G[A,B]$ where $A$ is small relative to $B$, and every vertex of $A$ has many neighbours in $B$, one can cover $A$ with only a few monochromatic cycles.

\begin{lem}\label{lem:absorption}
    Let $1/n \ll 1/r, 1/K$ and $1/K \ll 1$, and let $\delta \in [r^{-7r}, 1/24]$. Let $G$ be an $r$-edge-coloured bipartite graph on vertex classes $A $ and $B$ with $|B| = n$ and $|A| \leq \frac{n}{r^{40r}}$. Suppose that $\deg(v, B) \geq (1-\delta)n$ for each $v \in A$. Then $G$ contains a collection of at most \[Kr \left\lceil\frac{r \log r}{\log(1/\delta)} \right\rceil\]
    vertex-disjoint monochromatic cycles covering $A$.
\end{lem}

As the proof of this lemma introduces several novel ideas, we now provide a brief sketch. Suppose we are given an $r$-edge-coloured bipartite graph $G[A,B]$ as in the statement. At a high level, the proof strategy is to reduce the problem from monochromatic cycle partitioning to tree covering, for which we can use the upper bound from \cref{thm:tree_cover_main}. To do so, we start with a cleaning procedure on $G$ ensuring that each of its monochromatic components, say $T$, satisfies the following two conditions.
\begin{enumerate}[label = (\alph*)]
    \item\label{prop:fewcycles} $T$ contains a small collection of monochromatic cycles $C_1, \dots, C_t$ covering almost all of the vertices of $V(T) \cap A$. 
    \item\label{prop:robust_conn} Any two vertices $v, u \in V(T)$ are connected by many paths that are internally vertex-disjoint from $V(C_i)$, $i \in [t]$. 
\end{enumerate}

If these two properties hold with suitable parameters, then \ref{prop:robust_conn} allows us to `stitch together' the cycles $C_1, \dots, C_t$ from \ref{prop:fewcycles} into a single monochromatic cycle covering almost all of $V(T) \cap A$. 

Following this approach, a cover of $A$ with few monochromatic trees can be transformed into an almost-cover of $A$ by monochromatic cycles. This still falls slightly short of \cref{lem:absorption}, however, which demands that the cycles be vertex-disjoint and cover all of $A$. While the first requirement can be guaranteed through a slight modification of the argument for \ref{prop:fewcycles}, the second presents a delicate issue. Indeed, even if the leftover is small relative to $A$, it is far from clear that covering it should be any easier than our original problem unless it is \emph{extremely} small.

To address this issue, our proof will need to be exceptionally efficient in controlling the leftover of the almost-cover, specifically ensuring that it contains not just a bounded number of vertices as a function of $r$, but in fact no more than $r^{O(r)}$ vertices. This is essentially the largest number of vertices that can be covered by a simple greedy argument with the prescribed number of cycles (see \cref{lem:bounded_leftover} below).

Let us outline how we attempt to prove \ref{prop:fewcycles} and \ref{prop:robust_conn} in order to achieve such a small leftover. Condition \ref{prop:robust_conn} is the main obstacle; by contrast, \ref{prop:fewcycles} is relatively straightforward, since it is generally easy to find a bounded collection of cycles covering $A$ when each component has not-too-small minimum degree (which can be guaranteed by deleting a small number of edges). To establish \ref{prop:robust_conn}, we introduce certain structures that we call \define{connecting hubs} (see \cref{defn:connecting_hub} below). Roughly speaking, a connecting hub $\mathcal H$ consists of a large set $S \subseteq B$ and a set of colours $C \subseteq [r]$ satisfying the following properties:
\begin{enumerate}[label = (\roman*)]
    \item\label{prop:robust_connect} for each $c \in C$, any two vertices of $S$ are connected by many internally vertex-disjoint $c$-coloured paths; and  
    \item\label{prop:colours} every vertex $v \in A$ either sends few edges to $S$ or many of these edges are $C$-coloured.
\end{enumerate}

It is not difficult to show that such a structure exists in $G[A,B]$. Iterating this, we obtain a family of vertex-disjoint hubs $\mathcal H_1, \dots, \mathcal H_\ell$ that together cover almost all of $B$, where each $\mathcal H_i$ consists of a set $S_i \subseteq B$ together with a colour set $C_i \subseteq [r]$. By discarding a few additional edges and vertices, we can further ensure that  
\begin{itemize}
    \item each $S_i$ is incident only with $C_i$-coloured edges, and for each $v \in A$ and $c \in [r]$, the degree $\deg_c(v, S_i)$ is either $0$ or large; and  
    \item for each $i,j \in [\ell]$ and $c \in [r]$, there are either no or many so-called \define{$c$-links} between $\mathcal H_i$ and $\mathcal H_j$: that is, vertices in $A$ having many $c$-coloured neighbours in both $S_i$ and $S_j$.
\end{itemize}

After these preparations, if two vertices $u,v \in A$ are connected by a $c$-coloured path in $G$, it becomes easy to find many internally vertex-disjoint $c$-coloured $u$--$v$ paths, as required by \ref{prop:robust_conn}. Indeed, any such path necessarily passes through a sequence of hubs $\mathcal H_{i_1}, \dots, \mathcal H_{i_\ell}$, where consecutive hubs admit many $c$-links. Many alternative paths can therefore be constructed by varying the $c$-link used between successive hubs and by choosing different internal paths within each hub (using property \ref{prop:robust_connect}).  

This is essentially enough to carry out our entire strategy. The hubs can only be guaranteed to satisfy $|S_i| \ge n / r^{\ca O(r)}$, which means that, for technical reasons, for condition \ref{prop:fewcycles} we cannot do better than $r^{\ca O(r)}$ cycles. Fortunately, the monochromatic connectivity we can guarantee inside the hubs (i.e. property \ref{prop:robust_connect}) is of order $r^{\Omega(r)}$ vertex-disjoint paths, which provides enough room to stitch these cycles into one. All of this comes at the cost of $r^{\ca O(r)}$ discarded vertices of $A$, which we handle separately, as noted earlier.

The remainder of the section is structured as follows. Connecting hubs are introduced in \cref{subsec:conn_hubs}, where we also show how to find a single one. In \cref{subsec:linked_families}, we construct families of hubs with the properties we need, which we refer to as \define{linked hub families}. \cref{subsec:covering_almost} shows how to cover all but $r^{O(r)}$ vertices with monochromatic cycles, whereas \cref{subsec:covering_few} shows how to cover the rest and proves \cref{lem:absorption}.

\subsection{Connecting hubs}\label{subsec:conn_hubs}

In order to formally introduce connecting hubs, we first require the following definition.

\begin{defn}[contracted graph]
Let $k \geq 1$ and let $G$ be a bipartite graph on vertex classes $A$ and $B$. The \define{contracted graph} $\tilde{G}(A, B, k)$ is defined as the graph on vertex set $A$ and edge set consisting of all pairs of distinct vertices $u, v \in A$ such that \(\codeg_G(u, v) \geq k\).
\end{defn}

We record, for later use, the following simple fact about contracted graphs.

\begin{fact}\label{fact:small_alpha}
    Let $k \geq 1$. Let $G$ be a bipartite graph on vertex classes $A$ and $B$ with $\deg(v, B) \geq \frac{|B|}{k}$ for each $v \in A$. Then $\tilde{G}(A, B, \frac{|B|}{5k^2})$ has no independent set of size $2k$.
\end{fact}
\begin{proof}
    Suppose to the contrary that it contains an independent set $Y$ of size $2k$. Then, by the inclusion-exclusion principle,
    \begin{equation*}\begin{split}|N_{G}(Y)| &\geq \sum_{v \in Y} \deg_G(v, B) - \sum_{vu \in \binom{Y}{2}}\codeg_G(v,u, B) \\ &\geq \frac{|Y| |B|}{k} - \binom{|Y|}{2} \frac{|B|}{5k^2} \\&\geq 2|B| - 4k^2\cdot \frac{|B|}{5k^2} > |B|, \end{split}\end{equation*}
    which contradicts the fact that $N_G(Y) \subseteq B$. 
\end{proof}

Now we have the main definition of this section. Note that one of its conditions requires the notion of an $(\eps,t)$-expander, as in \cref{defn:sublinear_expander}. 

\begin{defn}[connecting hub]\label{defn:connecting_hub}
Let $r, t, m\geq 1$ and $\eps > 0$. Let $G$ be an $n$-vertex $r$-edge-coloured bipartite graph on vertex classes $A$ and $B$ and let $X \subseteq A$. An \define{$(r,t,m, \eps)$-connecting hub} for $X$ is a triple $(C, S, \ca T)$ where
\begin{itemize}
    \item $C \subseteq [r]$ is the \define{colour set} of the hub,
    \item $S \subseteq B$ is a set of size $m$ referred to as its \define{core}, and
    \item $\ca T = \{T_c: c \in C\}$ and all the $T_c$ are disjoint subsets of $A \setminus X$ with $t \leq |T_c| \leq 80rt$, referred to as its \define{routing sets}.
\end{itemize}
Moreover, $(C, S, \ca T)$ is required to satisfy the following properties. 
\begin{enumerate}[label = (H\arabic*)]
    \item\label{connhub:1} For each $c \in C$, the contracted graph $\tilde{G_c}(T_c, B, m)$ is an $(\eps, 1)$-expander;
    \item\label{connhub:2} For each $v \in S$ and $c \in C$, $\deg_c(v, T_c) \geq \frac{|T_c|}{20r}$;
    \item\label{connhub:3} For each $v \in X$, either $\deg(v, S) \leq |S|/2$ or there is some $c \in C$ such that $\deg_c(v, S) \geq \frac{|S|}{10r}$. 
\end{enumerate}

\end{defn}

See \cref{fig:connectinghub} for a visual depiction of this notion. As a simple example, note that a set $S$ of size $m$ with distinct vertices $v_1, \dots, v_k \in A$ and colours $C = \{c_1, \dots, c_k\}$ such that $S \subseteq N_{c_i}(v_i)$ for all $i \in [k]$ constitutes an $(r,1,m, \eps)$-connecting hub if \ref{connhub:3} is satisfied (here we take $T_{c_i} = \{v_i\}$). Indeed, since each $T_{c_i}$ is just a singleton set, \ref{connhub:1} holds trivially. This structure was originally studied by \cite{haxellkohayakawa} for monochromatic tree partitioning in complete graphs, and served as inspiration for the definition above.

\begin{figure}
    \centering
    \includegraphics[width=0.5\linewidth]{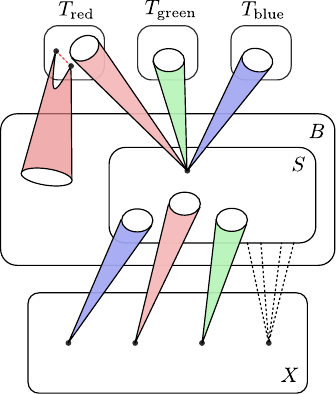}
    \caption{A connecting hub with colour set $C = \{\textrm{red}, \textrm{green}, \textrm{blue}\}$. For each $c \in C$, every vertex in $S$ sends many $c$-coloured edges to $T_c$. Moreover, every vertex in $X$ either sends many $c$-coloured edges to $T_c$ for \emph{some} $c \in C$, or it sends few edges to $S$. Finally, the dashed red edge in $T_\textrm{red}$ represents an edge of $\tilde{G}_\textrm{red}(T_\textrm{red}, B,m )$, which guarantees a large red coneighbourhood in $B$.}  
    \label{fig:connectinghub}
\end{figure}

\begin{rmk}
    The definition of a connecting hub allows $C =\emptyset$, in which case the hub exhibits no useful connectivity properties. If this happens, however, note that every $v \in X$ must satisfy $\deg(v, S) \leq |S|/2$ by \ref{connhub:3}. In applications, we will consider a graph with high minimum degree and find many hubs inside of it. Though some of these may satisfy $C = \emptyset$, the minimum degree ensures that most of them do not (which is ultimately necessary for the proof). Including this degenerate case in the definition allows us to show that connecting hubs can be found in arbitrary edge-coloured graphs.
\end{rmk}

Now we show how to find, in an edge-coloured bipartite graph $G[A,B]$, a connecting hub for a set $A' \subseteq A$ covering most of $A$.

\begin{lem}\label{lem:connectinghub} Let $t, r \geq 2$, let $\eps \ll 1$ and $1/n \ll 1/r$. Let $m :=  \lfloor r^{-7r}n \rfloor$. Let $G$ be an $r$-edge-coloured bipartite graph on vertex classes $A$ and $B$ with $|A| \leq |B| = n$. Then $G$ contains an $(r, t, m,  \eps)$-connecting hub for some set $A' \subseteq A$ of size $|A'| \geq |A| - 240r^2t$. 
\end{lem}

\begin{proof}
    Let $C$ be a maximal subset of $[r]$ such that the following holds. There exist sets $S' \subseteq B$ and $\mathcal{T}= \{T_c: c\in C\}$ with $T_c \subseteq A$ for each $c \in C$ satisfying
    \begin{enumerate}[label = (\roman*)]
        \item\label{prop:Tc_disj} all the $T_c$ are pairwise disjoint and satisfy $t \leq |T_c| \leq 80rt$, 
        \item\label{prop:S_large} letting $\ell := |C|$, we have $|S'| \geq \frac{n}{(20r)^{\ell}}$,
        \item\label{prop:H1} for each $c \in C$, the contracted graph $\tilde{G}_c(T_c, B, m)$ is an $(\eps, 1)$-expander, and
        \item\label{prop:H2} for each $v \in S'$ and $c \in C$, $\deg_c(v, T_c) \geq \frac{|T_c|}{20r}$.
    \end{enumerate}

    Note that $C$ is well-defined since conditions \ref{prop:Tc_disj}--\ref{prop:H2} are satisfied by setting $C := \emptyset, S' = B$ and $\ca T = \emptyset$ (all the properties regarding the $T_c$ are vacuously satisfied). 

    \begin{claim}\label{claim:allbutfew}
        All but at most $160r^2t$ vertices $v \in A$ satisfy: either $\deg(v, S') \leq |S'|/4$ or there is some $c \in C$ such that $\deg_c(v, S') \geq \frac{|S'|}{8r}$.  
    \end{claim}

    \begin{proof}[Proof of claim.]
    Suppose to the contrary that there are at least $160r^2t$ of vertices $v \in A$ with $\deg(v, S') \geq |S'|/4$ and, for each $c  \in C$, $\deg_c(v, S') \leq \frac{|S'|}{8r}$. Note that by \ref{prop:Tc_disj} at least $80r^2t$ of these vertices are not contained in any $T_c$. Each of these vertices satisfies
    \[\deg_{[r]\setminus C}(v, S') = \deg(v, S') - \deg_{C}(v, S') \geq \frac{|S'|}{4} - \frac{|C||S'|}{8r} \geq \frac{|S'|}{8}.\]
    This implies that there is a colour $c_v \in [r]\setminus C$ for which $\deg_{c_v}(v, S') \geq \frac{|S'
    |}{8r}.$
    Then, we can find a colour $c_0 \in [r] \setminus C$ and a set $X \subseteq A$ with $|X| = 80rt$ such that $|X \cap T_c| = \emptyset$ for each $c \in C$ and, for each $v \in X$, 
    \[\deg_{c_0}(v, S') \geq \frac{|S'|}{8r}.\]
    
    Now, let us consider the contracted graph $\tilde{G} := \tilde{G}_{c_0}(X, S', (20r)^{-2}|S'|)$, whose vertex set is $X$. If the complement of $\tilde{G}$ had more than $\left(1 - \frac{1}{20r}\right)\frac{|X|^2}{2}$ edges, then it would contain a clique on $20r$ vertices by Tur\'an's theorem. However, this yields an equally large independent set in $\tilde{G}$, which cannot exist since $\deg_{c_0}(v, S') \geq \frac{|S'|}{8r}$ for each $v \in X$ by \cref{fact:small_alpha} applied with $8r$ playing the role of $k$. Thus, we have
        \[e(\tilde{G}) \geq {\binom{|X|}{2}} - \left(1 - \frac{1}{20r}\right) \frac{|X|^2}{2} \geq \frac{|X|^2}{40r} - \frac{|X|}{2} \geq \frac{|X|^2}{80r},\]
    where the last inequality used the fact that $|X|\geq 40r$. This implies that
    \[d(\tilde{G}) \geq \frac{|X|}{40r} \geq 2t.\]
    By \cref{thm:sublinearexpander}, $\tilde{G}$ contains an $(\eps,1)$-expander with average degree at least $d(\tilde{G})/2 \geq t$. In particular, since the $(\cdot,\cdot)$-expander property is monotone, there is a subset $W \subseteq X$ such that $\tilde{G}[W]$ is an $(\eps,1)$-expander with $d(\tilde{G}) \geq t$. Combining this with $W \subseteq X$ yields $t \leq |W| \leq |X| \leq 80rt$. Note that $\tilde{G}[W]$ is just $\tilde{G}_{c_0}(W, S', (20r)^{-2}|S'|)$. 

    Recall that each vertex $v \in W \subseteq X$ satisfies $\deg_{c_0}(v, S') \geq |S'|/(8r)$. Let $S'' \subseteq S'$ be the subset of vertices $u$ satisfying $\deg_{c_0}(u, W) \geq \frac{|W|}{20r}$. By double counting the number of edges between $W$ and $S'$, we obtain 
    \[ \frac{|W||S'|}{8r} \leq e_{c_0}(W,S') \leq |W||S''| + \frac{|W||S'|}{20r}.\]
 This rearranges to 
    \begin{equation}\label{eq:sizeofS''}|S''| \geq \frac{|S'|}{20r} \geq \frac{n}{(20r)^{\ell+1}},\end{equation}
    where we used \ref{prop:S_large}.
Now let $T_{c_0} := W$. We claim that \ref{prop:Tc_disj}--\ref{prop:H2} are satisfied by taking $C'' := C \cup \{c_0\}$, $S''$ and $\ca T'' := \ca T \cup \{T_{c_0}\}$. Indeed, property \ref{prop:Tc_disj} holds by construction and \ref{prop:S_large} follows from \eqref{eq:sizeofS''}. For properties \ref{prop:H1} and \ref{prop:H2}, it suffices to show that the relevant properties hold for $c_0$, as we already know they hold for all $c \in C$. First, note that for each edge $vu$  in $\tilde{G}[T_{c_0}] =\tilde{G}_{c_0}(T_{c_0}, S', (20r)^{-2}|S'|)$ we have
\[\codeg(v, u, B) \geq \codeg(v,u, S') \geq \frac{|S'|}{400r^2} \geq \frac{n}{(20r)^{\ell+2}} \geq \frac{n}{r^{7r}} \geq m,\] and so $\tilde{G}[T_{c_0}]$ is a subgraph of  $\tilde{G}_{c_0}(T_{c_0}, B, m)$. Since the first is an $(\eps, 1)$-expander, so is the second. Finally, our choice of $S''$ ensures that $\deg_{c_0}(u, T_{c_0}) \geq \frac{|T_{c_0}|}{20r}$. This shows that $C'', \ca T'', S''$ satisfy \ref{prop:Tc_disj}--\ref{prop:H2} where $C \subsetneq C''$, contradicting the maximality of $C$. This completes the proof of claim.   
    \end{proof}

To finish the proof, we select uniformly at random a subset $S\subseteq S'$ of size $m$. This is indeed possible since $m := \lfloor r^{-7r} n\rfloor \leq |B|/(20r)^\ell \leq |S'|$ by \ref{prop:S_large}. 

We claim that with positive probability all but at most $160r^2t$ vertices $v \in A$ satisfy: either $\deg(v, S) \leq |S|/2$ or there is some $c \in C$ such that $\deg_c(v, S) \geq |S|/(10r)$. From \cref{claim:allbutfew}, it suffices to show that this holds for vertices $v \in A$ such that either $\deg(v, S') \leq |S'|/4$ or there is some $c \in C$ such that $\deg_{c}(v, S') \geq |S'|/(8r)$. If $\deg(v, S') \leq |S'|/4$, then $\mu_v := \mathbb{E}[\deg(v, S)] = \frac{\deg(v, S')|S|}{|S'|} \leq m/4$. So, by Chernoff's inequality for the hypergeometric distribution (\cref{lem:chernoff}), we have
\[\mathbb{P}\left(\deg(v, S) \geq |S|/2 \right) \leq \mathbb{P}\left(|\deg(v, S) - \mu_v| \geq m/4 \right) \leq 2e^{-\frac{m^2}{12m}} \leq e^{-\sqrt{n}}.\]
If instead $\deg_{c}(v, S') \geq |S'|/(8r)$ for some $c \in C$, then $\mu_v:= \mathbb{E}[\deg_{c}(v, S)] \geq m/(8r)$ and, again by Chernoff's inequality,
\[\mathbb{P}\left(\deg_c(v, S) \leq |S|/(10r)\right) \leq \mathbb{P}(|\deg_c(v, S) - \mu_v| \geq m/(40r)) \leq 2e^{-\frac{m}{3(40r)^2}} \leq e^{-\sqrt{n}}.\]
In conclusion, with probability at least $1 - |A| e^{-\sqrt{n}} \geq 1 - ne^{-\sqrt{n}} > 0$, all but at most $160r^2t$ vertices $v \in A$ satisfy: either $\deg(v, S) \leq |S|/2$ or there is some $c \in C$ such that $\deg_c(v, S) \geq |S|/(10r)$. After fixing such a choice of $S$, we let $A' \subseteq A$ be precisely the set of vertices satisfying this condition that are not in $\bigcup_{c \in C} T_c$, so that $|A'| \geq |A| - 160r^2t - 80r^2t \geq |A| - 240r^2t$ using property \ref{prop:Tc_disj}.

Now we verify that $(C, S, \mathcal{T})$ is a $(r,t,m,\eps)$-connecting hub for $A'$. Note that the conditions on the sizes of the sets in $(C, S, \mathcal{T})$ indeed hold. Property \ref{connhub:1} corresponds directly to \ref{prop:H1} and property \ref{connhub:2} corresponds to \ref{prop:H2}. We argued that \ref{connhub:3} holds in the previous paragraph.
\end{proof}

\subsection{Linked hub families}\label{subsec:linked_families}

Two $(r, t, m, \eps)$-connecting hubs $\ca H = (C, S, \ca T)$ and $\ca H' = (C', S' , \ca T')$ for some set $X$ are said to be vertex-disjoint if $S$ and $S'$ are disjoint and so are any $T \in \ca T$ and $T' \in \ca T'$. As discussed at the start of \cref{sec:absorption}, we will need to find not just one connecting hub, but many vertex-disjoint hubs related to each other in a particular way. This is codified in the following definition.

\begin{defn}[linked hub family]\label{defn:linkedfamily}
    Let $r, t, m\geq 1$ and $\eps > 0$. Let $G$ be an $r$-edge-coloured bipartite graph on vertex classes $A$ and $B$. An \define{$(r,t,m,\eps)$-linked hub family} for $X \subseteq A$ is a collection of pairwise vertex-disjoint $(r, t,m, \eps)$-connecting hubs $\ca H_1, \dots, \ca H_{\ell}$ for $X$, where $\ca H_i = (C_i, S_i, \ca T_i)$, satisfying the following property. For each $1 \leq i < j \leq \ell$ and $c \in [r]$, the set $L_c(i,j) \subseteq X$ defined as 
    \[L_{c}(i,j) := \left\{v \in X: \deg_c(v, S_i), \deg_c(v, S_j) \geq \frac{m}{10r} \right\}\]
    is either empty or of size at least $t$. We refer to the elements of $L_c(i,j)$ as the \define{$c$-links} between $\ca H_i$ and $\ca H_j$. 
\end{defn}

Note that $t$ controls two distinct features of this structure: the size of the routing sets inside the connecting hubs and the size of the sets of $c$-links between hubs. We use the same parameter for both since the play a similar role in the proof, namely providing enough space to build vertex-disjoint connecting paths. 

Now we prove that any edge-coloured bipartite graph $G[A,B]$ contains a linked hub family for nearly all of $A$, where in addition the cores of the hubs in the family covers all but a $r^{-\Omega(r)}$-fraction of $B$.

\begin{lem}\label{lem:linkedfam}
Let $r \geq 2$, let $1/K, \eps \ll 1$ and $1/n \ll 1/r$. Let $m := \lfloor r^{-14r}n \rfloor$. Let $G$ be an $r$-edge-coloured bipartite graph on vertex classes $A$ and $B$ with $|B| = n$ and $|A| \leq m$. Then $G[A,B]$ contains a set $A' \subseteq A$ with $|A'| \geq |A| - r^{3Kr}$ and a $(r, r^{Kr}, m , \eps)$-linked hub family $\ca H_1, \dots, \ca H_\ell$ for $A'$ such that
\[\big| B \setminus \bigcup_{i \in [\ell]} S_i| \leq \frac{n}{r^{7r}},\]
where $S_i$ is the core of $\ca H_i$ for each $i \in [\ell]$.
\end{lem}

\begin{proof}
First, we construct a collection of vertex-disjoint hubs by iteratively applying \cref{lem:connectinghub}. Let $\ell \geq 0$ be maximal such that $G[A,B]$ contains a set $A' \subseteq A$ with $|A'| \geq |A| - \ell \cdot r^{2Kr}$ and vertex-disjoint $(r,r^{Kr}, m, \eps)$-connecting hubs $\ca H_1, \dots, \ca H_\ell$ for $A'$. Note that $\ell$ is well-defined since we can take $\ell = 0$ and $A' = A$ and the condition is vacuously satisfied. For each $i \in [\ell]$, let $\ca H_i = (C_i, S_i, \ca T_i)$. Note that $S_i \subseteq B$ for each $1 \leq i \leq \ell$ and $|S_i| = m \geq \frac{n}{r^{15r}}$. The $S_i$ are pairwise disjoint, which ensures that $\ell \leq r^{15r}$ and so $|A'| \geq |A| - r^{(2K + 15)r}$. 

We claim that
\begin{equation}\label{eq:smallleftoverinB}\big| B \setminus \bigcup_{i \in [\ell]} S_i \big| \leq \frac{n}{r^{7r}}.\end{equation}
Suppose that this is not the case. We will show that on this assumption a new $(r, r^{Kr}, m, 
\eps)$-connecting hub $\ca H_{\ell + 1}$ can be added to this collection, contradicting the maximality of $\ell$. 

Let $B' \subseteq B \setminus \bigcup_{i \in [\ell]}S_\ell$ be a set of size precisely equal to $r^{7r}m$ (which exists since $r^{7r}m \leq r^{7r} (r^{-14r} n) = r^{-7r}n$). By \cref{defn:connecting_hub}, $A'$ is disjoint from all the routing sets $T_c \in \ca T_i$ for each $i \in [\ell]$.

Now we apply \cref{lem:connectinghub} to $G[A', B']$ with $r, r^{Kr}, \eps, |B'|, m $ playing the roles of $r,t,\eps, n,m$ (note that $|A'| \leq m \leq  |B'|$). Thus, $G[A', B']$ contains a $(r, r^{Kr}, m, \eps)$-connecting hub $\ca H_{\ell+1}$ for a set $A^* \subseteq A'$ of size
\[|A^*| \geq |A'| - 240r^2 \cdot r^{Kr} \geq |A| - (\ell+1) r^{2Kr}. \]
Each $\ca H_j$ with $1 \leq i \leq \ell$ is vertex-disjoint from $\ca H_{\ell+1}$ and it is also a $(r, r^{Kr}, m ,\eps)$-connecting hub for $A^*$, which contradicts the maximality of $\ell$. This shows that \eqref{eq:smallleftoverinB} holds.

It remains to find a suitable subset $A^\dagger \subseteq A'$ so that $\ca H_1, \dots, \ca H_\ell$ forms a $(r,r^{Kr}, m, \eps)$-linked hub family with respect to $A^\dagger$. To this end, we run the following iterative procedure: as long as there is a pair $\{i,j\} \in \binom{[\ell]}{2}$ and some $c \in [r]$ with $0 < |L_c(i,j)| \leq r^{Kr}$, remove all the elements of $L_c(i,j)$ from $A'$ and continue (the other sets $L_{c'}(i', j')$ are updated analogously, retaining only the surviving elements of $A'$). We terminate once this is no longer possible, and let $A^\dagger$ be the resulting set. Each choice of $i, j$ and $c$ is processed at most once in this procedure, and when this occurs no more than $r^{Kr}$ vertices are discarded from $A'$. Thus,
\[|A^\dagger| \geq |A'| - {\binom{\ell}{2} } \cdot r^{Kr+1} \geq |A'| - r^{30r} \cdot r^{Kr + 1} \geq |A| - r^{3Kr}.\]
The procedure's termination guarantees that every non-empty $L_c(i,j)$ satisfies $|L_c(i,j)| \geq r^{Kr}$. So, $\{\ca H_1, \dots, \ca H_\ell\}$ is a $(r, r^{Kr}, m, \eps)$-linked hub family for $A^\dagger$, as desired. \end{proof}

Given a linked hub family $\ca F$ for a set $X$ in an edge-coloured bipartite graph $G[A, B]$, we now introduce a new graph $R(X, \ca F)$ which only encodes the adjacencies in $G$ that are relevant to the family's special properties. This can be viewed as an analogue of the reduced graph in the regularity method, and it is ultimately the graph to which our tree covering upper bound is applied. We will then show that any $u,v  \in X$ that belong to the same $c$-coloured component in $R(X, \ca F)$ are robustly connected in $G_c$. 

Let $G[A,B]$ be an $r$-edge-coloured bipartite graph. Let $\ca F = \{\ca H_1, \dots, \ca H_\ell\}$ be an $(r,t,m\, \eps)$-linked hub family for a set $X \subseteq A$, where $\ca H_i= (C_i, S_i, \ca T_i)$ for each $i\in [\ell]$. We define the \define{simplified graph}, denoted \define{$R(X, \ca F)$}, as the $(r+1)$-edge-coloured graph on vertex set $X \cup \bigcup_{i \in [\ell]} S_i$ whose adjacencies and colours are given by the following rules:
\begin{itemize}
    \item $X$ induces an independent set;
    \item $\bigcup_{i \in [\ell]} S_i$ induces a complete graph in which all edges take colour $r+1$;
    \item if $v \in X$ and $\deg_G(v, S_i) \leq \frac{m}{2}$, then $v$ sends no edges to $S_i$; 
    \item if $v \in X$ and $\deg_G(v, S_i) > \frac{m}{2}$, then, for some $c \in C_i$ such that $\deg_c(v, S_i) \geq \frac{m}{10r}$, every edge between $v$ and $S_i$ is present and $c$-coloured. 
\end{itemize}

The fact that each $\ca H_i$ is an $(r,t,m,\eps)$-connecting hub guarantees that this graph is well-defined: by \ref{connhub:3}, each $v \in X$ with $\deg(v, S_i) > m/2$ must satisfy $\deg_c(v, S_i) \geq \frac{m}{10r}$ for some $c \in C_i$. However, the simplified graph is not necessarily unique as a given $v \in X$ could satisfy $\deg_c(v, S_i) \geq \frac{m}{10r}$ for different choices of $c \in C_i$. So, whenever we talk about $R(X, \ca F)$, we will simply assume that an arbitrary graph satisfying the conditions set above has been fixed. 

Let us remark that adding the colour $r+1$ serves the purpose of `boosting' the minimum degree of $R(X, \ca F)$ so that we can apply our bounds on $tc_{r+1}(\delta)$, which only concern graphs with very high minimum degree. This causes no issue as in this graph the vertices of $A$---that is, the vertices we aim to cover in \cref{lem:absorption}---are not incident with any $(r+1)$-coloured edges. 

Next, we show that monochromatic connectivity in $R(X, \ca F)$ translates to robust monochromatic connectivity in $G$.

\begin{lem}\label{lem:connecting_lemma}
Let $r, \ell, t, m \geq 2$ and $\eps > 0$ such that $1/m \ll 1/t \ll \eps \ll 1$ and $t \geq 20r^4$. Let $G$ be an $r$-edge-coloured bipartite graph on vertex classes $A$ and $B$. Let $\ca F = \{\ca H_1, \dots, \ca H_\ell\}$ be an $(r,t,m,\eps)$-linked hub family for $X \subseteq A$, where $\ca H_i = (C_i, S_i, \ca T_i)$ for each $i \in [\ell]$. Suppose that $U \subseteq A \cup B$ satisfies 
    \begin{enumerate}[label = \textnormal{(U\arabic*)}]
        \item\label{prop:intersecB'} $|U \cap B| \leq \frac{m}{20r}$, 
        \item\label{prop:intersecTc} $|U \cap T| \leq t^{1/2}$ for each $i \in [\ell]$ and $T \in \ca T_i$, and
        \item\label{prop:intersecL'} $|L_{c}(i,j) \setminus U| \geq \ell$ for each distinct $i, j \in [\ell]$ and $c \in [r]$ such that $L_{c}(i,j)$ is non-empty.
    \end{enumerate}
    Let $x, y \in X \setminus U$ be distinct and suppose that $R(X, \ca F)$ contains a $c$-coloured path from $x$ to $y$. Then $G$ contains a $c$-coloured $x$--$y$ path avoiding $U$ of length at most $\ell \cdot \log^5 t$. 
\end{lem}

\begin{proof}
    For each $i \in [\ell]$, let $\ca H_i := (C_i, S_i, \ca T_i)$. Let $Y := \bigcup_{i \in [\ell]} S_i$ and recall from its definition that $R(X, \ca F)$ is an $(r+1)$-edge-coloured graph on vertex set $X \cup Y$. Also recall that, in $R(X, \ca F)$, the vertices $x, y \in X$ are only incident with edges whose colours are in $[r]$, and so the given monochromatic $x$--$y$ path, say $P$, takes a colour $c \in [r]$. But the subgraph of $R(X, \ca F)$ induced by colours in $[r]$ is bipartite with partition $X\cup Y$, which implies that $P$ alternates between vertices of $X$ and $Y$. So, let us write
    \[P = x u_1 v_1 u_2 v_2 \dots u_k y,\]
    where $u_i \in Y$ and $v_i \in X$ for each $i \in [k]$. For each $i \in [k]$, let $a_i \in [\ell]$ be the unique integer with $u_i \in S_{a_i}$. Crucially, observe that for all $i \in [k-1]$, if $S_{a_i} \neq S_{a_{i+1}}$, then we have that $v_i$ is a $c$-link between $\ca H_{a_i}$ and $\ca H_{a_{i+1}}$ (as defined in \cref{defn:linkedfamily}). Indeed, the fact that there is a $c$-coloured edge from $v_i$ to $u_{i} \in S_{a_i}$ implies that $\deg_c(v_i, S_{a_i}) \geq \frac{m}{10r}$ by the definition of $R(X, \ca F)$, and the same holds with respect to $S_{a_{i+1}}$. For the same reason, we have $\deg_c(x, S_{a_1}), \deg_c(y, S_{a_k})\geq \frac{m}{10r}$. 

    The sequence $\ca H_{a_1} \dots \ca H_{a_k}$ can be viewed as a `walk' from $\ca H_{a_1}$ to $\ca H_{a_k}$, in the sense that for any two consecutive hubs in this sequence their set of $c$-links $L_c(a_i, a_{i+1})$ is non-empty. Now we take such a walk of minimal length \(W = \ca H_{b_1} \dots \ca H_{b_s}, \) where $b_1 = a_1$ and $ b_s = a_k$. By the minimality of this walk, the indices $b_1, \dots, b_s$ are all distinct elements of $[\ell]$ (in other words, $W$ is a path), and so $s \leq \ell$. Using property \ref{prop:intersecL'}, we can easily choose $s-1 \leq \ell$ distinct vertices $q_1, \dots, q_{s-1}$ such that $q_i \in L_c(b_i, b_{i+1}) \setminus U$. Now we build a $c$-coloured path from $x$ to $y$ by showing that any two consecutive vertices in $x, q_1, \dots, q_{s-1}, y$ are connected by a short $c$-coloured path passing through a hub. 

    The argument is the same for any of these pairs, so it suffices to show how to connect $x$ to $q_1$. Observe that $\deg_c(x, S_{b_{1}}), \deg_c(q_1, S_{b_1}) \geq \frac{m}{10r}$, and so we can pick distinct vertices $w,z \in S_{b_1} \setminus U$ such that $xw, zq_1 \in E(G_c)$ (using property \ref{prop:intersecB'}). Note that $x$ being incident with $c$-coloured edges to $S_{b_1}$ in $R(X, \ca F)$ implies that $c$ belongs to the colour set $C_{b_1}$ of $\ca H_{b_1}$, and thus $\ca T_{b_1}$ contains a routing set $T_c$ in colour $c$. Now, by property \ref{connhub:2}, we have
    \[\deg_c(z, T_c), \deg_c(w, T_c) \geq \frac{|T_c|}{20r} \geq \frac{t^{3/4}}{15},\]
    where we used $|T_c| \geq t$ by \cref{defn:connecting_hub} together with $t \geq 20r^4$. 
    
    Recall from property \ref{connhub:1} that $\tilde{G_c}(T_c, B, m)$ is an $(\eps, 1)$-expander. So, we can apply \cref{thm:conn_in_expanders} to this graph with $N_c(z, T_c), N_c(w, T_c)$ and $U \cap T_c$ playing the roles of $U_1,U_2$ and $W$ and $t^{3/4}/15$ playing the role of $x$. Let us verify that its assumptions are met. Note that
    \[\rho\left(\frac{t^{3/4}}{15}\right) = \frac{\eps}{\log^2(t^{3/4})} =  \frac{16\eps}{9\log^2 t}.\]
   As noted earlier, we have $|N_c(z, T_c)|,|N_c(w, T_c)| \geq t^{3/4}/15$. Furthermore, \[|U \cap T_c| \leq t^{1/2} \leq \frac{\eps t^{3/4}}{36\log^2 t} \leq \frac{1}{4} \cdot \rho\left(\frac{t^{3/4}}{15}\right) \cdot \frac{t^{3/4}}{15},\]
    where in the second inequality we used the fact that $1/t \ll \eps$. Therefore, \cref{thm:conn_in_expanders} yields a path from $N_c(z, T_c) $ to $N_c(w, T_c)$ in $\tilde{G}_c(T_c, B, m) - U$ of length at most 
    \[\frac{2}{\eps} \log^3(15|T_c|) \leq \frac{2}{\eps} \log^3(1200rt) \leq \log^4 t,\]
    where the first inequality used the fact that $|T_c| \leq 80rt$ from \cref{defn:connecting_hub}, and the second used $t \geq 20 r^4$ and $1/t \ll \eps$. Let this path be $P' = p_1 \dots p_{s'}$, where $s' \leq \log^4 t.$ For any $1 \leq i \leq s' - 1$, we have $\codeg_c(p_i, p_{i+1}, B) \geq m$ by the definition of $\tilde{G_c}(T_c, B, m)$. Thus, 
    \[\codeg_c(p_i, p_{i+1}, B) - |U| - \left|\{z,w\} \right| \geq m -\frac{m}{20r} - 2 \geq s'-1,\]
    using the fact that $1/m \ll 1/t$. We can thus easily pick $s'-1$ distinct vertices $d_1, \dots, d_{s'-1} \in B \setminus (U \cup \{z,w\})$ such that $p_id_i$ and $d_i p_{i+1}$ are $c$-coloured edges. This yields the $c$-coloured path
    \[xwp_1d_1p_2d_2 \dots d_{s'-1} p_{s'}zq_1,\]
    connecting $x$ to $q_1$ while avoiding $U$. This path, which we denote by $Q_1$, is of length at most $2s' + 4 \leq 3 \log^4 t$.

    The argument just given can now be repeated for each pair $q_i, q_{i+1}$ with $1 \leq i \leq s-2$, and also for the pair $q_{s-1}, y$. In each iteration, we find unused $z_i, w_i \in S_{b_{i+1}}$ such that $q_iz_i, w_iq_{i+1}$ are $c$-coloured edges (with $q_{s} := y$), and then connect them via a $c$-coloured path $Q_{i+1}$ of length at most $3 \log^4 t$ using the routing set $T_c \in \ca T_{b_{i+1}}$ of $\ca H_{b_{i+1}}$. Note that we can always pick from a non-empty set of unused neighbours or coneighbours throughout the procedure since the total number of unavailable vertices in $B$ is at most $|U| + s \cdot 3 \log^4 t \leq m/(15r)$, whereas the lower bounds on degrees and codegrees into $B$ that we use are at least $m/(10r)$. Furthermore, for each $i \in [s]$ the set $T_c \in \ca T_{b_{i+1}}$ is only used once, namely to connect $q_i$ and $q_{i+1}$, and so the inequalities regarding available vertices in $T_c$ remain the same as above. 

    \begin{figure}
        \centering
        \includegraphics[width=.8\linewidth]{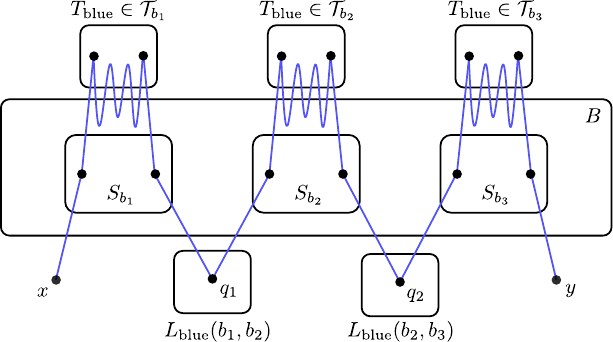}
        \caption{A visual depiction of the construction of the monochromatic $x$--$y$ path in \cref{lem:connecting_lemma}, assuming for simplicity that $s = 3$.}
        \label{fig:connectinglemma}
    \end{figure}

    At the end of this procedure, we obtain a $c$-coloured path 
    \[xQ_1q_1 Q_2 q_2 \dots q_{s-1} Q_s y \]
    avoiding $U$ of length at most $3s \log^4 t \leq \ell \cdot \log^5 t$. See \cref{fig:connectinglemma} for a visual of the final path (assuming, for simplicity, that $s= 3$ and $c= \textrm{blue}$). This completes the proof of the lemma.
\end{proof}

Our next lemma shows how to lower bound the minimum degree of $R(X, \ca F)$ as a function of the degrees of vertices in $X$.

\begin{lem}\label{lem:mindeg} Let $n,m, r, t \geq 2$, let $1/n \leq \delta \leq 1/2$ and $\eps > 0$.
    Let $G$ be an $r$-edge-coloured bipartite graph on vertex classes $A$ and $B$ with $|B| = n$. Let $\ca F = \{\ca H_1, \dots, \ca H_\ell\}$ be a $(r,t,m ,\eps)$-linked hub family for a non-empty set $X \subseteq A$ with $|X| \leq \delta n$. Suppose that $\deg(v, \bigcup_{i\in[\ell]}S_i) \geq (1- \delta)n$ for each $v \in X$, where $S_i$ is the core of $\ca H_i$ for each $i \in [\ell]$. Then
    \[\delta \left(R(\ca F, X)\right) \geq (1 - 6\delta)|R(X, \ca F)|. \]
\end{lem}
\begin{proof}
    Let $Y := \bigcup_{i \in [\ell]} S_i$. Note that, in $R := R(X, \ca F)$, $Y$ forms a monochromatic $(r+1)$-coloured clique. Since $X$ is non-empty, some element of $X$ has at least $(1- \delta)n$ neighbours in $Y$. This implies
    \[\deg_R(v) \geq |Y| - 1 \geq (1- 2\delta)n \]
    for each $v \in Y$ (also using $1/n \leq \delta$). 

    For each $u \in X$, let $Q_u := \{ i \in [\ell]: \deg_G(u, S_i) \leq m/2\}$. We can count the degree of $u$ into $Y$ in two ways as
    \[(1 - \delta)n \leq \deg_G(u, Y) \leq \frac{|Q_u|m}{2} + (\ell - |Q_u|)m.\]
    However, we have $\ell \cdot m = |Y|$ and so this inequality rearranges to
    \[|Q_u| \leq \frac{2}{m} \left(|Y| -  (1- \delta)n \right) \leq \frac{2\delta n}{m} \leq 4\delta \ell,\]
    where in the last inequality we used $n/2 \leq (1-\delta)n \leq |Y| = \ell m$. By the construction of $R$, the vertex $u$ is complete (in some colour) to each $S_i$ with $i \in [\ell] \setminus Q_u$. Therefore, 
    \[\deg_{R}(u) = (\ell - |Q_u|) \cdot m \geq (1- 4\delta) |Y| \geq (1 - 4\delta)(1- \delta)n \geq (1- 5\delta)n.\]
    But $|X \cup Y| \leq (1 + \delta)n$ and so $n \geq (1 - \delta)|X\cup Y|$. Finally, 
    \[\delta(R) \geq (1 -5\delta)(1-\delta)|X\cup Y| \geq (1 - 6\delta)|X \cup Y|,\] which completes the proof as $V(R) = X \cup Y$. \end{proof}

\subsection{\texorpdfstring{Covering all but $r^{\ca O(r)}$ vertices}{Covering all but r**O(r) vertices}}\label{subsec:covering_almost}

In this section, we apply the machinery of connecting  hubs to show that edge-coloured bipartite graphs $G[A,B]$ (where $A$ is small) contain monochromatic cycles covering all but $r^{\ca O(r)}$ vertices of $A$. We start with the following classical result of P\'osa (see \cite{lovaszcombinatorial}). 

\begin{thm}[P\'osa]\label{thm:posa} The vertices of any graph $G$ can be covered with at most $\alpha(G)$ vertex-disjoint cycles. 
\end{thm}

This leads to the following corollary, essentially already proven in \cite{korandilangletzterpokrovksiy}. 

\begin{corollary}\label{cor:classiccover}
    Let $n, k \geq 1$. Let $G$ be a bipartite graph on vertex classes $A$ and $B$ with $|B| = n$ and $|A|\leq \frac{n}{5k^2}$. If each vertex $v \in A$ satisfies $\deg(v, B) \geq \frac{n}{k}$, then $G$ contains a collection of at most $2k$ vertex-disjoint cycles covering $A$. 
\end{corollary}
\begin{proof}
   First, we apply \cref{fact:small_alpha} to $G[A,B]$ to deduce that the contracted graph $\tilde{G}(A,B, \frac{n}{5k^2})$ has no independent set of size $2k$. Thus, $A= V(\tilde{G}(A, B, \frac{n}{5k^2}))$ can be covered with at most $2k$ vertex-disjoint cycles $C_1, \dots, C_\ell$ by \cref{thm:posa}. By the definition of $\tilde{G}(A, B, \frac{n}{5k^2})$, for any pair of vertices $x, y \in A$ that are neighbours in some $C_j$, we have $\codeg_G(x,y, B) \geq \frac{n}{5k^2} \geq |A|$. So, we can greedily pick a unique coneighbour in $G$ for each of these pairs, thus yielding a collection of vertex-disjoint cycles covering $A$.   \end{proof}

Now we prove the main lemma of this subsection.

\begin{lem}\label{lem:allbutrOr} Let $1/n \ll 1/r, 1/K$ and $1/K \ll 1$. Let $\delta \in [r^{-7r}, 1/24]$. Let $G$ be an $r$-edge-coloured bipartite graph on vertex classes $A$ and $B$ with $|B| = n$ and $|A| \leq \frac{n}{r^{40r}}$. Suppose that $\deg(v, B) \geq (1 - \delta)n$ for each $v \in A$. Then there is a collection of at most $tc_{r+1}(12\delta)$ vertex-disjoint monochromatic cycles covering all but at most $r^{Kr}$ vertices of $A$. 
\end{lem}

\begin{proof}
    Introduce a new constant $\eps > 0$ with
    \[1/K \ll \eps \ll  1. \]
    Also, let $k := \lfloor K/6 \rfloor$, so that $1/k \ll \eps$. Now we apply \cref{lem:linkedfam} to $G[A, B]$ with $k$ playing the role of $K$ and $r, \eps, n$ playing their own roles. Letting $m := \lfloor r^{-14r} n \rfloor$, this yields a set $A' \subseteq A$ with $|A'| \geq |A| - r^{3kr}$ and a  $(r, r^{kr}, m, \eps)$-linked hub family $\ca F = \{\ca H_1, \dots, \ca H_\ell\}$ for $A'$ such that, letting $\ca H_i = (C_i, S_i, \ca T_i)$ and $Y := \bigcup_{i \in [\ell]} S_i$, we have
    \begin{equation}\label{eq:leftoversize}\big|B \setminus Y \big| \leq \frac{n}{r^{7r}}. \end{equation}
    This implies \(\deg(v, Y) \geq (1- \delta)n - \frac{n}{r^{7r}} \geq (1 - 2\delta)n \geq (1-2\delta)|Y|\) for all $v \in A'$ since $\delta \geq r^{-7r}$. Thus, the simplified graph $R := R(A', \ca F)$ satisfies 
    \[\delta(R) \geq (1 - 12\delta)|V(R)|\]
    by \cref{lem:mindeg}. Then, by definition, $R$ can be covered with monochromatic components $D_1, \dots, D_t$ for some $t \leq tc_{r+1}(12\delta)$. For each $i \in [t]$, let $c_i$ be the colour of the component $D_i$ in $R$.  

    For each choice of $1 \leq i < j \leq \ell$ and $c \in [r]$ such that $|L_c(i,j)| > 0$, we have $|L_c(i,j)| \geq r^{kr}$ by \cref{defn:linkedfamily}. So let us reserve a subset $L'_c(i,j) \subseteq L_c(i,j)$ of size precisely $ r^{kr}$. These sets are not necessarily disjoint. Let $A''$ be obtained from $A'$ by removing all elements belonging to each $L_c'(i,j)$. Observe that $\ell \leq r^{15r}$ since the $S_i$ are vertex-disjoint subsets of $B$ of size $m \geq r^{-15r}n$. Then
    \[|A''| \geq |A'| - {\binom{\ell}{2}} r^{kr} \geq |A| - 2r^{3kr}.\]

    For each $i \in [t]$, let $X_i := A'' \cap V(D_i)$. The rest of the proof consists in finding vertex-disjoint monochromatic cycles $\ca C_1, \dots, \ca C_t$ with each $\ca C_i$ covering the vertices of $X_i \setminus \bigcup_{j \in [i-1]} V(\ca C_i)$ (possibly together with some other vertices). Note that this effectively implies that $\ca C_1, \dots, \ca C_t$ jointly cover the entirety of $A''$ (i.e. all but at most $2r^{3kr} \leq r^{Kr}$ vertices of $A$, proving the lemma).
    
    We describe an iterative procedure to construct these cycles. Suppose that we have already constructed vertex-disjoint monochromatic cycles $\ca C_1, \dots, \ca C_{s-1}$ in $G$ for some $s \leq t$ such that the following conditions are satisfied. Letting $U := \bigcup_{i \in [s-1]} V(\ca C_i)$, we have $X_1, \dots, X_{s-1} \subseteq U$ and

    \begin{enumerate}[label = (\roman*)]
        \item\label{prop:intersectL'} $|U \cap L'_{c}(i,j)| \leq (s-1) \cdot r^{kr/4}$ for each $1 \leq i < j\leq \ell$ and $c \in [r]$;
        \item\label{prop:intersectTc} $|U \cap T| \leq (s-1) \cdot r^{kr/4}$ for each $i \in [\ell]$ and $T\in \ca T_i$.
    \end{enumerate}

    We will now show how to construct a monochromatic cycle $\ca C_s$ such that $X_s \setminus U \subseteq V(\ca C_s)$, all the while maintaining properties \ref{prop:intersectL'} and \ref{prop:intersectTc} with $s-1$ replaced by $s$. 

    First observe that, if $|X_s \setminus U| = 1$, then we can simply take the single vertex in $X_s \setminus U$ as a (degenerate) cycle in our final collection of cycles. This clearly preserves \ref{prop:intersectL'} and \ref{prop:intersectTc} and completes this step. So we may assume that $|X_s \setminus U| \geq 2$. Then $D_s$ is a monochromatic component (with colour $c_s$) on at least two vertices, which implies, by the definition of $R =R(X, \ca F)$, that each vertex $v \in X_s$ is incident with a $c_s$-coloured edge, whose other incident vertex must then belong to some $S_{i_v}$. In turn, this implies that $c_s \in C_{i_v}$ and, moreover, $\deg_{c_s}(v, S_{i_v}) \geq \frac{m}{10r}$. Thus, we can partition $X_s \setminus U$ into sets $F_1, \dots, F_\ell$ in such a way that each $v \in F_i$ satisfies $\deg_{c_s}(v, S_i) \geq 
    \frac{m}{10r}$ (note that some $F_i$ could be empty, for instance if $c_s \notin C_i$). 

    Since the set $U$ is obtained from the union of vertex-disjoint cycles in a bipartite graph on vertex classes $A$ and $B$, we have
    \[|U| \leq 2|A| \leq \frac{2n}{r^{40r}} \leq \frac{n}{r^{39r}}.\]

    Now let us consider the graphs $G_{c_s}[F_i, S_i \setminus U]$ for each $i \in [\ell]$ such that $F_i$ is non-empty. Both the $S_i$ and the $F_i$ are pairwise disjoint, and so the graphs $G_{c_s}[F_i, S_i \setminus U]$ share no vertices with each other. Moreover, for each $v \in F_i$ we have
    \[\deg_{c_s}(v, S_i \setminus U) \geq \frac{m}{10r} - \frac{n}{r^{39r}} \geq \frac{m}{20r}.\]
    Then we also have $m/(20r) \leq |S_i \setminus U| \leq  m$, which gives \[|F_i| \leq |A| \leq \frac{n}{r^{39r}} \leq \frac{n}{40000r^{18r}} \leq \frac{|S_i \setminus U|}{2000r^2},\] using the fact that $r \geq 2$ in the third inequality. So, we can apply \cref{cor:classiccover} to $G_{c_s}[F_i, S_i \setminus U]$ with $|S_i \setminus U|$ and $20r$ playing the roles of $n$ and $k$ to find at most $40r$ vertex-disjoint $c_s$-coloured cycles covering $F_i$. After having done this for each $G_{c_s}[F_i, S_i \setminus U]$, we obtain a collection of vertex-disjoint $c_s$-coloured cycles $\ca C'_1, \dots, \ca C'_p$ in $G[X_s \setminus U,Y \setminus U]$ covering $X_s \setminus U$, where
    \begin{equation}\label{eq:boundonp} p \leq 40r \cdot \ell \leq 40r \cdot r^{15r} \leq r^{20r}.\end{equation}

    It remains to `stitch' all the cycles $\ca C'_i$ together into a single cycle $\ca C_s$. For each $i \in [p]$, let $x_i, y_i \in V(\ca C'_i) \cap A$ be the two distinct neighbours of an arbitrarily chosen vertex of $V(\ca C'_i) \cap B$. Note that these exist unless $\ca C'_i$ is a degenerate cycle on at most two vertices. If this is the case and one of these vertices belongs to $A$, we let $x_i$ and $y_i$ be both equal to that vertex; otherwise, we simply ignore $\ca C'_i$ as it serves no purpose. 

    Now we join all the pairs $x_i, y_{i+1}$ (with $i+1$ being taken modulo $p$) using paths $P_i$ that are internally vertex-disjoint from each other and from $U \cup \bigcup_{i \in [p]} V(\ca C'_i)$. To be precise, suppose that we have already found paths $P_1, \dots, P_{p'-1}$ as desired for some $p' \leq p$. Further suppose that $|P_i| \leq \ell \cdot k^6 r^6$ for each $i \in [p'-1]$. We now show how to construct $P_{p'}$ while ensuring that $|P_{p'}| \leq \ell \cdot k^6 r^6$. Let
    \[N := \Big(U \cup \bigcup_{i \in [p]} V(\ca C'_i) \cup \bigcup_{i \in [p'-1]} V(P_i)\Big) \setminus \{x_{p'}, y_{p'+1}\},\]
    i.e. the set of vertices we need to avoid while building $P_{p'}$. First, observe that $N$ is spanned by a vertex-disjoint collection of cycles and at most $p' -1$ paths in $G$  (namely, the $\ca C _i$ for $i \in [s-1]$, the $\ca C'_i$ for $i \in [p]$, and the subpaths $P_i \setminus \{x_i, y_{i+1}\}$ for $i \in [p'-1]$) and so
    \begin{equation}\label{eq:1}|N \cap B| \leq |A| + 2(p'-1) \leq \frac{n}{r^{30r}},\end{equation}
    as $G$ is bipartite on $A \cup B$.
    Second, for each $j \in [\ell]$, each routing set $T \in \ca T_j$ is disjoint from $A''$ and from $Y$. In particular, no $\ca C'_i$ intersects $T$. Therefore,
    \begin{equation}\label{eq:2}
        |N \cap T| \leq |U \cap T| + \sum_{i \in [p'-1]}|V(P_i)| \leq (s-1) r^{kr/4} + p \ell k^6 r^6 \leq r^{kr/3} + r^{kr/3} = 2r^{kr/3}, 
    \end{equation}
    where we used \ref{prop:intersectTc} along with the crude bound $s-1 \leq t \leq tc_{r+1}(12\delta) \leq kr^2$ from the second part of \cref{thm:tree_cover_main} (and also using $p, \ell \leq r^{20r}$). Third, recall that for each choice of distinct $i, j \in [\ell]$ and $c \in [r]$, $L'_{c}(i,j)$ is a subset of $A$ disjoint from $A''$ and thus also from $\bigcup_{i' \in [p]}V(\ca C'_{i'})$, which yields
    \begin{equation}\label{eq:3}
        |N \cap L'_{c}(i,j)| \leq |U \cap L'_{c}(i,j)| + \sum_{i' \in [p'-1]} |V(P_{i'})| \leq (s-1)r^{kr/4} + p\ell k^6r^6 \leq 2r^{kr/3}, 
    \end{equation}
    where this time we used \ref{prop:intersectL'}. With these inequalities in hand, we can now apply \cref{lem:connecting_lemma} to $G[A,B]$ with $\ca F, A', N$ playing the roles of $\ca F, X,U$, with $x_{p'}$ and $y_{p'+1}$ playing the roles of $x$ and $y$, and with $c_s$ in place of $c$. For our numerical parameters, we take $r, \ell, r^{kr}, m, \eps$ to play the role of $r, \ell, t, m, \eps$. Indeed, $\ca F$ is a $(r, r^{kr}, m, \eps)$-linked hub family for $A'$, and $R = R(A', \ca F)$ contains a $c_s$-coloured path from $x_{p'}$ to $y_{p'+1}$ since they belong to the same $c_s$-coloured component $D_s$. Regarding the conditions on $N$, properties \ref{prop:intersecB'} and \ref{prop:intersecTc} immediately follow from \eqref{eq:1} and \eqref{eq:2}, whereas for \ref{prop:intersecL'} we have
    \[|L_c(i,j) \setminus N| \geq |L'_c(i,j) \setminus N| \geq r^{kr} - 2r^{kr/3} \geq r^{15r} \geq \ell,\]
    using \eqref{eq:3}. This verifies the assumptions of \cref{lem:connecting_lemma} and so we find a $c_s$-coloured path $P_{p'}$ in $G[A \setminus N, B \setminus N]$ of length at most $\ell \cdot \log^5 (r^{kr}) \leq \ell \cdot k^6 r^6$. This completes the construction of $P_{p'}$ while maintaining the bound on the length of each $P_i$ required in the next iteration.

    Recall that $y_i$ and $x_i$ have a coneighbour in $\ca C'_i$, and so if $W_i \subseteq \ca C'_i$ is the longest $y_i$--$x_i$ path in $\ca C'_i$, we have $V(\ca C'_i) \cap X_s \subseteq W_i$. After the procedure just described has been performed for each pair $x_i, y_{i+1}$, we obtain the $c_s$-coloured cycle
    \[\ca C_s = x_1 P_1 y_2 W_2 x_2  \dots x_p P_p y_1 W_1x_1,\]
    covering $X_s \setminus U$ and avoiding $U$. Observe that the only times we used vertices of $L_c(i,j)$ (for $1 \leq i < j \leq \ell$ and $c \in [r]$) or $T \in \ca T_i$ (for $i \in [\ell]$) was when constructing all the paths $P_{j}$. This shows that \ref{prop:intersectL'} is preserved, as after updating $U := \bigcup_{i\in [s]} V(\ca C_i)$ the increase in $|U \cap L_c'(i,j)|$ is given by an additive term of at most \[\sum_{j \in [p]}V(P_{j}) \leq p \ell r^6k^6 \leq  r^{kr/4}.\]
Similarly for \ref{prop:intersectTc}, the increase in $|U \cap T|$ for each $i \in [\ell]$ and $T \in \ca T_i$ is also at most $\sum_{j\in [p]} V(P_{j}) \leq r^{kr/4}$. 

    After repeating this procedure for each $s \in [t]$, we obtain a collection of vertex-disjoint monochromatic cycles $\ca C_1, \dots, \ca C_t$ with each $\ca C_i$ covering the elements of $X_i \setminus \bigcup_{j \in [i-1]} V(\ca C_j)$. Thus, $A'' = X_1 \cup \dots \cup X_t \subseteq V(\ca C_1) \cup \dots \cup V(\ca C_t)$. Since $t \leq tc_{r+1}(12\delta)$ and $|A''| \geq |A| - 2r^{3kr} \geq |A| - r^{Kr}$, this finishes the proof. 
\end{proof}

\subsection{Covering a very small leftover}\label{subsec:covering_few}

Our next lemma shows how to cover a leftover of $r^{\ca O(r)}$ vertices with few monochromatic cycles. It is the last ingredient we require for the proof of \cref{lem:absorption}, whose proof is given immediately afterwards and represents the end of \cref{sec:absorption}. 

\begin{lem}\label{lem:bounded_leftover} Let $1/n \ll 1/r, 1/K \leq 1/2$ and let $\delta \in (0, 1/4]$. Let $G$ be an $r$-edge-coloured bipartite graph on vertex classes $A$ and $B$ where $|B| = n$ and $|A| \leq r^{Kr}$. Suppose that $\deg(v, B) \geq (1-\delta)n$ for each $v \in A$. Then there is a collection of at most 
    \[r \left\lceil \frac{3K r \log r}{\log(1/\delta)} \right\rceil\]
    vertex-disjoint monochromatic cycles covering $A$.
\end{lem}

\begin{proof} Let $A := \{v_1, \dots, v_m\}$ for some $m \leq r^{Kr}$. For each $u \in B$, we encode its adjacencies to vertices in $A$ (together with their colours) through a tuple $\sigma^u \in [r+1]^{m}$ by the following rule. For each $i \in [m]$, if $uv_i$ is an edge in $G$, then $\sigma^u_i$ is defined as the unique colour of $uv_i$ in $G$ (which is an element of $[r]$); if there is no edge $uv_i$ in $G$, then we let $\sigma^u_i := r+1$. 

For each $\bbold{x} \in [r+1]^{m}$, define 
\[B_{\bbold{x}} := \{v \in B: \sigma^v = \bbold{x}\}.\]
Now we remove from $B$ all the vertices belonging to the sets $B_{\bbold{x}}$ (for some $\bbold{x} \in [r+1]^{m}$) satisfying $|B_{\bbold{x}}| \leq m$, and call the resulting set $B'$. This set has size
\begin{equation}\label{eq:sizeofB0}|B'| \geq |B| - m \cdot (r+1)^{m} \geq \frac{n}{2},\end{equation}
since $m \leq r^{Kr}$ and $1/n \ll 1/K, 1/r$. Then for each $u\in B'$, we have $|B_{\sigma^u}| \geq m$ and $B_{\sigma^u} \subseteq B'$.

The total number of edges between $A$ and $B'$ is at least
\[|A| (|B'| - \delta n) \geq (1-2\delta)|A||B'|,\]
where we used \eqref{eq:sizeofB0}. By averaging, there is some vertex $u \in B'$ that is adjacent to at least $(1- 2\delta)|A|$ vertices $v \in A$. For each $i \in [r]$, the graph $G[B_{\sigma^u}, N_i(u)]$ is complete and $i$-coloured. Also, $|B_{\sigma^u}| \geq m$ as noted earlier. Taking each colour $i \in [r]$ in order, we can thus easily find an $i$-coloured cycle $C_i$ alternating between vertices of $N_i(u)$ and $B_{\sigma^u}$ and covering $N_i(u)$, all the while avoiding previously constructed cycles $C_1, \dots, C_{i-1}$. After going through all the colours in $[r]$, this yields a collection of vertex-disjoint cycles $C_1, \dots, C_r$ covering $|N(u)| \geq(1-2\delta)|A|$ vertices of $A$.

Now we iterate the argument just presented, thus obtaining a procedure which finds $r$ vertex-disjoint monochromatic cycles at each step and shrinks the leftover by a factor of $2\delta$. Indeed, if $A^* \subseteq A$ and $B^* \subseteq B'$ are the uncovered sets at any given step $i \geq 1$, we still have
\[|B^*| \geq |B'| - m \geq |B| - m \cdot (r+1)^m - m \geq \frac{n}{2}\]
since these sets are obtained by removing a collection of vertex-disjoint cycles from $G[A, B']$, which implies $|B' \setminus B^*| = |A \setminus A^*| \leq m$. Thus, each vertex of $A'$ has degree at least $(1-2\delta)|B^*|$ into $B^*$, and we can find a vertex $u \in B^*$ with $(1-2\delta)|A^*|$ neighbours in $A^*$. Finally, we always have enough vertices in $B^* \cap B_{\sigma^u}$ to build vertex-disjoint cycles since \[|B^* \cap B_{\sigma^u}| \geq |B_{\sigma^u}| - |B' \setminus B^*| = |B_{\sigma^u}| - |A \setminus A^*| \geq m - |A \setminus A^*| = |A^*|.\]

Let $\ell := \left\lceil \frac{3K r \log r}{\log(1/\delta)} \right\rceil$. After $\ell$ steps, we have a leftover $A^* \subseteq A$ of size at most
\[(2\delta)^\ell |A| \leq r^{Kr} \cdot \exp\left\{-\frac{\log(1/(2\delta))}{\log(1/\delta)} \cdot 3K r \log r\right\} \leq r^{Kr} \cdot \exp\left\{-\frac{1}{2} \cdot 3K r \log r\right\} < 1, \]
where in the second inequality we used $\delta \leq 1/4$ which yields $\log(1/(2\delta)) \geq \log(1/\delta)/2$. Therefore, $\ell$ steps are sufficient to cover the entirety of $A$. The total number of cycles we used is at most $r \cdot \ell$, as required by the statement of the lemma. \end{proof}

\begin{proof}[Proof of \cref{lem:absorption}]
    Introduce a new constant $K' \geq 1$ with
    \[ 1/K \ll 1/K' \ll 1.\]
    Let $G[A, B]$ be an $r$-edge-coloured graph with $|B| = n$ and $|A| \leq \frac{n}{r^{40r}}$ as in the statement of the lemma. Apply \cref{lem:allbutrOr} with $K'$ playing the role of $K$ and all other parameters playing their own roles to find a collection of at most $tc_{r+1}(12\delta)$ vertex-disjoint monochromatic cycles covering all vertices of $A$ aside from a subset $A' \subseteq A$ with $|A'| \leq r^{K'r}$. By \cref{lem:generalUB}, we can upper bound the number of cycles used here by
    \begin{equation*}\begin{split}tc_{r+1}(12\delta) \leq 100(r+1) \cdot \left\lceil \frac{r+1}{\log(1/(12\delta))} \log\left(1 + \frac{r+1}{\log(1/(12\delta))}  \right)  \right\rceil  \leq \frac{Kr}{2} \left\lceil \frac{r \log r}{\log(1/\delta)} \right\rceil.
    \end{split}\end{equation*}

    Let $B' \subseteq B$ be obtained by removing the vertices of the cycles found in the previous step. As we only removed cycles in a bipartite graph, we have $|B'| \geq |B| - |A| \geq n/2$ and so
    \[\deg(v, B') \geq |B'| - \delta n \geq (1 - 2\delta)|B'| \]
    for each $v \in A'$. So we can apply \cref{lem:bounded_leftover} to $G[A', B']$ with $|B'|, r, K', 2\delta $ playing the roles of $n, r, K, \delta$. This yields a cover of $A'$ by at most 
    \[r \left\lceil \frac{3K'r\log r}{\log(1/(2\delta))} \right\rceil \leq \frac{Kr}{2} \left\lceil \frac{r \log r}{\log(1/\delta)} \right\rceil\]
    vertex-disjoint monochromatic cycles. Since these cycles are also vertex-disjoint from the ones covering $A \setminus A'$, this completes the proof.\end{proof}

\section{Covering the reduced graph}\label{sec:reduced_graph}

In this section, we will show that, roughly speaking, any $r$-edge-coloured graph with high minimum degree admits a partition into monochromatic matchings contained in the union of a small number of monochromatic components. This result will be applied in our main proof to cover the reduced graph obtained by applying the regularity lemma. These matchings will then be converted to a small number of monochromatic cycles covering almost the entire host graph. 

In order for this almost-cover to combine well with the cycles obtained in the absorption step, we further need to cover a large portion of the reduced graph not just by matchings, but also by certain other subgraphs which carry crucial additional properties. Some of these subgraphs will be triangles, whereas others will be of the following form.

\begin{defn}
    A \define{barbell} is a graph obtained from two vertex-disjoint triangles $a_1a_2a_3$ and $b_1b_2b_3$ by adding a new vertex $c$ together with the edges $ca_1$ and $cb_1$ (see \cref{fig:barbell}). 
\end{defn}

Though triangles and barbells do not contain a perfect matching, they exhibit a perhaps even more useful property: if you blow them up in an approximately balanced way, the resulting graph has a perfect matching. This is encoded in the following definition. 

\begin{defn}[perfect $b$-matching]
    Let $b : V(G) \to \mathbb{Z}^{\geq 0}$ be a function on the vertices of a graph $G$. A \define{perfect $b$-matching} is a non-negative function $\omega: E(G) \to \mathbb{Z}^{\geq 0}$ such that $\sum_{w \in N(v)} \omega(wv) = b$ for each $v \in V(G)$. 
\end{defn}

The next two statements establish the abovementioned fact about perfect matchings in blow-ups of triangles and barbells. 

\begin{fact}\label{fact:trianglerobust}
    Let $n \geq 1$ be an integer. Let $a_1a_2a_3$ be a triangle. Let $b : \{a_1, a_2, a_3\} \to \mathbb{Z}^{\geq 0}$ be a function such that $3n/10 \leq b(a_1) \leq n$, $9n/10 \leq b(a_2), b(a_3) \leq n$, and each $b(a_i)$ is even. Then there is a perfect $b$-matching $\omega(\cdot)$ of $a_1a_2a_3$ such that $\omega(a_ia_{i+1}) \geq n/10$ for each $i \in [3]$. 
\end{fact}
\begin{proof}For each $i \in [3]$, define
    \[\omega(a_i a_{i+1}) := \frac{b(a_i)+ b(a_{i+1}) - b(a_{i+2})}{2} \geq \frac{n}{10},\]
    where indices are taken modulo $3$. Observe that $\omega(a_i a_{i+1})$ is an integer since each $b(v)$ is even. Furthermore, 
    \[\omega(a_ia_{i+1}) + \omega(a_{i+2} a_i) = \frac{b(a_i) + b(a_{i+1}) - b(a_{i+2})}{2} + \frac{b(a_{i+2}) + b(a_i) - b(a_{i+1})}{2} = b(a_i) \]
    for each $i \in [3]$.\end{proof}

\begin{lem}\label{fact:barbellrobust}
Let $n \geq 20$. Let $G$ be a barbell. Given any function $b : V(G) \to \mathbb{Z}^{\geq 0}$ such that $9n/10 \leq b(v) \leq n$ and $b(v)$ is even for each $v \in V(G)$, $G$ has a perfect $b$-matching $\omega(\cdot)$ such that $\omega(e) \geq n/10$ for each $e \in E(G)$.     
\end{lem}

\begin{proof}
Suppose that $G$ consists of triangles $a_1a_2a_3$ and $b_1b_2b_3$, together with another vertex $c$ adjacent to $a_1$ and $b_1$. First, we define
    \[\omega(c a_1) := 2 \left\lceil  \frac{b(c)}{4}\right\rceil, \hspace{1cm} \omega(c b_1) := 2 \left\lfloor \frac{b(c)}{4}\right\rfloor; \]
    thus, $\omega(ca_1) + \omega(cb_1) = b(c)$ since $b(c)$ is even. Also, $\omega(ca_1), \omega(cb_1) \geq n/10$ since $n \geq 20$. 
    
    Define a new weighting $b'$ of $a_1a_2a_3$ as $b'(a_1) := b(a_1) - \omega(ca_1)$ and $b'(a_i) := b(a_i)$ for $i = 2,3$. Observe that $b'(a_i)$ is even for each $i \in [3]$ and, moreover, $b'(a_1) \geq b(a_1) - n/2 - 2\geq 3n/10$ since $n \geq 20$. By \cref{fact:trianglerobust}, $a_1a_2a_3$ has a perfect $b'$-matching $\omega'(\cdot)$ such that $\omega'(a_ia_{i+1}) \geq n/10$ for each $i \in[3]$. We now apply the same argument to obtain a perfect $b''$-matching $\omega'':\{b_1b_2, b_2b_3, b_3b_1\} \to \mathbb{Z}^{\geq 0}$ with $\omega''(b_ib_{i+1}) \geq n/10$, where $b''(b_1) := b(b_1) - \omega(cb_1)$ and $b''(b_i):= b(b_i)$ for $i = 2,3$. 
    
    Finally, we let $\omega(a_ia_{i+1}) := \omega'(a_ia_{i+1})$ and $\omega(b_ib_{i+1}) := \omega''(b_ib_{i+1})$ for each $i \in [3]$. This is a perfect $b$-matching since $\omega(a_1a_2) + \omega(a_1 a_3) + \omega(a_1c) = b''(a_1) + \omega(ca_1) = b(a_1)$ (and the analogous property holds for $b_1$). \end{proof}

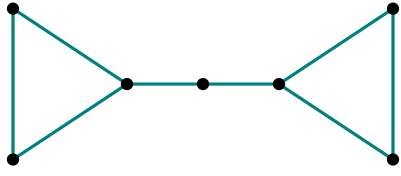
\begin{figure}[t]
    \centering
    \begin{tikzpicture}
        \tkzDefPoint(1,4){v_1}
        \tkzDefPoint(1,2){v_2}
        \tkzDefPoint(2.5,3){v_3}
        \tkzDefPoint(3.5,3){v_4}
        \tkzDefPoint(4.5,3){v_5}
        \tkzDefPoint(6,2){v_6}
        \tkzDefPoint(6,4){v_7}
        
        \tkzDrawPolygon[teal, very thick](v_1,v_2,v_3)
        \tkzDrawSegments[teal, very thick](v_3,v_4 v_4,v_5)
        \tkzDrawPolygon[teal,  very thick](v_5,v_6,v_7)
        \tkzDrawPoints[ultra thick](v_1,v_2,v_3,v_4, v_5, v_6, v_7)
        
    \end{tikzpicture}
    \caption{A barbell.}
    \label{fig:barbell}
\end{figure}

Now we turn to our main lemma for covering the reduced graph, whose proof occupies the remainder of this section. 

\begin{lem}\label{lem:reducedgraphcover}
    Let $1/n \ll 1/r, 1/K \ll 1$ and $\delta \in [e^{-r/10}, 1/K)$. Any $n$-vertex $r$-edge-coloured graph $G$ with $\delta(G) \geq (1-\delta)n$ contains a set $S \subseteq V(G)$ with $|S| \geq \frac{n}{10^{22}\log^3(1/\delta)}$ such that
    \begin{enumerate}[label = (\arabic*)]
        \item $G[S]$ has a spanning collection of vertex-disjoint triangles and barbells contained in the union of at most $6r\log r + 61tc_{r+1}(\delta^{1/2})$ monochromatic components of $G[S]$, and
        \item $G-S$ has a perfect matching contained in the union of at most $11r \log r$ monochromatic components of $G$.
    \end{enumerate}
\end{lem}

For the proof of \cref{lem:reducedgraphcover}, we will proceed in the following order: 

\begin{enumerate}[label = (\roman*)]
    \item construct a large family of vertex-disjoint triangles $T_1, \dots, T_k$ in our graph $G$;
    \item show that $V(G) \setminus \bigcup_{i \in [k]} V( T_i)$ can be approximately covered by matchings; and
    \item\label{step3} using some of the triangles, construct a family of vertex-disjoint barbells covering the leftover of the previous two steps.
\end{enumerate}
The steps will need to be carried out while ensuring that the edges used belong to a small number of monochromatic components, which is the main obstacle. Rather than executing the whole strategy at once, we begin with two separate lemmas: \cref{lem:triangles} shows how to find many suitable triangles, whereas \cref{lem:cherries} is used in \ref{step3} to absorb leftover vertices into barbells. 

We require the following classical result.

\begin{thm}[Corr\'adi-Hajnal \cite{corradihajnal}]\label{thm:corradihajnal}
    Any $n$-vertex graph $G$ with $\delta(G) \geq \frac{2n}{3}$ contains $\lfloor n/3\rfloor$ vertex-disjoint triangles.
\end{thm}

\begin{lem}\label{lem:triangles} Let $1/n \ll \gamma, 1/r$ and $\delta \in [e^{-r/10}, 1/4]$. Let $G$ be an $n$-vertex $r$-edge-coloured graph with $\delta(G) \geq \frac{3n}{4}$. Then there is a set $S \subseteq V(G)$ with $|S| \geq \frac{n }{10^{22}\log^3(1/\delta)}$ satisfying the following property. In $G[S]$, there is a collection of at most $r^2/\log(1/\delta)$ monochromatic components whose union contains at least
    \((1- \gamma) |S|/3\) vertex-disjoint triangles.
\end{lem}

\begin{proof} 

Let us start by iteratively deleting all $c$-coloured edges incident with $v$ for each $c \in [r]$ and $v \in V(G)$ with $\deg_c(v) \leq \frac{n}{10^4r}$. That is, we perform this step over all possible pairs $(v,c)$ until it is no longer possible, and let $G^*$ be the resulting graph. Note that the total number of steps is bounded above by $|V(G) \times [r]| = rn$, with no more than $\frac{n}{10^4r}$ edges being deleted in any given step, and thus the total number of deleted edges is at most $rn \cdot \frac{n}{10^4r} = \frac{n^2}{10^4}$. 

Let $X \subseteq V(G)$ be the set of vertices with $\deg_{G^*}(v) \leq \deg_G(v) - \frac{n}{60}$, and note that $|X| \leq \frac{n}{60}$, as otherwise the total number of deleted edges would be at least
\[\frac{|X|n}{120} \geq \frac{n^2}{7200}, \]
giving a contradiction.

Let $d := 1/(10^7r^2)$. Also, introduce new constants $\eps > 0$ and $M \geq 1$ with updated hierarchy being
\[1/n \ll M \ll \eps \ll \gamma, 1/r.\]
Now we apply the regularity lemma (\cref{lem:regularitylemma}) to $G^*$ with all constants playing their own role, thus obtaining a partition $\{V_0, \dots, V_t\}$ and a subgraph $G'$ of $G^*$ satisfying properties \ref{prop:reg1}--\ref{prop:reg5}. Let $m := |V_1| = \dots |V_t|$, so that $m \geq (1- \eps)n/t$.

For each $i \in [t]$, let $U_i \subseteq V_i$ be a uniformly chosen random subset on $m' := \eps^{1/2} m$ vertices (and independently of other $j \in [t])$. By Chernoff's bound (\cref{lem:chernoff}), with probability at least $1 - rnt e^{-\sqrt{n}} > 0$, we have 
\begin{equation}\label{eq:degree_to_rand}\deg_{G'_c}(v, U_i) \geq \left( \frac{\deg_{G'_c}(v, V_i) }{m} - \eps  \right)m' , \end{equation}
for each $v \in V(G), c \in [r]$ and $i \in [t]$. Let us fix a choice of each $U_i$ so that this property is satisfied. 

Note that $\deg_{G^*_c}(v) > 0$ implies $\deg_{G^*_c}(v) \geq \frac{n}{10^4r}$ for each $v \in V(G)$ and $c \in [r]$. Thus, using property \ref{prop:reg3}, each vertex $v \in V(G)$ with $\deg_{G'_c}(v) > 0$ satisfies
\begin{equation}\label{eq:degc}\deg_{G'_c}(v) \geq \deg_{G^*_c}(v) - (rd+\eps)n \geq  \frac{n}{10^4 r}- (rd + \eps)n \geq \frac{n}{10^5 r}.\end{equation}
 So, all nontrivial monochromatic components of $G'$ have size at least $n/(10^5r)$, thus implying that $G'$ contains no more than $10^5r^2$ nontrivial monochromatic components in total.  

Given not necessarily distinct $i,j \in [t]$, suppose that $x \in V_i$ and $y \in V_j$ are connected by a $c$-coloured path in $G'$. Then $V_i$ and $V_j$ belong to the same $c$-coloured component in $R(G)$. Using \eqref{eq:degc}, it follows by averaging that there exists $i' \in [t]$ with \[\deg_{G'_c}(x, V_{i'}) \geq \frac{m}{10^5r}.\] Again by \eqref{eq:degc}, there exists $j' \in [t] \setminus \{i'\}$ with \[\deg_{G'_c}(y, V_{j'}) \geq t^{-1}(n/(10^5r) - m) \geq \frac{m}{10^6r}.\]
Observe that $R(G)$ contains the $c$-coloured edges $V_i V_{i'}$ and $V_j V_{j'}$. Since $V_i$ and $V_{j}$ are in the same $c$-coloured component in $R(G)$, so are $V_{i'}$ and $V_{j'}$. Therefore, there exists a $c$-coloured path $V_{i'} V_{a_1} \dots V_{a_k} V_{j'}$ in $R(G)$ (which is possible since $i' \neq j'$).

By \cref{lem:basicfacts}, any consecutive pair of sets in $U_{i'}  U_{a_1} \dots U_{a_k} U_{j'}$ forms a $\eps^{1/2}$-regular pair of density at least $d/2$. By \eqref{eq:degree_to_rand}, we have 
\[\deg_{G'_c}(x, U_{i'}), \deg_{G'_c}(y, U_{j'}) \geq  \frac{m'}{10^7 r}.\]
So we can apply \cref{lem:reg_connecting_lemma} to conclude that
\begin{enumerate}[label = (P)]
    \item\label{prop:connectivityinreservoir} for each $c \in [r]$, any vertices $x, y \in V(G')$ that belong to the same $c$-coloured component in  $G'$ are connected by a $c$-coloured path in $G'$ whose internal vertices lie in $\bigcup_{i \in [t]} U_i$. 
\end{enumerate}

By property \ref{prop:reg3} and the fact that $|X| \leq \frac{n}{60}$, we have
\[\delta(G' - X) \geq \frac{3n}{4} - (rd + \eps) n - \frac{n}{60} > \frac{2n}{3}.\]
The Corr\'adi-Hajnal theorem (\cref{thm:corradihajnal}) now tells us that $G' - X$ contains at least \(\lfloor |V(G') \setminus X|/3 \rfloor \geq \frac{n}{6}\) vertex-disjoint triangles. Let us enumerate these triangles arbitrarily and refer to the vertices of the $i$th triangle as $a_i, b_i, c_i \in V(G') \setminus X$, where $i$ takes values in $[q]$ for some $q \geq n/6$. 

Let $C_1, \dots, C_p$ be the nontrivial monochromatic components of $G'$, and recall that $p \leq 10^5r^2$. Now define $\ell := \lceil 10^6\log(1/\delta)\rceil$ and let us partition $\{C_1, \dots, C_p\}$ arbitrarily into disjoint sets $\ca C_1, \dots \ca C_\ell$ with $\lfloor p/ \ell \rfloor \leq |\ca C_1| \leq \dots \leq |\ca C_\ell | \leq \lceil p/\ell \rceil $. In particular, we have
\[|\ca C_i| \leq \lceil p/\ell\rceil \le \left\lceil \frac{r^2}{10 \log(1/\delta)} \right\rceil \leq \frac{r^2}{3\log(1/\delta)}, \]
where in the last inequality we used the fact that $\delta \geq e^{-r/10}$.

Observe that for each $i \in [q]$, the edges of the triangle $a_ib_ic_i$ belong to monochromatic components in $\{ C_1, \dots, C_p\}$. Thus, each triangle $a_ib_ic_i$ is associated with a unique triple $(x,y,z) \in [\ell]^3$ such that the edge $a_ib_i$ is contained in a monochromatic component in $\ca C_x$, the edge $b_ic_i$ is contained in a component in $\ca C_y$, and $a_ic_i$ in a component in $\ca C_z$. By averaging, there is a triple $(x,y,z) \in [\ell]^3$ such that the union of components in $\ca C_x \cup \ca C_y \cup \ca C_z$ contains
\[q' := \lceil q/\ell^3\rceil \geq \frac{n}{10^{22} \log^3(1/\delta)}\]
vertex-disjoint triangles, $T_1, \dots, T_{q'}$. These triangles are contained in the union of $|\ca C_x \cup \ca C_y \cup \ca C_z| \leq \frac{r^2}{\log(1/\delta)}$ monochromatic components of $G'$. Let $S' := \bigcup_{i \in [q']} V(T_i)$ and $S := S' \cup \bigcup_{i \in [t]} U_i$. 

By \ref{prop:connectivityinreservoir}, any two vertices of $S' \subseteq V(G') $ that belong to the same $c$-coloured component in $G'$ for some $c \in [r]$ are connected by a $c$-coloured path in $S$. Thus, $T_1, \dots, T_{q'}$ are contained in the union of at most $\frac{r^2}{\log(1/\delta)}$ monochromatic components in $G[S]$ too. Finally, we have $|S| \geq 3q' \geq \frac{n}{10^{22} \log^3(1/\delta)}$ and $|S \setminus S'| \leq |\bigcup_{i \in [t]} U_i| \leq \eps n \leq \gamma |S|$, which finishes the proof. 
\end{proof}

Now we turn to \cref{lem:cherries} below, which will be used to construct barbells in the proof of \cref{lem:reducedgraphcover}. A \define{cherry} is a $3$-vertex path, and its \define{centre} is its vertex of degree $2$. Our barbells will be constructed by taking two triangles $a_1b_1c_1$ and $a_2b_2c_2$ and finding a cherry containing $a_1$, $a_2$ and a centre $c$ distinct from each $a_i, b_i$ and $c_i$. This justifies the statement of \cref{lem:cherries}, which finds a large collection of vertex-disjoint cherries. Its proof is similar to the proof of Lemma 6 in \cite{GYARFAS2006855}.

We require the following classical theorem. 

\begin{thm}[Erd\H{o}s-Gallai \cite{erdosgallai}]\label{thm:erdosgallai}
    If an $n$-vertex graph $G$ satisfies
    \[e(G) > \frac{(k-2)n}{2},\]
    then it contains a $k$-vertex path.
\end{thm}

\begin{lem}\label{lem:cherries} Let $1/n \ll 1/r, 1/K \ll 1$ and $\delta \in [e^{-r}, 1/K]$.
    Let $G= G[A, B]$ be an $r$-edge-coloured bipartite graph with $|B| = n$, $|A| \leq \max\{\delta n,\frac{n}{12r} \}$, and $\deg(v, B) \geq (1 - \delta)n$ for each $v \in A$. Then $G$ contains a collection of vertex-disjoint cherries covering $A$, with each cherry having its centre in $A$. Moreover, these cherries are contained in the union of at most
    \[6 r \log r + tc_{r+1}(6\delta)\]
    monochromatic components.
\end{lem}

\begin{proof}

Most of the proof consists in proving the following claim, which shows how to cover the majority of $A$. 

\begin{claim}\label{clm:almost}
    All but at most $\frac{\delta n }{r}$ vertices of $A$ can be covered with vertex-disjoint cherries having their centre in $A$ contained in the union of at most $6r\log r$ monochromatic components.
\end{claim}

Before proving the claim, let us show how to use it to complete the proof. Let $A' \subseteq A$ and $B' \subseteq B$ be obtained by removing the vertices of the cherries guaranteed by \cref{clm:almost}. Note that 
\begin{equation}\label{eq:sizeofB'} |B'| \geq n - 2|A| \geq n - 2\max\left\{\delta n, \frac{n}{12r}\right\} \geq \frac{9n}{10}, \end{equation} where we used $\delta \leq 1/K$ and $1/K, 1/r \ll 1$. Then we have
\[\deg(v, B') \geq |B'| - \delta n \geq (1- 2\delta)|B'|\]
for each $v \in A'$. Now we construct an auxiliary $(r+1)$-edge-coloured graph $G'$ on vertex set $A' \cup B'$ as follows:
\begin{enumerate}[label = (\roman*)]
    \item $B'$ induces a complete graph whose edges all take colour $r+1$, and
    \item $ab \in A' \times B'$ is present as an edge in $G'$ and $c$-coloured iff $ab$ is a $c$-coloured edge in $G$ and $a$ is incident with at least $2|A'|$ $c$-coloured edges in $G$. 
\end{enumerate}

Note that, for each $v \in B'$, we have
\[\deg_{G'}(v) \geq |B'| \geq |V(G')| - |A'| \geq |V(G')| - \frac{\delta n}{r} \geq (1 - \delta) |V(G')|,\]
where in the last inequality we used $|V(G')| \geq |B'|$ and \eqref{eq:sizeofB'}. On the other hand, for each $v \in A'$, we have
\[\deg_{G'}(v) \geq (1 - 2\delta)|B'| - 2r|A'| \geq (1- 2\delta)|V(G')| - 3r|A'| \geq (1 - 6\delta)|V(G')|, \]
again using \eqref{eq:sizeofB'} in the last inequality. This shows that $\delta(G') \geq (1- 6\delta)|V(G')|$ and so $G'$ can be covered with $t \leq tc_{r+1}(6\delta)$ monochromatic components $C_1 , \dots, C_t$, whose corresponding colours are $c_1, \dots, c_t$. Let us choose an arbitrary partition $X_1 \cup \dots \cup X_t$ of $A'$ such that $X_i \subseteq V(C_i)$. Observe that, unless $|V(C_i)| = 1$, each vertex in $X_i$ has at least $2|A'|$ incident $c_i$-coloured edges by the construction of $G'$. Then, for each $i \in [t]$ in order, we can pick greedily a collection of vertex-disjoint $c_i$-coloured cherries inside of component $C_i$ covering $X_i$ and avoiding previously constructed cherries. If $|V(C_i)| = 1$ instead, then either $X_i = \emptyset$ and there is nothing left to do, or $X_i = \{x_i\}$ for some $x_i \in A'$: in this case, note that there is some $c \in [r]$ such that $\deg_{G'_c}(x_i, B') \geq (1-6\delta)|V(G')|/r \geq 2|A'|$, and we can greedily pick a cherry covering $x_i$ with unused vertices in colour $c$. 

This procedure yields a collection of vertex-disjoint cherries in $G'$ covering $A'$ contained in the union of at most $tc_{r+1}(6\delta)$ monochromatic components. These cherries cannot take colour $r+1$ since they are contained in $G'[A', B']$, and so they are also contained in $G[A', B']$. Moreover, the components of $G'$ using colours in $[r]$ are subgraphs of the monochromatic components of $G$. Therefore, combining these cherries with the ones guaranteed by \cref{clm:almost}, we obtain a collection of cherries covering $A$ contained in no more than $tc_{r+1}(6\delta) + 6r\log r$ components, as desired. It now only remains to show how to obtain the claim. 

\begin{proof}[Proof of \cref{clm:almost}.]
    The proof splits into two cases depending on the value of $\delta$. 

    \paragraph{Case 1:} $\delta \geq 1/(12r)$. Then our assumptions imply that $|A| \leq \delta n$. We proceed by repeatedly finding and removing a largest collection of vertex-disjoint cherries with centres in $A$ all belonging to the same monochromatic component of $G$, until the set of uncovered vertices in $A$ has size at most $\frac{\delta n}{r}$. This generates nested sequences of sets $A = A_0 \supseteq \dots \supset A_\ell$ and $B = B_0 \supseteq \dots \supseteq B_\ell$, where $A_i$ and $B_i$ are the sets of unused vertices at the end of the $i$th step of the procedure and the $\ell$th step is the last, so that
    \[|A_\ell| \leq \frac{\delta n}{r} \leq |A_{\ell-1}|.\]
    Note that $|B_\ell| \geq |B| - 2|A| \geq (1 - 2\delta) n \geq \frac{n}{2}$ since at each step we remove a collection of cherries.

    We claim that for each $i \in [\ell]$ we have $|A_i| \leq (1 - 1/(5r))|A_{i-1}|$. Indeed, for each $i \in [\ell]$ and $v \in A_{i-1}$,
    \[\deg(v, B_{i-1}) \geq |B_{i-1}| - \delta n \geq (1- 2\delta)|B_{i-1}| \geq \frac{|A_{i-1} \cup B_{i-1}|}{2}, \]
    where we used that $|A_{i-1}| \leq \delta n, |B_{i-1}| \geq n/2$ and $\delta \leq 1/K \ll 1$. Thus, $e(G[A_{i-1}, B_{i-1}]) \geq \frac{|A_{i-1}||A_{i-1} \cup B_{i-1}|}{2}$, which in turn implies that there is some $c \in [r]$ such that $e(G_c[A_{i-1}, B_{i-1}]) \geq \frac{|A_{i-1}||A_{i-1}\cup B_{i-1}|}{2r}$. By the Erd\H{o}s-Gallai theorem (\cref{thm:erdosgallai}), $G[A_{i-1}, B_{i-1}]$ contains a $c$-coloured path on at least $|A_{i-1}|/r$ vertices. This path contains a collection of at least $|A_{i-1}|/(5r)$ vertex-disjoint $c$-coloured cherries with centres in $A$ all belonging to the same monochromatic component. This implies $|A_i| \leq (1 - 1/(5r))|A_{i-1}|$, as desired. 

    Thus, 
    \[\frac{\delta n}{r} \leq |A_{\ell - 1}| \leq \left(1 - \frac{1}{5r} \right)^{\ell -1} |A_0| \leq e^{-(\ell -1)/(5r)} \cdot \delta n,\]
    which gives
    \[\ell \leq 6r \log r.\] Therefore, we have found $6r \log r$ monochromatic components whose union contains cherries covering all but $\delta n/r $ vertices, as desired. 

    \paragraph{Case 2:} $\delta \leq 1/(12r)$. Then we have $|A| \leq \frac{n}{12r}$. In this case, we proceed by iteratively finding and removing a largest collection of vertex-disjoint cherries centred in $A$ contained in the union of $r$ components of distinct colours, until the leftover has size at most $\frac{\delta n }{r}$. Similarly to the previous case, this generates nested sequences of sets of unused vertices $A = A_0 \supseteq \dots \supseteq A_\ell$ and $B = B_0 \supseteq \dots \supseteq B_\ell$, where
    \[|A_\ell| \leq \frac{\delta n }{r} \leq |A_{\ell -1}|,\]
    and $|B_\ell| \geq |B| - 2|A| \geq n/2$. 

    We claim that, for each $i \in [\ell]$,
    \begin{equation}\label{eq:expondecrease}|A_{i}| \leq \frac{6r}{n} \cdot |A_{i-1}|^2.\end{equation}
    To see this, we take $G[A_{i-1}, B_{i-1}]$ and delete all edges $ab \in A_{i-1} \times B_{i-1}$ such that, if $c$ is the colour of $ab$, then $a$ is incident with fewer than $2|A_{i-1}|$ $c$-coloured edges whose other endpoint is in $B_{i-1}$. Letting the modified graph be $G'$, we have
    \[e(G'[A_{i-1}, B_{i-1}]) \geq |A_{i-1}|(|B_{i-1}| - \delta n - 2r|A_{i-1}|) \geq |A_{i-1}|( |B_{i-1}| - 3r |A_{i-1}|) , \]
    where in the last inequality we used the fact that $|A_{i-1}| \geq \frac{\delta n}{r}$. By averaging over vertices in $B_{i-1}$, we find a $u \in B_{i-1}$ with
    \[\deg_{G'}(u, A_{i-1}) \geq |A_{i-1}| - \frac{3r}{|B_{i-1}|} \cdot |A_{i-1}|^2 \geq |A_{i-1}| - \frac{6r}{n} \cdot |A_{i-1}|^2,\]
    where in the last inequality we used $|B_{i-1}| \geq |B_\ell| \geq n/2$. For each $c \in [r]$, define $X_c := N_c(u, A_{i-1})$, and observe that each vertex in $X_c \subseteq A_{i-1}$ has at least $2|A_{i-1}|$ incident $c$-coloured edges whose other endpoint lies in $B_{i-1}$. Thus, for each $c \in [r]$ in order, we can greedily pick vertex-disjoint $c$-coloured cherries centred in $A$ covering all of $X_c$ and avoiding cherries in other colours $c' \in [r]$. All the $c$-coloured cherries thus constructed are mutually connected in colour $c$ since their centres belong to $N_c(u)$. Putting everything together, we obtain a collection of cherries from $r$ components of distinct colours covering all but $\frac{6r}{n} \cdot |A_{i-1}|^2$ vertices, as claimed in \eqref{eq:expondecrease}.

    Finally, unwinding \eqref{eq:expondecrease} we obtain
    \begin{equation*}
        \begin{split}
            \frac{\delta n}{r} \leq |A_{\ell -1}| \leq \left(\frac{6r}{n} \right)^{2^{\ell -1} -1} |A|^{2^{\ell-1}} \leq \left(\frac{6r}{n} \right)^{2^{\ell -1} -1} \cdot \left(\frac{n}{12r} \right)^{2^{\ell-1}} \leq \frac{n}{2^{2^{\ell-1}}}.
        \end{split}
    \end{equation*}
    However, $\delta \geq e^{-r}$ and so we obtain
    \[2^{2^{\ell-1}} \leq re^r,\]
    which forces $\ell \leq 2\log r$ (using the fact that $1/r \ll 1$). Altogether, this procedure finds a collection of vertex-disjoint cherries covering all but $\frac{\delta n}{r}$ vertices of $A$ contained in at most $r \cdot \ell \leq 2r \log r$ monochromatic components. This finishes the proof of claim. \end{proof} 
\end{proof}

Now we are ready to complete the proof of \cref{lem:reducedgraphcover}, thereby finishing this section.  

\begin{proof}[Proof of \cref{lem:reducedgraphcover}] Introduce a new constant $\gamma > 0$ such that
\[1/n \ll \gamma \ll 1/r, 1/K \ll 1.\]

We apply \cref{lem:triangles} to $G$ with $n,\gamma, r$ playing their own roles. Its assumption is met since $\delta(G) \geq (1-\delta)n \geq \frac{3n}{4}$. The lemma application yields a set $S' \subseteq V(G)$ with $|S'| \geq \frac{n}{10^{22} \log^3(1/\delta)}$ such that, in $G[S']$, there is a collection of vertex-disjoint triangles $T_1, \dots, T_{q}$ with \begin{equation}\label{eq:sizeofS'}|S' \setminus \bigcup_{i\in[q]}V(T_i)| \leq \gamma |S'| \leq \gamma n.\end{equation} Moreover, the edges of these triangles are contained in at most $r^2/\log(1/\delta)$ monochromatic components of $G[S']$.

Recall (from \cref{sec:notation}) that a monochromatic connected matching is a matching contained in a monochromatic component. Our next step is to cover a large portion of the vertices of $V':= V(G) \setminus S'$ by iteratively removing monochromatic connected matchings from $G[V']$. 

\begin{claim}\label{clm:greedycover}
    $G[V']$ contains a collection of at most $11r \log r$ vertex-disjoint monochromatic connected matchings covering all but at most $\max\{2\delta n, \frac{n}{r^5}\}$ vertices of $V'$. 
\end{claim}

\begin{proof}
We proceed by iteratively removing the largest monochromatic connected matching from the graph spanned by the set of uncovered vertices in $V'$, until its size satisfies the desired bound. This generates a nested sequence of sets of leftover $V' = V_0 \supseteq V_1 \supseteq \dots \supseteq V_t$, where $|V_t| \leq \max\{2\delta n, \frac{n}{r^5}\} \leq |V_{t-1}|$.

For each $i \in [t]$, we have $|V_{i-1}| \geq 2\delta n$ and so $\delta(G[V_{i-1}]) \geq |V_{i-1}| - \delta n \geq |V_{i-1}|/2$. This implies $e(G[V_{i-1}]) \geq |V_{i-1}|^2/4$ and thus there is a colour $c \in [r]$ such that $e(G_c[V_{i-1}]) \geq |V_{i-1}|^2/(4r)$. By the Erd\H{o}s-Gallai theorem (\cref{thm:erdosgallai}), $G[V_{i-1}]$ contains a $c$-coloured path on at least $|V_{i-1}|/(2r) + 1$ vertices. This path contains a monochromatic connected matching on $|V_{i-1}|/(2r)$ vertices, which implies
\[|V_i| \leq \left(1 - \frac{1}{2r} \right)|V_{i-1}|.\]
Hence, 
\[\frac{n}{r^5} \leq |V_{t-1}| \leq \left(1 - \frac{1}{2r} \right)^{t-1}|V_0| \leq e^{-(t-1)/(2r)}n,\]
which rearranges to 
\[\frac{t-1}{2r} \leq 5 \log r.\]
This yields $t \leq11r\log r$, as desired.
\end{proof}

Let $V''$ be obtained from $V'$ by removing the vertices of the at most $11r \log r$ connected matchings guaranteed by \cref{clm:greedycover}, so that $|V''| \leq \max\{2\delta n, \frac{n}{r^5}\}$. Now define $A := V'' \cup (S' \setminus \bigcup_{i \in [q]} V(T_i)))$ which by \eqref{eq:sizeofS'} satisfies
\[|A| \leq \max\left\{2\delta n, \frac{n}{r^5}\right\} + \gamma n \leq \max\left\{3\delta n ,  \frac{2n}{r^5}\right\}.\]

For each $i \in [q]$, let $t_i \in V(T_i)$ be an arbitrarily chosen vertex, and let $B := \{t_1, \dots, t_q\}$, so that \( |B| = q \geq \frac{n}{10^{23} \log^3(1/\delta)}.\) Observe that
\[3\delta n \leq 10^{24}\delta \log^3(1/\delta) q \leq \frac{\delta^{1/2} q}{6},\]
using the fact that $\delta \leq 1/K \ll 1$. Thus,
\[\deg(v, B) \geq |B| - \delta n \geq  (1 - \delta^{1/2}/6)q.\]
Furthermore,
\[\frac{2n}{r^5} \leq \frac{10^{24} \log^3(1/\delta)q}{r^5} \leq \frac{10^{21} r^3 q}{r^5} \leq \frac{q}{12r},\]
where we used the fact that $1/r \ll 1$ and $\delta \geq e^{-r/10}$. 

The previous three inequalities verify the assumptions of \cref{lem:cherries} applied to $G[A, B]$ with $q, r, K, \delta^{1/2}/6$ playing the roles of $n, r, K, \delta$. The lemma application yields a collection of vertex-disjoint cherries covering $A$, each having its centre in $A$, contained in the union of at most $6r \log r + tc_{r+1}(\delta^{1/2})$
monochromatic components of $G[A, B]$. Let $S := A \cup \bigcup_{i \in [q]} V(T_i)$. These cherries, together with the triangles $T_i$ that they intersect, form a collection of vertex-disjoint barbells covering the entirety of $A$ as well as part of $S \setminus A$. The remaining vertices in $S$ are spanned by all the triangles $T_i$ that were not used by the cherries. Altogether, we obtain a collection of vertex-disjoint triangles and barbells covering $S$ contained in the union of at most
\[\frac{r^2}{\log(1/\delta)} + 6r\log r + tc_{r+1}(\delta^{1/2}) \leq 6r \log r + 61tc_{r+1}(\delta^{1/2}),\]
monochromatic components of $G[S]$, where we used \cref{lem:lower_bound_tree} to upper bound $r^2/\log(1/\delta)$ with a function of $tc_{r+1}(\delta^{1/2})$. 

Finally, observe that \[V(G) \setminus S=  V(G) \setminus (A \cup \bigcup_{i \in [q]} V(T_i))  = V(G) \setminus (V'' \cup S')= V' \setminus V''.\]
It was proven in \cref{clm:greedycover} that $V' \setminus V''$ has a perfect matching contained in the union of at most $11r\log r$ monochromatic components, which finishes the proof.\end{proof}

\section{Proof of main theorem}\label{sec:main_theorem}

In this section, we combine all the results obtained thus far into a proof of the following theorem, our main upper bound on $cp_r(\delta)$. 

\begin{thm}\label{lem:ub_mcpart} Let $1/n \ll 1/r$, let $1/K \ll 1$ and $\delta \in (0, 1/2)$. Any $n$-vertex $r$-edge-coloured graph $G$ with $\delta(G) \geq (1- \delta)n$ admits a partition into at most
    \[K r \log r \left\lceil \frac{r}{\log(1/\delta)} \right\rceil\]
monochromatic cycles. 
\end{thm}

The proof of this theorem essentially follows the outline provided in \cref{sec:overview}. First, we will need the following result.

\begin{thm}[\cite{korandilangletzterpokrovksiy}]\label{thm:monopart}
    For $r \geq 2$, let $n$ be sufficiently large. Then any $r$-edge-coloured graph $G$ on $n$ vertices with $\delta(G) \geq n/2 + 1200r \log n $ admits a partition into at most $10^7 r^2$ monochromatic cycles.
\end{thm}

\begin{proof}[Proof of \cref{lem:ub_mcpart}]
Introduce new constants $r', C \geq 1$ and $\eps > 0$ with updated hierarchy being
\[1/n \ll 1/C \ll \eps \ll 1/r, 1/K \text{ and } 1/K \ll 1/r' \ll1 .\]
First let us remark that we already know the statement to be true if $r \leq r'$ as this implies $10^7r^2 \leq K \leq Kr \log r$, in which case the result follows from \cref{thm:monopart}. The same theorem also gives the statement if $\delta \geq 1/r$, as then we have $10^7 r^2 \leq \frac{Kr^2 \log r}{\log(1/\delta)}$. Finally, without loss of generality we have $\delta \geq e^{-r}$ because decreasing $\delta$ further does not reduce the number of cycles required by the theorem. Summarising, we may and will assume that
\[1/n \ll 1/C \ll \eps \ll  1/r, 1/K \ll 1 \text{ and } \delta \in [e^{-r}, 1/r], \]
where we dropped $r'$ as it is not needed other than to establish $1/r \ll1$. 

Let $d:= e^{-2r}$. Let us now apply the regularity lemma (\cref{lem:regularitylemma}) to $G$ with $n, \eps, r, d$ all playing their own roles and $C$ playing the role of $M$. This yields a partition $\{V_0, \dots, V_t\}$ of $V(G)$ and a subgraph $G'$ of $G$ on vertex set $V(G) \setminus V_0$ satisfying properties \ref{prop:reg1}--\ref{prop:reg5}. Together with this, we have the $(\eps, d)$-reduced multigraph $R(G)$ on vertex set $\{V_1, \dots, V_t\}$. We discard all but one (coloured) edge between each pair of vertices in $R(G)$, thereby obtaining an $r$-edge-coloured graph $R$. By \cref{fact:reduceddeg}, we have
\[\delta(R) \geq (1- \delta - rd - \eps)t \geq (1 - 2\delta) t, \]
where we used $rd, \eps \leq e^{-r}/2 \leq \delta/2$. 

Let $\delta' := \max\{2\delta, e^{-r/10}\}$ and note that $\delta^{1/10} \geq \delta'$ since $\delta \leq 1/r \ll1$ and $\delta \geq e^{-r}$. Now we apply \cref{lem:reducedgraphcover} to $R$ with $t, r, r/2, \delta'
$ playing the role of $n, r, K ,\delta$. Indeed, its assumptions are met since $1/t \leq \eps \ll 1/r$ by \ref{prop:reg1} and $\delta ' \in [e^{-r/10}, 2/r]$. The lemma application yields a set $S \subseteq V(R)$ with
\[|S| \geq \frac{t}{10^{22} \log^3(1/\delta')} \geq \frac{t}{10^{22}\log^{3}(1/\delta)}\]
such that the following conditions are satisfied. 
\begin{enumerate}[label = \textnormal{(S\arabic*)}]
    \item\label{prop:RinducedonS} $R[S]$ has a spanning subgraph $F$ which is a collection of vertex-disjoint triangles and barbells, and, moreover, $F$ is a subgraph of the union of monochromatic components $T_1, \dots, T_p$ of $R[S]$, where $p \leq 6r \log r + 61tc_{r+1}(\delta'^{1/2})$.
    \item\label{prop:RminusS} $R - S$ has a perfect matching $M$ which is a subgraph of the union of monochromatic components $T'_1, \dots, T'_q$ of $R$, where $q \leq 11r \log r$. 
\end{enumerate}
Note that 
\begin{equation}\label{eq:sizeofp}
    p \leq 6r\log r + 61tc_{r+1}(\delta'^{1/2}) \leq 6r \log r + 61tc_{r+1}(\delta^{1/20}) \leq 10^6 r \log r \left\lceil \frac{r}{\log(1/\delta)} \right\rceil,
\end{equation}
where in the last inequality we used \cref{lem:generalUB} to bound $tc_{r+1}(\delta^{1/20})$ from above.

Let $m := |V_1| = \dots =|V_t|$, so that $m \geq (1 - \eps)n/t$ by \ref{prop:reg2}. For each edge $V_i V_j \in E(F \cup M)$, supposing that its colour is $c \in [r]$, let us remove from $V_i$ all vertices $v$ such that $\deg_{G'_c}(v, V_j) \leq dm/2 $ and similarly remove from $V_j$ all vertices $v$ such that $\deg_{G'_c}(v, V_i) \leq dm/2 $. By \cref{lem:basicfacts}, this removes at most $\eps m$ vertices from each of $V_i$ and $V_j$ since $G'_c[V_i, V_j]$ forms an $\eps$-regular pair of density at least $d$. Observe that $\Delta(F \cup M) \leq 3$ and so, letting $V'_i \subseteq V_i$ be the subset obtained after carrying this out for each edge of $F \cup M$, we have $|V'_i| \geq (1 - 3\eps)m$. So, for each $V_i V_j \in E(F \cup M)$ of colour $c \in [r]$, 
\begin{equation}\label{eq:delta}
    \delta(G'_c[V_i', V_j']) \geq \frac{dm}{2} - 3\eps m \geq \frac{dm}{4}. 
\end{equation}

Let $I \subseteq [t]$ satisfy $\{V_i: i \in I\} = S$ and let $W := \bigcup_{i \in I} V'_i$, so that
\begin{equation}\label{eq:sizeofW}|W| \geq |S| \cdot (1- 3\eps)m  \geq \frac{(1- 3\eps)tm}{10^{22}\log^3(1/\delta)} \geq \frac{n}{10^{23}\log^3(1/\delta)}. \end{equation}
Hence, each $v \in V(G)$ satisfies
\[\deg_G(v, W) \geq |W| - \delta n \geq (1 - 10^{23}\delta \log^3(1/\delta))|W| \geq (1- \delta^{1/2})|W|,\]
where in the last inequality we used $\delta \leq 1/r \ll 1$. Now we select uniformly at random a subset $U \subseteq W$ of size exactly $\lceil \eps^{1/2} |W| \rceil$. By Chernoff's inequality (\cref{lem:chernoff}), with probability at least $1 - ne^{-\sqrt{n}}$, we have
\begin{equation}\label{eq:degintoU}
    \deg_G(v, U) \geq \left(\frac{\deg_G(v, W)}{|W|} - \eps \right)|U| \geq (1- 2\delta^{1/2}) |U|
\end{equation}
for each $v \in V(G)$. Also, with probability at least $1 - |I|e^{-\sqrt{n}}$, we have
\begin{equation}\label{eq:intersecVi}
    |U \cap V_i'| \leq 2\frac{|V_i'||U|}{|W|} \leq 4\eps^{1/2} m
\end{equation}
for each $i \in I$. Both the previous inequalities hold for all relevant choices with positive probability, and so let us fix a choice of $U$ satisfying all of these. 

Now we construct a family of short monochromatic cycles which, in a suitable sense, covers all the edges of $F \cup M$. For each $i \in [p]$ and $j \in [q]$, let $c_i \in [r]$ be the colour of the monochromatic component $T_i$ and let $c'_i \in [r]$ be the colour of $T'_i$.  

\begin{claim}\label{clm:smallcycles}
    Let $V' := \bigcup_{i \in [t]} V'_i$. Then $G'[V' \setminus U]$ contains vertex-disjoint monochromatic cycles $C_1, \dots, C_p$ and $C'_1, \dots, C'_q$ such that, for each $i \in [p]$ and $j \in [q]$,
    \begin{enumerate}[label = \textnormal{(C\arabic*)}]
        \item\label{prop:c1} $C_i$ is $c_i$-coloured and $C'_j$ is $c'_j$-coloured,
        \item\label{prop:c2} $|V(C_i)| , |V(C'_j)| \leq t^4$,
        \item\label{prop:c3} for each edge $V_a V_b \in E(T_i \cap F)$, $C_i$ contains an edge of $G'[V'_a, V'_b]$, and
        \item\label{prop:c4} similarly, for each $V_a V_b \in E(T'_j\cap M)$, $C'_j$ contains an edge of $G'[V'_a, V'_b]$.
    \end{enumerate}
\end{claim}

\begin{proof}[Proof of claim.]
    The argument for constructing these cycles is the same for the $C_i$ and $C'_j$ (with essentially the same inequalities), so let us just show how to construct the first. Suppose that we have already constructed $C_1, \dots, C_{\ell-1}$ satisfying the desired properties for some $\ell \leq p$. Now we construct $C_\ell$ using the $c_\ell$-coloured edges of $G'[V' \setminus U]$. 

    Let $X := U \cup  \bigcup_{i \in [\ell-1]} V(C_i)$, so that, for each $i \in [t]$, \[|X \cap V'_i| \leq (\ell-1)t^4  + |U \cap V'_i| \leq pt^4 + |U \cap V'_i| \leq 5\eps^{1/2} n,\]
    where we used \eqref{eq:sizeofp} and \eqref{eq:intersecVi} together with the fact that $1/n \ll 1/C, \eps, 1/r$ and $t \leq C$ by \ref{prop:reg1}. These are the vertices we wish to avoid when constructing $C_\ell$. 

    Define $H := T_\ell \cap F$. Consider an enumeration $e_1, \dots, e_k$ of the edges of $H$ together with an orientation $V_{a_i} V_{b_i}$ of each edge $e_i$ of $H$, with the property that, for each $i \in [k]$, 
    \[a_i \neq b_{i+1},\]
    where indices are modulo $k$. To see that the combination of such an orientation and enumeration exists, note that $F$ is a union of vertex-disjoint barbells and triangles, each of which admits an Eulerian trail. Taken together, these trails yield an orientation of $F$ (and so one of $H$ as well) together with an ordering of the edges of each barbell and triangle. We thus concatenate all these orderings to obtain an enumeration of $E(F)$, and consider the subordering induced on $E(H)$. Now the above property is satisfied with respect to the given enumeration and orientation.
    
    Note that $\delta(G'_{c_\ell}[V'_{a_i}, V'_{b_i}]) \geq dm/4$ by \eqref{eq:delta}. So, using the fact that $k \leq {\binom{t}{2}}$, we can easily pick a $c_\ell$-coloured matching $x_1 y_1, \dots, x_k y_k$ in $G'[V' \setminus X]$ such that $x_iy_i$ is contained in $V'_{a_i} \times V'_{b_i}$. 

    Now it remains to weave together the edges of this matching into a cycle. To this end, we simply apply \cref{lem:reg_connecting_lemma} to connect each pair $y_i x_{i+1}$ (where indices are modulo $k$), simultaneously ensuring that the connecting paths avoid each other. For the application, note that $\deg_{c_\ell}(y_i, V'_{a_i}), \deg_{c_\ell}(x_{i+1}, V'_{b_{i+1}}) \geq dm/4$ and that $T_\ell$ contains a $c_\ell$-coloured $a_i$--$b_{i+1}$ path $V_{d_1}, \dots, V_{d_{\ell'}}$ on at least $2$ clusters (since $a_i \neq b_{i+1}$ by our choice of enumeration and orientation above) and at most $t$. So applying \cref{lem:reg_connecting_lemma} to this path allows us to construct a $c_{\ell}$-coloured $y_i$--$x_{i+1}$ path in $G'$ on at most $t+2$ vertices. Furthermore, the lemma allows us to do this while avoiding: each set $(V_{d_i} \setminus V'_{d_i}) \cup X$, which has size at most $6\eps^{1/2}m $; all previously constructed $y_{i'}$--$ x_{i'+1}$ paths, each of which also only contains at most $t+2$ vertices; and all edges $x_{i'} y_{i'}$, which contribute no more than $2k \leq 2t^2$ vertices. At the end of this step, we obtain a $c_\ell$-coloured cycle $C_\ell$ on at most $(t+2)k \leq t^4$ vertices, as desired. 
    
    The same argument goes through for constructing each $C'_j$, making sure in this step to avoid $U$, all the $C_i$, and all previously constructed $C'_{j'}$. In this case, we take $H := T'_{j} \cap M$, which is a matching and trivially satisfies $a_i \neq b_{i+1}$ for each $i\in [k]$ under any enumeration and orientation of its edges. \end{proof}

    Recall that $I$ is the set of clusters spanned by $F$. Now, for each $i \in [t] \setminus I$, we let $V''_i \subseteq V_i'$ be a set of size exactly $\lceil (1 - 5\eps)m \rceil$ of vertices not contained in any of the cycles guaranteed by \cref{clm:smallcycles}. This is indeed possible since $|V_i'| \geq (1 - 3\eps)m$ and the cycles from the claim altogether only cover at most $(p+q)t^4 \leq \eps m$ extra vertices of $V_i'$. Let us further write $V^*_i$ for the set of vertices of $V_i$ that are not contained in $V_i''$ or in any of the cycles, i.e. \[V^*_i = V_i \setminus \left( V_i''  \cup \bigcup_{j \in [p]} V(C_j) \cup \bigcup_{j \in [q]} V(C'_j) \right),\]
    and note that this set has size at most $5\eps m$. Now we form a set of global leftovers $L$ defined as
    \[L := V_0 \cup \bigcup_{i \in I} V_i \setminus V_i' \cup \bigcup_{i \in [t] \setminus I} V_i^*, \]
    which has size at most $\eps n + 3\eps m t + 5\eps m t \leq 9\eps n$. Observe that $L$ is disjoint from $U$ since $U \subseteq \bigcup_{i \in I} V_i'$. Moreover, 
    \[|L| \leq 9\eps n \leq \eps^{2/3}|W| \leq \eps^{1/6}|U|,\]
    where we used \eqref{eq:sizeofW}, $|U| \geq \eps^{1/2}|W|$, and $\eps \ll 1/r$. We are now in a position to apply our absorption lemma (\cref{lem:absorption}) to $G[L, U]$ with $|U|,r, K/10, 2\delta^{1/2}$ playing the roles of $n, r, K, \delta$. Our previous inequality ensures that $L$ is small enough relative to $U$, whereas \eqref{eq:degintoU} verifies the minimum degree requirement. The lemma yields a collection of vertex-disjoint monochromatic cycles $C^{\dagger}_1, \dots, C^{\dagger}_s$ in $G[L, U]$ covering $L$, where
    \begin{equation}\label{eq:sizeofs}s \leq \frac{Kr}{10}\left\lceil \frac{r \log r}{\log(1/(2\delta)^{1/2})} \right\rceil \leq \frac{Kr^2 \log r}{2\log(1/\delta)}.\end{equation}

    Let $Q$ be the set of vertices contained in the union of all the monochromatic cycles we have constructed so far, namely $\{C_i\}_{i \in [p]}, \{C'_i\}_{i \in [q]}$ and $\{C^\dagger_i\}_{i\in [s]}$, which are all vertex-disjoint. Summarising, we have
    \begin{enumerate}[label = \textnormal{(Q\arabic*)}]
        \item\label{prop:Q1} for each $i \in I$, we have $V_i \setminus Q \subseteq V_{i}'$ and $|V_i \cap Q| \leq 5\eps^{1/2}m$ (see below),
        \item\label{prop:Q2} for each $i \in [t] \setminus I$, $V_i \setminus Q = V_i''$ and $|V_i''| = \lceil (1-5\eps)m \rceil$, and
        \item\label{prop:Q3} $V_0 \subseteq Q$.
    \end{enumerate}
    Here, \ref{prop:Q1} follows from the fact that $V_i \setminus V_i' \subseteq L \subseteq Q$ for each $i \in I$, and \[|V_i \cap Q| \leq |V_i \setminus V_i' | + |V_{i}' \cap U| + \big| \bigcup_{i \in [p]} C_i \big| + \big| \bigcup_{i \in [q]} C_i' \big| \leq (3\eps + 4\eps^{1/2})m + t^5 \leq 4\eps^{1/2}m\] by \eqref{eq:intersecVi} and \ref{prop:c2}. Similarly, \ref{prop:Q2} follows from the fact that $V_i^* \subseteq L \subseteq Q$ and $V_i \setminus (V_i'' \cup V_i^*) \subseteq  \bigcup_{j \in [p]} V(C_j) \cup \bigcup_{j \in [q]} V(C'_j)$ for each $i \in [t] \setminus I$, and \ref{prop:Q3} from the fact that $V_0 \subseteq L \subseteq Q$.

    In the remainder of the proof, our aim will be to edit the cycles constructed so far (in particular, the $C_i$ and $C_i'$) to incorporate all the vertices of $V \setminus Q$, thereby obtaining a monochromatic cycle partition of $V$. All the sets $V_i \setminus Q$ with $i \in [t] \setminus I$ are easy to cover thanks to \ref{prop:Q2} by applying \cref{lem:spanningpath} in combination with the fact that the clusters they are contained in are spanned by the matching $M$ in $\bigcup_{i \in [q]} T'_i$. For the sets $V_i \setminus Q$ with $i \in I$, which, unlike the others, are not guaranteed to be all of the same size, we proceed differently. We use the fact that their clusters are spanned by the graph $F$, which is a collection of triangles and barbells, to find a perfect $b$-matching of $F$ using \cref{fact:trianglerobust} and \cref{fact:barbellrobust}, where $b(i) := |V_i \setminus Q|$ for each $i \in [t]$. This in turn yields an allocation of the vertices of each $V_i \setminus Q$ to the edges of $F$ which makes it easy to complete the cycle partition again using \cref{lem:spanningpath}. 

    For our applications of \cref{fact:trianglerobust} and \cref{fact:barbellrobust} to be successful, we further need all the sets $|V_i \setminus Q|$ for $i\in [t]$ to be even-sized. This is guaranteed, up to a small number of exceptions, by the following claim. 

    \begin{claim}\label{claim:even}
        The monochromatic cycles $\{C_i\}_{i \in [p]}$ can be extended (i.e. by replacing edges with longer paths), only using vertices of $\bigcup_{i \in I} V_i \setminus Q$ and preserving vertex-disjointness relative to each other, so that the following holds. After modifying the $C_i$ and updating $Q$ consequently, properties \ref{prop:c1}, \ref{prop:c3}, \ref{prop:Q2}, \ref{prop:Q3} are preserved, properties \ref{prop:c2} and \ref{prop:Q1} are preserved up to a factor of $2$ in the RHS, and
        \begin{enumerate}[label = \textnormal{(Q4)}]
            \item\label{prop:Q4} $|V_i \setminus Q|$ is even for all but at most $p$ indices $i \in I$.
        \end{enumerate}
    \end{claim}

    \begin{proof}
        We describe a procedure which, in a given step, extends some cycle $C_{i_0}$ and decreases the number of odd-sized sets $V_i \setminus Q$ by $2$, until there are at most $p$ of these. In such a step, $C_{i_0}$ gains at most $2t+2$ vertices, so in total no more than  $(2t+2) |I| \leq t^4$ vertices are added to $Q$, which means that at the end \ref{prop:c2} and \ref{prop:Q1} remain true up to a factor of $2$.

        Let us describe a single step. Suppose that there are at least $p + 1$ indices $i \in I$ such that $V_i \setminus Q$ is odd. Note that each of these clusters is incident with an edge in $F$, and that this edge is contained in a monochromatic component in $\{T_i\}_{i \in [p]}$. By pigeonhole, there is some monochromatic component $T_{i_0}$ in colour $c_{i_0}$ with two distinct clusters $V_a$ and $V_b$ such that $V_a, V_b \in V(T_{i_0})$ and $V_a \setminus Q, V_b \setminus Q$ are both odd.

        For each $i \in I $, we have
    \[|V_i \setminus Q| \geq (1- 5\eps^{1/2})m - t^4 \geq (1- 6\eps^{1/2})m,\]
    where we used \ref{prop:Q1} and the fact that in previous steps $Q$ gained at most $t^4$ vertices. Thus, for each $c_{i_0}$-coloured edge $V_i V_j \in E(F)$ and $v \in V_i \setminus Q$, we have
    \begin{equation}\label{eq:last_eq}
        \deg_{c_{i_0}}(v, V_j \setminus Q) \geq \frac{dm}{8},
    \end{equation} where we used \eqref{eq:delta} and the fact that $V_i \setminus Q \subseteq V_i'$.

    First suppose that $V(T_{i_0}) = \{V_a, V_b\}$. Then $V_a V_b \in E(F)$, which in turn implies by \ref{prop:c3} that $C_{i_0}$ is a $c_{i_0}$-coloured cycle containing an edge $x_a x_b \in V_a \times V_b$. In this case, we apply \cref{lem:reg_connecting_lemma} to find a path $x_a x_b' x_a' x_b \in V_a \times V_b \times V_a \times V_b$ avoiding $Q$ (also using \eqref{eq:last_eq}). We replace the edge $x_ax_b$ with this new path and update $Q$ consequently. This decrease the number of odd-sized $V_i \setminus Q$ by $2$ (i.e. $|V_a \setminus Q|, |V_b \setminus Q|$ are now even) while only using $2$ extra vertices, thereby completing this step. Note that \ref{prop:c1}, \ref{prop:c3}, \ref{prop:Q2} and \ref{prop:Q3} are preserved.

    Now suppose that $V(T_{i_0}) \neq \{V_a, V_b\}$ instead. This implies that there exists $V_c \neq V_a, V_b$ connected to one of $V_a$ or $V_b$ by a $c_{i_0}$-coloured edge. Let us assume $V_a V_c \in E(T_{i_0})$, as the argument is identical in either case. By \ref{prop:c3}, $C_{i_0}$ contains an edge $x_a x_c \in V_a \times V_c$. Since $T_{i_0}$ is a monochromatic component of $R[S]$ by \ref{prop:RinducedonS}, there are indices $d_1, \dots, d_k \in I$ such that $V_a V_{d_1} \dots V_{d_k} V_b$ is a $c_{i_0}$-coloured path in $R[S]$. 

   For each $i \in \{a,b,c,d_1, \dots, d_k\}$, let us partition $V_i$ into two sets $V_{i}^1$ and $V_{i}^2$ of the same size (if $m$ is odd, we first discard an arbitrary vertex and then partition). Consider the sequence of sets 
   \[V^1_{d_1}, \dots , V^1_{d_k} , V^1_b , V_{d_k}^2 , \dots , V_{d_1}^2 , V_a^1,\]
   which are all of the same size, and any two consecutive ones form an $\eps^{1/2}$-regular pair of density at least $d/2$ by \cref{lem:basicfacts}. Moreover, by relabelling if necessary, we may assume (using \eqref{eq:last_eq}) that $\deg_{c_{i_0}}(x_a, V_{d_1}^1), \deg_{c_{i_0}}(x_c, V_a^1) \geq dm/20$ (if $k =0$, we relabel so that $\deg_{c_{i_0}}(x_a, V_b^1) \geq dm/20$ instead). By \cref{lem:reg_connecting_lemma}, there exists a $c_{i_0}$-coloured $x_a$--$x_c$ path avoiding $Q$ whose internal vertices lie in $V^1_{d_1} \times \dots \times V^1_{d_k} \times V^1_b \times V_{d_k}^2 \times \dots \times V_{d_1}^2 \times V_a^1$. We replace the edge $x_a x_c$ in $C_{i_0}$ by this path and update $Q$ consequently. As promised, $C_{i_0}$ gained no more than $2t+2$ vertices and the total number of indices $i \in I$ for which $V_i \setminus Q$ is odd decreased by $2$ (i.e. $|V_a \setminus Q|, |V_b \setminus Q|$ are no longer odd). Properties \ref{prop:c1}, \ref{prop:Q2} and \ref{prop:Q3} are clearly preserved, whereas \ref{prop:c3} remains true since $C_{i_0}$ now uses an edge in $ V_a^1 \times \{x_c\}$ in place of $x_ax_c$.
 \end{proof}

    With the claim in hand, we are now in a position to absorb all the vertices of $V \setminus Q$ into the cycles $\{C_i\}_{i \in [p]}$ and $\{C'_i\}_{i \in [q]}$. First, let us handle the sets $\{V_i \setminus Q\}_{i \in [t] \setminus I}$. Recall that the clusters $V_i$ with $i \in [t] \setminus I$ are spanned by the perfect matching $M$. We consider each edge $V_a V_b$ of $M$ in turn and show how to absorb the vertices of $V_a \setminus Q$ and $V_b \setminus Q$. Recall that $V_a V_b$ is contained in some monochromatic component $T'_f$ of colour $c'_f$. By \ref{prop:c4}, $C'_f$ contains an edge $x^+x^- \in V_a \times V_b$ (note that $x^+$ and $x^-$ must be contained in $Q$). Since $V(C'_f), V_a \setminus Q $ and $V_b \setminus Q$ are all subsets of $V'$ by \cref{clm:smallcycles} and \ref{prop:Q2}, we have
    \[\delta(G'_{c'_f}[(V_a \setminus Q) \cup \{x^+\}, (V_b \setminus Q) \cup \{x^-\}]) \geq \frac{dm}{8},\]
    where we used \eqref{eq:delta} and \ref{prop:Q2}. This bipartite graph forms a $\sqrt{\varepsilon}$-regular pair (by \cref{lem:basicfacts}) whose parts have size precisely $\lceil (1 - 5\eps)m \rceil +1$ by \ref{prop:Q2}. Thus, by \cref{lem:spanningpath}, there is a $c'_f$-coloured $x^+$--$x^-$ path whose internal vertices are precisely $(V_a \cup V_b) \setminus Q$. We now remove the edge $x^+x^-$ from the cycle $C'_f$ and replace it with this path. Repeating this for each edge of $M$ shows how to cover all the clusters $V_i$ with $i \in [t] \setminus I$. 

    Now, let us consider the sets $\{V_i \setminus Q\}_{i \in I}$. By \ref{prop:Q4}, we can make all these sets even-sized by setting aside a set $Y$ consisting of at most $p$ vertices (each of these, viewed as a cycle, will be included in our final partition). For each $i \in I$, we define
    \[b(i) := |V_i \setminus (Q \cup Y)| \geq (1 - 5\eps^{1/2}m )  - 1 \geq \frac{9m}{10}, \]
where we used \ref{prop:Q1}. Recall that the clusters $V_i$ with $i \in I$ are spanned by $F$, which is a collection of vertex-disjoint barbells and triangles. Since $b(i)$ is even for each $i \in I$, we can apply \cref{fact:trianglerobust} and \cref{fact:barbellrobust} to each of these triangles and barbells to construct a perfect $b$-matching $\omega: E(F) \to \mathbb{Z}^{\geq 0}$ with $\omega(e) \geq m/10$ for each $e \in E(F)$.        

Let us now choose uniformly at random a partition of each $V_i \setminus (Q \cup Y)$ of the form
\[\{V_i^e :  e \in E(F), e \ni V_i\},\]
where $|V_i^e| = \omega(e)$ for each $e \in E(F)$ with $V_i \in e$. The partitions for different $i \in I$ are chosen independently of each other. Each partition is well-defined since
\[\sum_{e\ni V_i} |V_i^e| = \sum_{e \ni V_i} \omega(e) = b(i) = |V_i \setminus (Q \cup Y)|.\]
By Chernoff's bound (\cref{lem:chernoff}), for each $v \in V, c \in [r], e \in E(F)$ and $i \in I$, we have
\begin{equation}\label{eq:degafteralloc}\deg_c(v, V_i^e) \geq \frac{\deg_c(v, V_i \setminus |Q \cup Y|)}{|V_i \setminus (Q \cup Y)|} \cdot |V_i^e| - \eps m \end{equation}
with probability at least $1 - rn t^3 e^{-\sqrt{n}} > 0$. So we may fix a choice of partitions satisfying \eqref{eq:degafteralloc} for each of the relevant choices. 

Now, for each $e = V_a V_b \in E(F)$, we can cover the vertices of $V_a^e$ and $V_b^e$ using essentially the same argument that was used for the matching pairs of clusters in $M$. Recall that $e$ belongs to some component $T_f$ of colour $c_f$, and thus, by \ref{prop:c3}, $C_f$ contains an edge $x^+x^- \in V_a \times V_b$. Since $V(C_f), V_a \setminus(Q \cup Y), V_b \setminus (Q \cup Y)$ are all subsets of $V'$, we get
\[\deg_{G'_{c_f}}(v, V_b \setminus (Q \cup Y)), \deg_{G'_{c_f}}(u, V_a \setminus (Q\cup Y)) \geq \frac{dm}{8} \]
for each $v \in \{x^+\} \cup ( V_a \setminus (Q \cup Y)$ and $u \in \{x^-\} \cup (V_b \setminus (Q \cup Y)$. Together with \eqref{eq:degafteralloc}, this yields
\[\delta(G'_{c_f}[V_a^e \cup \{x^+\}, V_b^e \cup \{x^-\}]) \geq \frac{dm}{100},\]
where we used the fact that $|V_a^e| = |V_b^e| = \omega(e) \geq m/10 $. Since this bipartite graph is $\sqrt{\eps}$-regular by \cref{lem:basicfacts}, \cref{lem:spanningpath} yields a $c_f$-coloured $x^+$--$x^-$ path whose internal vertices are precisely those in $V_a^e \cup V_b^e$. Then, replacing the edge $x^+x^-$ in $C_f$ with this path yields a monochromatic cycle covering the entirety of $V_a^e \cup V_b^e$. After repeating this step for each $e \in E(F)$, the resulting collection of monochromatic cycles covers the entirety of $\bigcup_{i \in I}V_i \setminus Q$, thereby completing the cycle partition. 

Our final cycle partition is \(\{C_i\}_{i \in [p]} \cup \{C_i'\}_{i \in [q]} \cup \{C_i^\dagger\}_{i \in [s]} \cup Y,\) which has size at most
\[2p + q + s \leq  Kr \log r \left\lceil \frac{r}{\log(1/\delta)} \right\rceil ,\]
using \eqref{eq:sizeofp}, \ref{prop:RminusS} and \eqref{eq:sizeofs}.  \end{proof}

\begin{proof}[Proof of \cref{thm:main_theorem}] The upper bound follows directly from \cref{lem:ub_mcpart} and the lower bound from the lower bound in \cref{thm:tree_cover_main} together with $tc_{r}(\delta) \leq cp_r(\delta)$.    
\end{proof}

\section{Concluding Remarks}\label{sec:concluding_remarks}

In this paper, we studied how many monochromatic cycles are required to partition a large $r$-edge-coloured graph as a function of its minimum degree, as captured by the function $cp_r(\delta)$. We determined $cp_r(\delta)$ up to a $(\log r)$-factor for all choices of $r$ and $\delta$. Along the way, we also investigated a related but easier problem where, instead of a cycle partition, we seek a cover of the vertex set by (not necessarily vertex-disjoint) monochromatic trees. For this variant, we obtained slightly stronger bounds. Our work leaves open several intriguing questions.

The most natural and significant open problem is to better understand the behaviour of $cp_r(\delta)$. We propose the following conjecture. 

\begin{conjecture}\label{conj:tight_cycle_part}
    There exists $K > 0$ such that, for all $r \geq2$ and $\delta \in (0,1/2)$,
    \[cp_r(\delta) \leq Kr \cdot \left\lceil \frac{r}{\log(1/\delta)}\right\rceil.\]
\end{conjecture}
If true, \cref{conj:tight_cycle_part} would be tight by \cref{lem:lower_bound_tree}. Moreover, it holds for $\delta = \Omega(1)$ by \cref{thm:monopart}, the main result of \cite{korandilangletzterpokrovksiy}. Notably, \cref{conj:tight_cycle_part} predicts that $\ca O(r)$ cycles already suffice for $n$-vertex edge-coloured graphs with minimum degree at least $(1- 2^{-\Omega(r)})n$. This would constitute a substantial strengthening of the corresponding statement for edge-coloured complete graphs, which remains a well-known open problem.

As a stepping stone towards \cref{conj:tight_cycle_part}, we also propose the analogous conjecture for tree covering. 

\begin{conjecture}\label{conj:tree_cover}
    There exists $K > 0$ such that, for all $r \geq 2$ and $\delta \in (0,1)$, 
    \[tc_r(\delta) \leq Kr \cdot \left\lceil\frac{r}{\log(1/\delta)} \right\rceil.\]
\end{conjecture}

This bound would also be tight by \cref{lem:lower_bound_tree}. Furthermore, \cref{thm:tree_cover_main} implies that \cref{conj:tree_cover} holds when $\delta = \Omega(1)$ and, in contrast to \cref{conj:tight_cycle_part}, also when $\delta = 2^{-\Omega(r)}$. The remaining gap thus lies in the intermediate range. Note that, in order to prove \cref{conj:tree_cover}, one may equivalently work with its reformulation in the language of transversals of $\delta$-intersecting multi-$r$-graphs provided by \cref{fact:intersectingaux} and \cref{proposition:frommultitocolour}. 

Pokrovskiy \cite{POKROVSKIY201470} initiated the study of monochromatic cycle partitioning problems in which one allows a leftover set of uncovered vertices whose size is bounded as a function of $r$. In a similar spirit, our methods yield the following result. 

\begin{proposition}\label{prop:bounded_leftover}
    Let $1/n \ll 1/r, 1/K$ and $1/K \ll 1$. Let $\delta \in (0,1/2)$. Every $n$-vertex $r$-edge-coloured graph with $\delta(G) \geq (1-\delta)n$ admits a collection of at most
    \[K(r \log r + tc_{r+1}(\delta^{1/4}))\]
    monochromatic cycles covering all but $r^{Kr}$ vertices of $G$.
\end{proposition}

The proposition follows directly from the proof of \cref{lem:ub_mcpart} if, in the absorption step, one replaces the application of \cref{lem:absorption} with \cref{lem:allbutrOr}. 

Interestingly, if \cref{conj:tree_cover} holds, then \cref{prop:bounded_leftover} is tight up to a constant factor (in the number of cycles used) for all values of $r$ and $\delta$ satisfying $tc_{r}(\delta) = \Omega(r \log r)$ (namely, for $\delta = 2^{-\ca O(r/ \log(r))}$). Indeed, \cref{conj:tree_cover} implies $tc_{r+1}(\delta^{1/4}) = \ca O(tc_r(\delta))$, and hence, for these values of $\delta$, the proposition guarantees that $\ca O(tc_r(\delta))$ cycles suffice to cover all but a bounded number of vertices. On the other hand, the proof of \cref{lem:lower_bound_tree} actually shows that, for some $k >0$, there are arbitrarily large $n$-vertex graphs in which any collection of $k r \lceil r/\log(1/\delta) \rceil$ monochromatic trees leaves not just a single vertex, but $\Omega(n)$ vertices uncovered. Consequently, under these assumptions, \cref{prop:bounded_leftover} is tight.

\paragraph{Acknowledgment.} We would like to thank Luke Collins, Kyriakos Katsamaktsis, and Alex Malekshahian for valuable discussions. We would also like to thank Eoin Long for pointing us to the family of hypergraphs $H_{r,t,m}$ used in the proof of \cref{lem:lower_bound_tree}.

\newcommand{\etalchar}[1]{$^{#1}$}


\begin{thebibliography}{BBG{\etalchar{+}}14}
\providecommand{\url}[1]{\texttt{#1}}
\providecommand{\urlprefix}{\textsc{url:} }
\expandafter\ifx\csname urlstyle\endcsname\relax
  \providecommand{\doi}[1]{doi:\discretionary{}{}{}#1}\else
  \providecommand{\doi}{doi:\discretionary{}{}{}\begingroup \urlstyle{rm}\Url}\fi

\bibitem[ABL{\etalchar{+}}24]{ALLEN2024103838}
\textsc{P.~Allen}, \textsc{J.~B\"ottcher}, \textsc{R.~Lang}, \textsc{J.~Skokan}, and \textsc{M.~Stein} (2024).
\newblock \href{https://doi.org/10.1016/j.ejc.2023.103838}{Partitioning a 2-edge-coloured graph of minimum degree {$2n/3+o(n)$} into three monochromatic cycles}.
\newblock \emph{European J. Combin.} \textbf{121}.

\bibitem[BBG{\etalchar{+}}14]{baloghbarat}
\textsc{J.~Balogh}, \textsc{J.~Bar\'at}, \textsc{D.~Gerbner}, \textsc{A.~Gy\'arf\'as}, and \textsc{G.~N. S\'ark\"ozy} (2014).
\newblock \href{https://doi.org/10.1007/s00493-014-2935-4}{Partitioning 2-edge-colored graphs by monochromatic paths and cycles}.
\newblock \emph{Combinatorica} \textbf{34}(5), 507--526.

\bibitem[BD17]{baldebiasio}
\textsc{D.~Bal} and \textsc{L.~DeBiasio} (2017).
\newblock \href{https://doi.org/10.37236/6089}{Partitioning random graphs into monochromatic components}.
\newblock \emph{Electron. J. Combin.} \textbf{24}(1).

\bibitem[BKS21]{bucickorandisudakov}
\textsc{M.~Buci\'c}, \textsc{D.~Kor\'andi}, and \textsc{B.~Sudakov} (2021).
\newblock \href{https://doi.org/10.1007/s00493-020-4292-9}{Covering graphs by monochromatic trees and {H}elly-type results for hypergraphs}.
\newblock \emph{Combinatorica} \textbf{41}(3), 319--352.

\bibitem[BT10]{BESSY2010176}
\textsc{S.~Bessy} and \textsc{S.~Thomass\'e} (2010).
\newblock \href{https://doi.org/10.1016/j.jctb.2009.07.001}{Partitioning a graph into a cycle and an anticycle, a proof of {L}ehel's conjecture}.
\newblock \emph{J. Combin. Theory Ser. B} \textbf{100}(2), 176--180.

\bibitem[CH63]{corradihajnal}
\textsc{K.~Corradi} and \textsc{A.~Hajnal} (1963).
\newblock \href{https://doi.org/10.1007/BF01895727}{On the maximal number of independent circuits in a graph}.
\newblock \emph{Acta Math. Acad. Sci. Hungar.} \textbf{14}, 423--439.

\bibitem[DN17]{debiasio}
\textsc{L.~DeBiasio} and \textsc{L.~L. Nelsen} (2017).
\newblock \href{https://doi.org/10.1016/j.jctb.2016.08.006}{Monochromatic cycle partitions of graphs with large minimum degree}.
\newblock \emph{J. Combin. Theory Ser. B} \textbf{122}, 634--667.

\bibitem[EG59]{erdosgallai}
\textsc{P.~Erd\H{o}s} and \textsc{T.~Gallai} (1959).
\newblock \href{https://doi.org/10.1007/BF02024498}{On maximal paths and circuits of graphs}.
\newblock \emph{Acta Math. Acad. Sci. Hungar.} \textbf{10}, 337--356.

\bibitem[EoGP91]{EPB}
\textsc{P.~Erd\H~os}, \textsc{A.~Gy\'arf\'as}, and \textsc{L.~Pyber} (1991).
\newblock \href{https://doi.org/10.1016/0095-8956(91)90007-7}{Vertex coverings by monochromatic cycles and trees}.
\newblock \emph{J. Combin. Theory Ser. B} \textbf{51}(1), 90--95.

\bibitem[GK17]{gyarfaspartite}
\textsc{A.~Gy\'arf\'as} and \textsc{Z.~Kir\'aly} (2017).
\newblock \href{https://doi.org/10.1016/j.disc.2017.08.014}{Covering complete partite hypergraphs by monochromatic components}.
\newblock \emph{Discrete Math.} \textbf{340}(12), 2753--2761.

\bibitem[GRSS06]{GYARFAS2006855}
\textsc{A.~Gy\'arf\'as}, \textsc{M.~Ruszink\'o}, \textsc{G.~N. S\'ark\"ozy}, and \textsc{E.~Szemer\'edi} (2006).
\newblock \href{https://doi.org/10.1016/j.jctb.2006.02.007}{An improved bound for the monochromatic cycle partition number}.
\newblock \emph{J. Combin. Theory Ser. B} \textbf{96}(6), 855--873.

\bibitem[Gy{\'a}11]{gyarfassurvey}
\textsc{A.~Gy{\'a}rf\'as} (2011).
\newblock \href{https://doi.org/10.1007/978-0-8176-8092-3\_5}{Large monochromatic components in edge colorings of graphs: a survey}.
\newblock \emph{Ramsey theory}, \emph{Progr. Math.}, vol. 285, 77--96.

\bibitem[HK96]{haxellkohayakawa}
\textsc{P.~E. Haxell} and \textsc{Y.~Kohayakawa} (1996).
\newblock \href{https://doi.org/10.1006/jctb.1996.0065}{Partitioning by monochromatic trees}.
\newblock \emph{J. Combin. Theory Ser. B} \textbf{68}(2), 218--222.

\bibitem[KLLP21]{korandilangletzterpokrovksiy}
\textsc{D.~Kor\'andi}, \textsc{R.~Lang}, \textsc{S.~Letzter}, and \textsc{A.~Pokrovskiy} (2021).
\newblock \href{https://doi.org/10.1016/j.jctb.2020.07.005}{Minimum degree conditions for monochromatic cycle partitioning}.
\newblock \emph{J. Combin. Theory Ser. B} \textbf{146}, 96--123.

\bibitem[KS96a]{komlossimonovits}
\textsc{J.~Koml\'os} and \textsc{M.~Simonovits} (1996).
\newblock Szemer\'edi's regularity lemma and its applications in graph theory.
\newblock \emph{Combinatorics, {P}aul {E}rd\H os is eighty, {V}ol.\ 2 ({K}eszthely, 1993)}, \emph{Bolyai Soc. Math. Stud.}, vol.~2, 295--352.

\bibitem[KS96b]{ks2}
\textsc{J.~Koml\'os} and \textsc{E.~Szemer\'edi} (1996).
\newblock \href{https://doi.org/10.1017/S096354830000184X}{Topological cliques in graphs. {II}}.
\newblock \emph{Combin. Probab. Comput.} \textbf{5}(1), 79--90.

\bibitem[KS97]{ks1}
\textsc{J.~Koml\'os} and \textsc{E.~Szemer\'edi} (1997).
\newblock Topological cliques in graphs.
\newblock \emph{Combinatorics, geometry and probability ({C}ambridge, 1993)}, 439--448.

\bibitem[KSS97]{blowuplemma}
\textsc{J.~Koml\'os}, \textsc{G.~N. S\'ark\"ozy}, and \textsc{E.~Szemer\'edi} (1997).
\newblock \href{https://doi.org/10.1007/BF01196135}{Blow-up lemma}.
\newblock \emph{Combinatorica} \textbf{17}(1), 109--123.

\bibitem[Let19]{shohammono}
\textsc{S.~Letzter} (2019).
\newblock \href{https://doi.org/10.37236/7239}{Monochromatic cycle partitions of 2-coloured graphs with minimum degree {$3n/4$}}.
\newblock \emph{Electron. J. Combin.} \textbf{26}(1).

\bibitem[Let24]{shohamsublinear}
\textsc{S.~Letzter} (2024).
\newblock Sublinear expanders and their applications.
\newblock \emph{Surveys in combinatorics 2024}, \emph{London Math. Soc. Lecture Note Ser.}, vol. 493, 89--130.

\bibitem[Lov07]{lovaszcombinatorial}
\textsc{L.~Lov\'asz} (2007).
\newblock \href{https://doi.org/10.1090/chel/361}{Combinatorial problems and exercises}.
\newblock second edn. (AMS Chelsea Publishing, Providence, RI).

\bibitem[{\L}RS98]{luczak}
\textsc{T.~{\L}uczak}, \textsc{V.~R\"odl}, and \textsc{E.~Szemer\'edi} (1998).
\newblock \href{https://doi.org/10.1017/S0963548398003599}{Partitioning two-coloured complete graphs into two monochromatic cycles}.
\newblock \emph{Combin. Probab. Comput.} \textbf{7}(4), 423--436.

\bibitem[{\L}uc99]{LUCZAK1999174}
\textsc{T.~{\L}uczak} (1999).
\newblock \href{https://doi.org/10.1006/jctb.1998.1874}{{$R(C_n,C_n,C_n)\leq(4+o(1))n$}}.
\newblock \emph{J. Combin. Theory Ser. B} \textbf{75}(2), 174--187.

\bibitem[Pok14]{POKROVSKIY201470}
\textsc{A.~Pokrovskiy} (2014).
\newblock \href{https://doi.org/10.1016/j.jctb.2014.01.003}{Partitioning edge-coloured complete graphs into monochromatic cycles and paths}.
\newblock \emph{J. Combin. Theory Ser. B} \textbf{106}, 70--97.

\bibitem[Pok23]{pokrovskiysparse}
\textsc{A.~Pokrovskiy} (2023).
\newblock \href{https://doi.org/10.1016/j.disc.2022.113161}{Partitioning a graph into a cycle and a sparse graph}.
\newblock \emph{Discrete Math.} \textbf{346}(1).

\bibitem[Sch12]{Schelp2012SomeRT}
\textsc{R.~H. Schelp} (2012).
\newblock \href{https://doi.org/10.1016/j.disc.2011.09.015}{Some {R}amsey-{T}ur\'an type problems and related questions}.
\newblock \emph{Discrete Math.} \textbf{312}(14), 2158--2161.

\bibitem[Sze78]{regularpartitions}
\textsc{E.~Szemer\'edi} (1978).
\newblock Regular partitions of graphs.
\newblock \emph{Probl\`emes combinatoires et th\'eorie des graphes ({C}olloq. {I}nternat. {CNRS}, {U}niv. {O}rsay, {O}rsay, 1976)}, \emph{Colloq. Internat. CNRS}, vol. 260, 399--401.

\end{thebibliography}
\end{document}